\documentclass[reqno,12pt]{amsart}
\usepackage{amssymb,amsmath}
\usepackage{graphicx}
\usepackage{color}
\usepackage{pgf,ifpdf,graphics,xcolor}

\newtheorem{theorem}{Theorem}
\newtheorem{proposition}[theorem]{Proposition}
\newtheorem{lemma}[theorem]{Lemma}
\newtheorem{definition}[theorem]{Definition}
\newtheorem{corollary}[theorem]{Corollary}
\newtheorem{remark}[theorem]{Remark}
\newtheorem{example}[theorem]{Example}

\newcommand{\C}{{\mathbb C}}
\newcommand{\R}{{\mathbb R}}
\newcommand{\Z}{{\mathbb Z}}
\newcommand{\N}{{\mathbb N}}
\newcommand{\Q}{{\mathbb Q}}
\newcommand{\CA}{{\mathcal A}}
\newcommand{\CB}{{\mathcal B}}

\newcommand{\CF}{{\mathcal F}}
\newcommand{\CG}{{\mathcal G}}
\newcommand{\CH}{{\mathcal H}}

\newcommand{\CM}{{\mathcal M}}
\newcommand{\CO}{{\mathcal O}}

\newcommand{\CT}{{\mathcal T}}
\newcommand{\CV}{{\mathcal V}}

\newcommand{\la}{{\langle}}
\newcommand{\ra}{{\rangle}}
\newcommand{\aff}{\operatorname{aff}}
\newcommand{\flip}{\operatorname{Flip}}
\newcommand{\suppresschi}{{}}

\newcommand{\geom}{\operatorname{Geom}}

\newcommand{\lin}{\operatorname{lin}}

\newcommand{\varchenko}{{\mathcal X}}
\newcommand{\ANA}{{A}}
\newcommand{\discretepol}{{S}}
\newcommand{\intpol}{{I}}
\renewcommand{\ll}{{\langle}}

\newcommand{\rr}{{\rangle}}

\newcommand{\e}{{\mathrm{e}}}
\renewcommand{\a}{{\mathfrak{a}}}
\renewcommand{\b}{{\mathfrak{b}}}
\renewcommand{\c}{{\mathfrak{c}}}

\newcommand{\f}{{\mathfrak{f}}}
\newcommand{\g}{{\mathfrak{g}}}
\renewcommand{\t}{{\mathfrak{t}}}
\newcommand{\p}{{\mathfrak{p}}}
\newcommand{\q}{{\mathfrak{q}}}
\newcommand{\lattice}{\Lambda}

\title[Parametric polytopes]{Analytic continuation of a parametric polytope and wall-crossing}

\author{N. Berline}
\address{Nicole Berline: Ecole Polytechnique, Centre de Mathematiques Laurent Schwartz,  91128 Palaiseau Cedex, France}
\email{nicole.berline@math.polytechnique.fr}
\author{M. Vergne}
\address{Mich\`ele Vergne: Universite Paris 7 Denis Diderot, Institut Math\'ematique de
Jussieu, 175 rue du Chevaleret - 75013 Paris, France  }
\email{vergne@math.jussieu.fr}

\begin{document}
\begin{abstract}
We define a set theoretic ``analytic continuation" of a  polytope defined by  inequalities. For the regular values of the parameter, our
construction coincides with the parallel transport of polytopes in a
mirage  introduced by Varchenko. We determine
the set-theoretic variation when crossing a wall in the parameter
space, and we relate this  variation to Paradan's wall-crossing
formulas for integrals and discrete sums. As another application, we
refine the theorem of Brion on generating functions of polytopes and
their cones at vertices.
We describe the relation of this work with the equivariant index of a line bundle  over a toric variety and Morelli constructible support function.
\end{abstract}
\maketitle
\date{}
\tableofcontents

\section*{Introduction}
Consider a polytope   $\q(b)$ in $\R^d$ defined by a system of $N$ linear inequalities:
\begin{equation}\label{eq:mirage}
\q(b):=\{y\in \R^d;\langle\mu_i,y\rangle\leq b_i, \;\; 1\leq i\leq N.\}
\end{equation}
In this article, we study  the variation of the polytope $\q(b)$
when the parameter $b=(b_i)$ varies in $\R^N$, but  the linear forms $\mu_i$ are fixed (the parametric arrangement of
hyperplanes $\langle \mu_i,y\rangle=b_i$ so obtained is called a mirage in \cite{varchenko}).

Our main construction  is the following.
Starting with a parameter $b^0$ which is regular (this is defined below),
 we construct a  function  $\varchenko(x_1,x_2,\ldots,x_N)$ on $\R^N$
which  is a linear combination of characteristic functions of  various semi-open coordinate quadrants
in $\R^N$. Define

$$
A(b)(y)=\varchenko(b_1-\langle\mu_1,y\rangle,\dots, b_N-\langle\mu_N,y\rangle).
$$

The crucial feature of the  function $\varchenko$ is that, for $b$ near $b^0$, $A(b)(y)$ is the characteristic function of the polytope  $\q(b)$, but $\ANA(b)$  enjoys analyticity properties with respect to the parameter $b$  when $b$ moves  in $\R^N$,  that we will explain below.
So we say that $\ANA(b)$ is the ``analytical continuation "of the polytope $\q(b)$  (with initial value $b^0$).

Before stating these properties, let us give two examples. We denote by $p_i$ the characteristic function
of the closed coordinate half-space,  $p_i=[x_i\geq 0]$, and we set $q_i=1-p_i=[x_i<0]$.
First,  let $\q $ be the $d$-dimensional simplex defined by the $d+1$ inequalities $y_i\geq 0, \sum_{i=1}^d y_i\leq 1$.
In this case we have, (see Example \ref{standard_simplex}),
$$
\varchenko(x)=p_1\cdots p_{d+1}+(-1)^d q_1\cdots q_{d+1}.
$$
Thus $\varchenko(x)$ is the sum of the [characteristic function of the]  closed positive coordinate quadrant in $\R^{d+1}$ and of $(-1)^d$ times the open negative one.
Let $b=(b_1,\ldots,b_{d+1})$.
If $b_1+\cdots +b_{d+1}\geq 0$, then $\ANA(b)(y)=\varchenko(b_1+y_1,\cdots, b_d+y_d, b_{d+1}-(y_1+\cdots + y_d))$
is the characteristic function of the simplex $\{y_i\geq -b_i, \sum_{i=1}^d y_i\leq b_{d+1}\}$,
while if $b_1+\cdots +b_{d+1}< 0$,
then $\ANA(b)(y)$  is equal to
$(-1)^d$ times the characteristic function of the symmetric open simplex $\{y_i< -b_i, \sum_{i=1}^d y_i> b_{d+1}\}$.
In particular, in dimension $d=1$, starting with the closed interval $[0,1]$,
the analytic continuation $\ANA(b)$ is the  closed interval
$\{-b_1\leq y\leq b_2\}$
when $b_1+b_2\geq 0$,
while $\ANA(b)$ is $(-1)$ times the open interval $\{b_2< y< -b_1\}$ when $b_1+b_2< 0$ (Fig.\ref{intervalfirst})
\begin{figure}[!h]
\begin{center}
  \includegraphics[width=2 in]{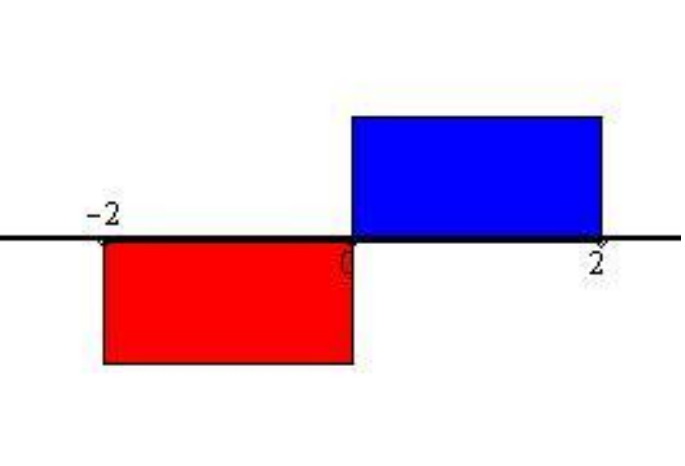}
\caption{ In blue for $b=(0,2)$, $\q(b)= [0,2]$ ,  in red for $b=(0, -2)$,
 $\ANA(b)= (-1)$ times $]-2,0[$  }
 \label{intervalfirst}
\end{center}
\end{figure}

For the second example, we start with the tetragon  illustrated in Fig.\ref{chamber24_intro}  defined by the $4$ inequalities $y_2+2\geq 0$, $y_1+1\geq 0$, $y_1+y_2\leq 0$, $y_1-y_2\geq 0$.
In this case we have (see Example \ref{ex:zonotopeB2_suite} and Subsection \ref{ex:tetragon})
$$
\varchenko(x)=
p_{{1}}p_{{2}}p_{{3}}p_{{4}}-p_{{1}}q_{{2}}q_{{3}}p_{{4}}-q_{{1}}p_{{2}}p_{{3
}}q_{{4}}+q_{{1}}q_{{2}}q_{{3}}q_{{4}},
$$
a signed sum of characteristic functions of $4$ semi-open quadrants.

Some values of the analytic continuation $\ANA(b)$ are illustrated in
Figs. \ref{chamber24_intro} and \ref{tope41_intro}.
For each value of $b$, it is a signed sum of semi-open polygons.
Components with a $+$ sign are colored in blue, components with a $-1$ sign are colored in red.
Semi-openness is indicated by dotted lines.
 \begin{figure}[!h]
\begin{center}
  \includegraphics[width=1.4 in]{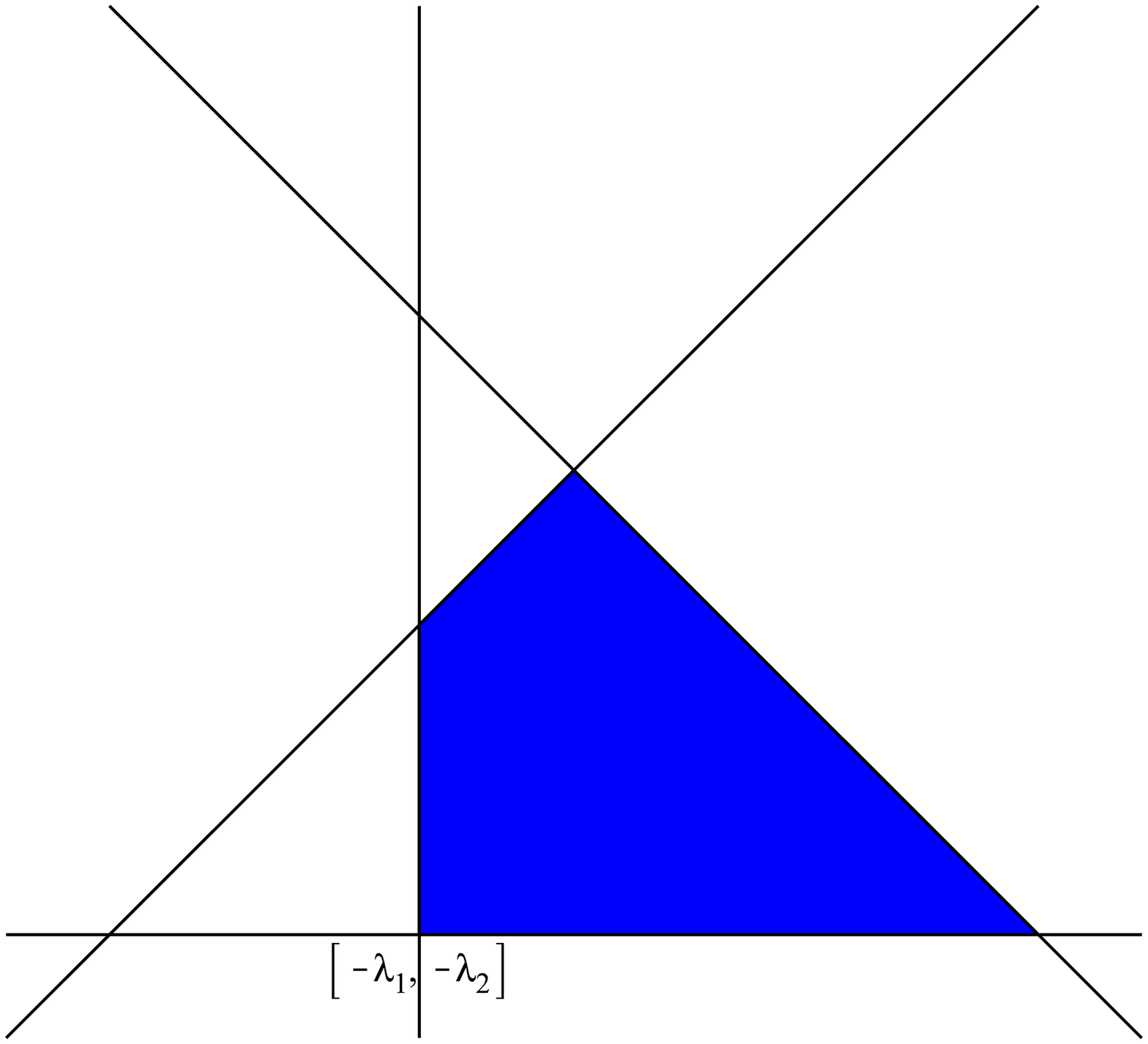}  \includegraphics[width=1.4 in]{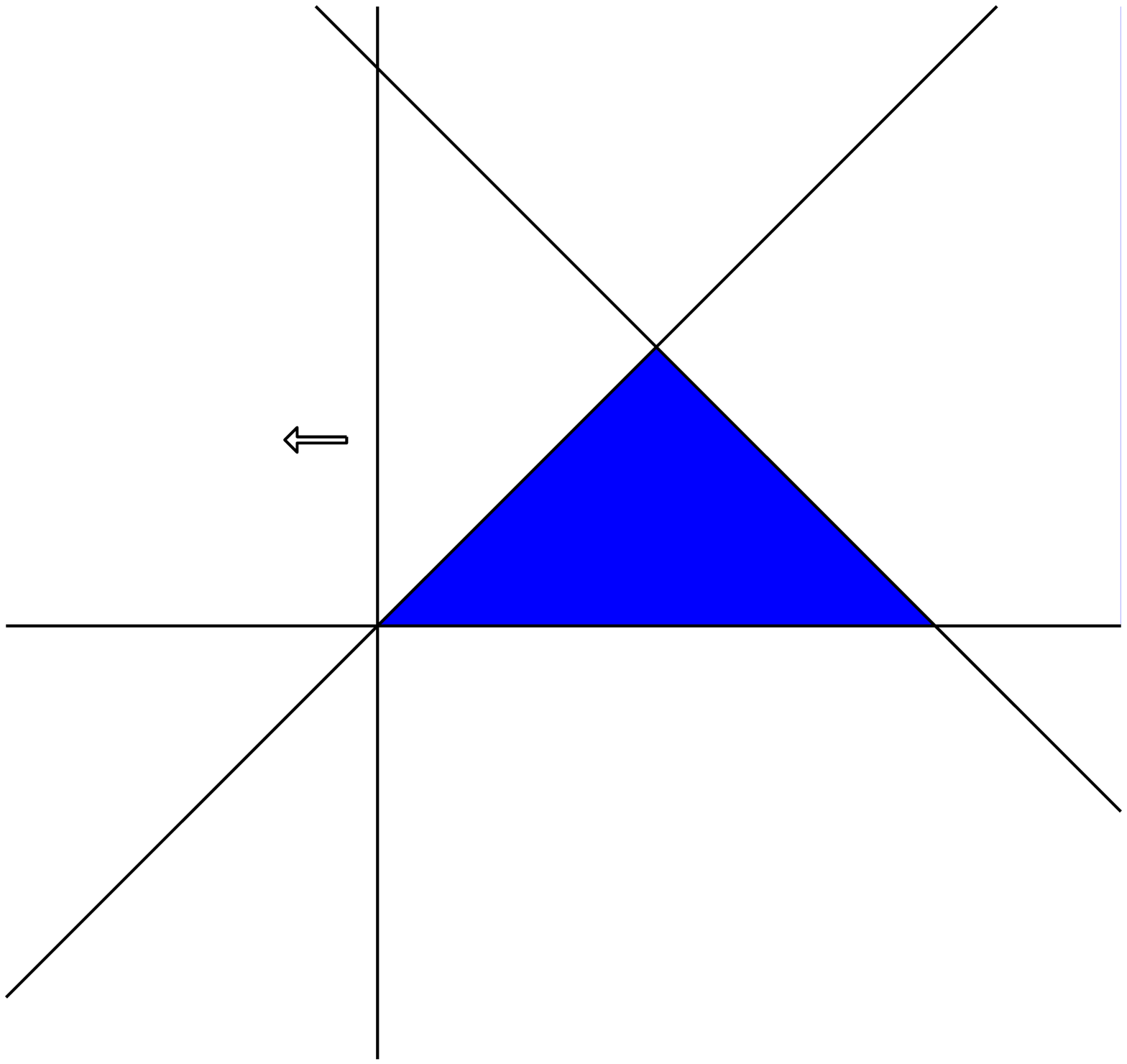}
  \includegraphics[width=1.4 in]{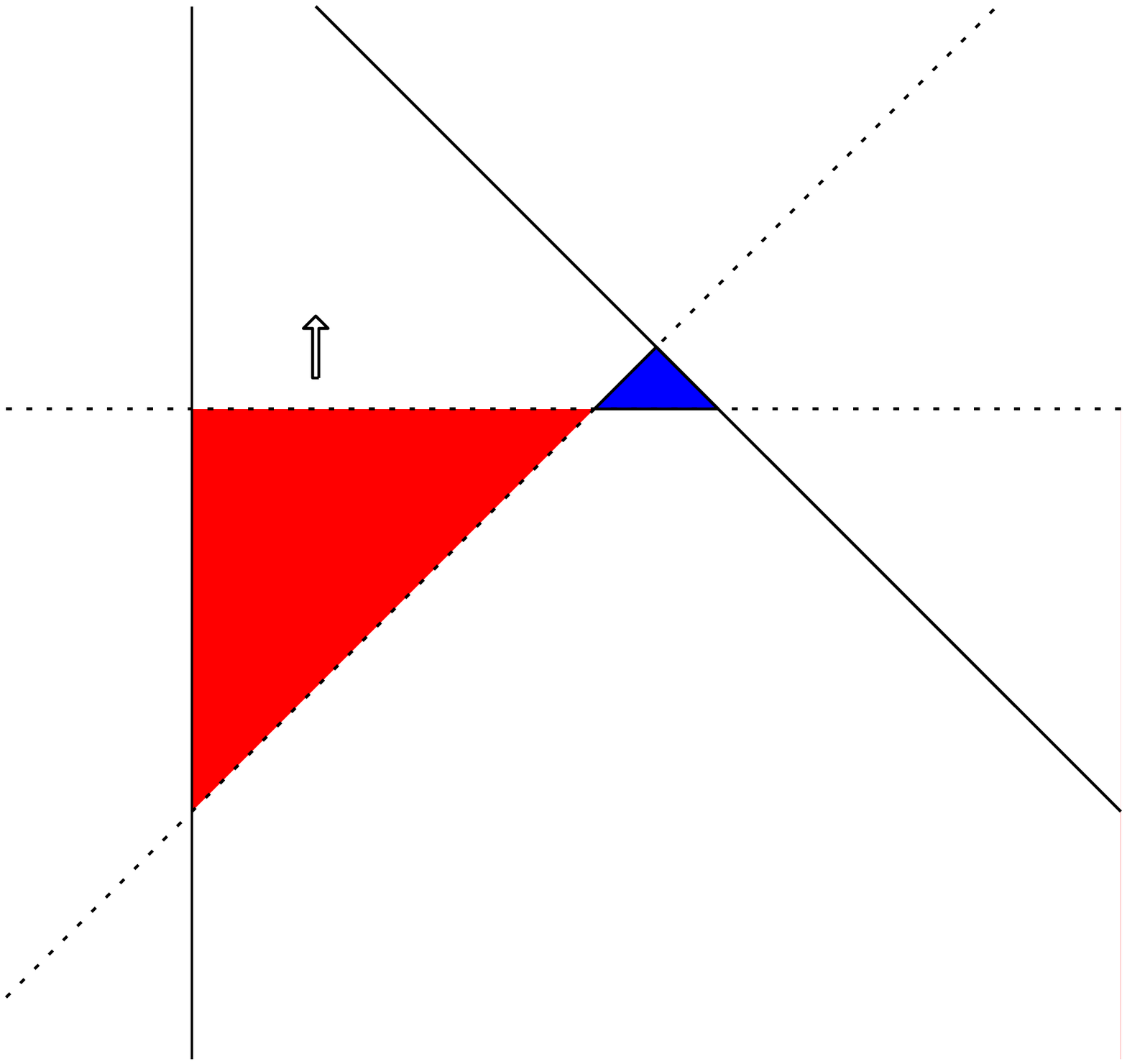}
   \caption{Analytic continuation of a tetragon}
   \label{chamber24_intro}
  \end{center}
\end{figure}
\begin{figure}[!h]
\begin{center}
  \includegraphics[width=1.4 in]{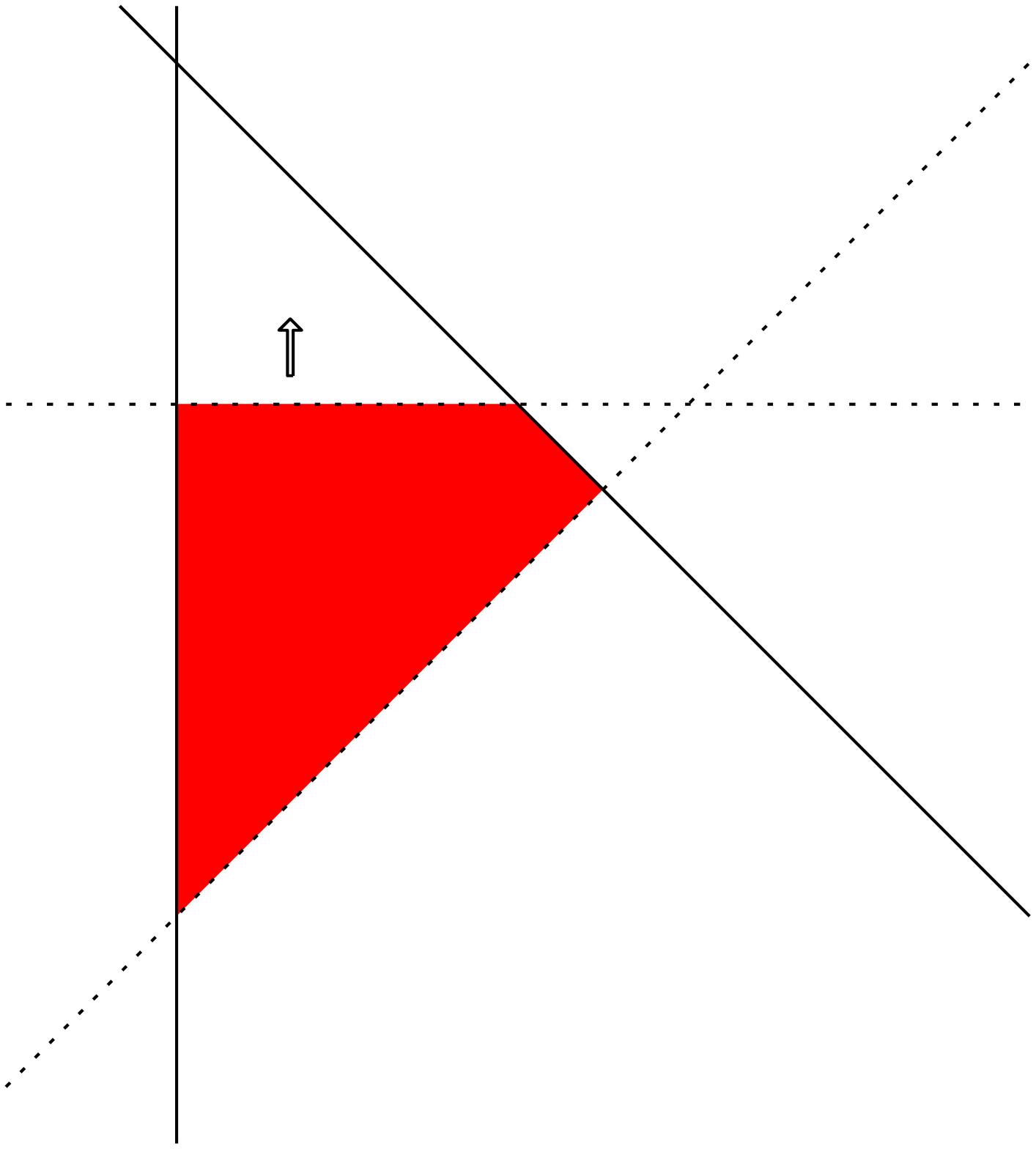}
  \includegraphics[width=1.4 in]{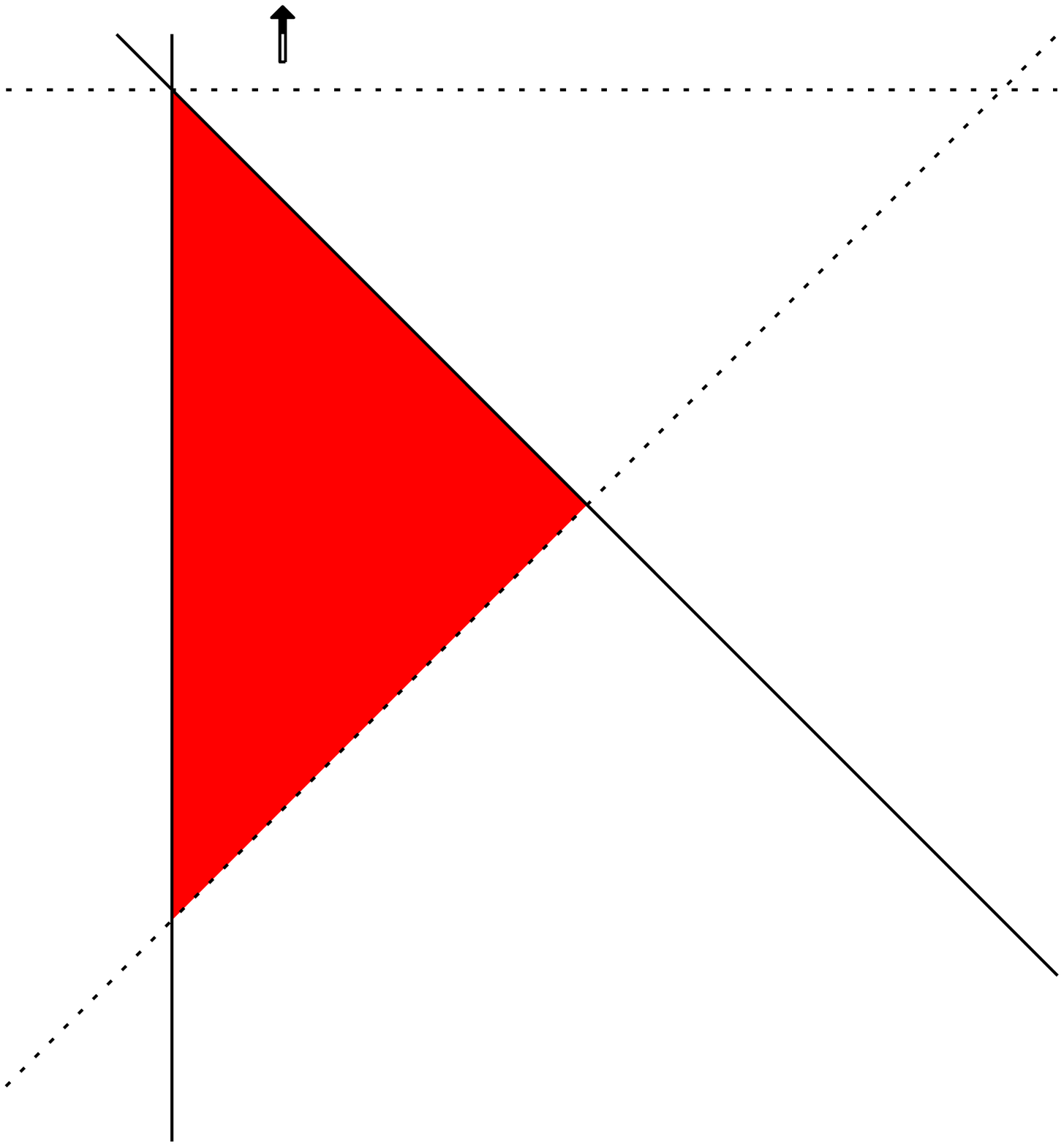}
 \includegraphics[width=1.4 in]{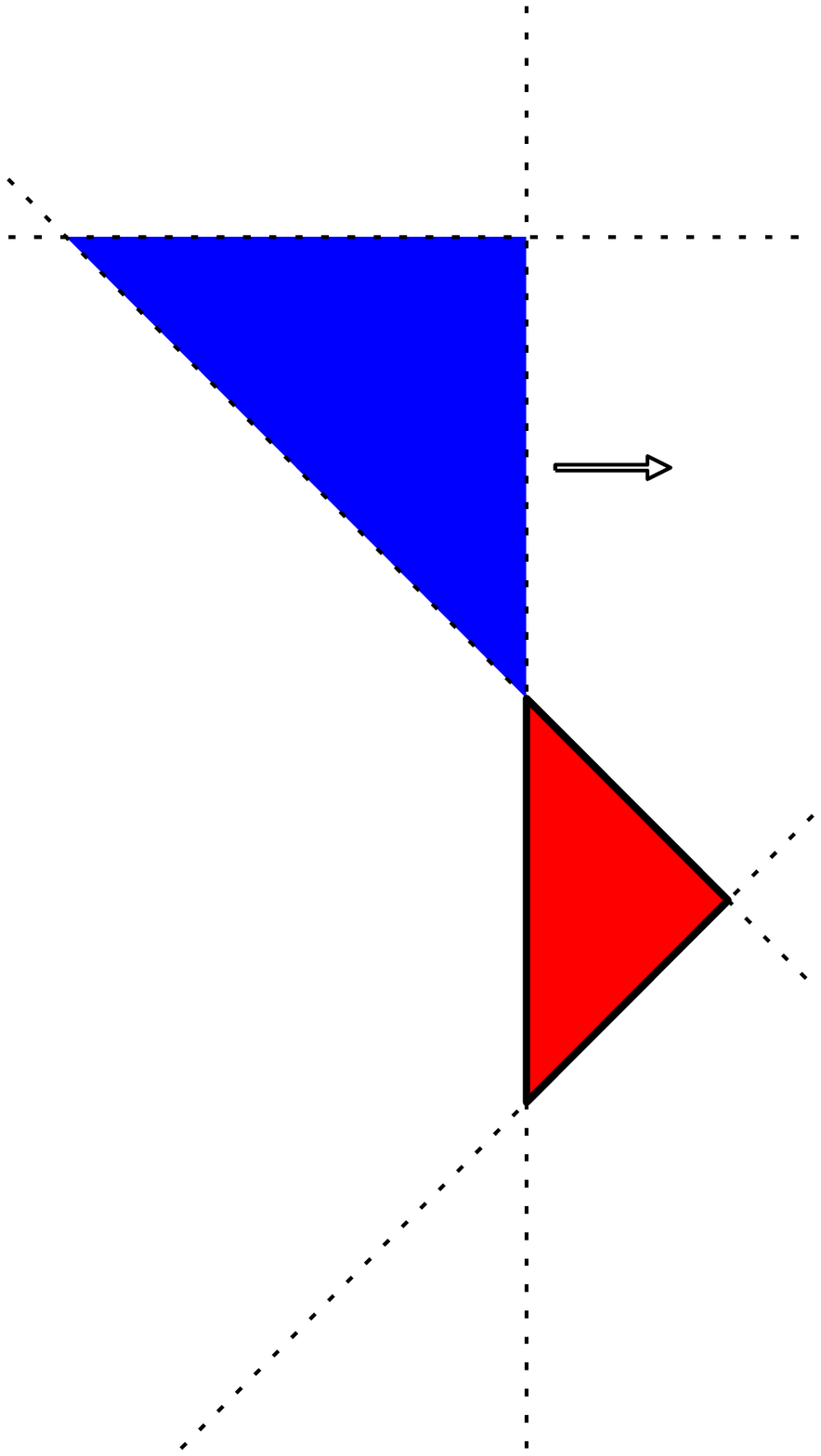}
  \caption{More analytic continuation of a tetragon}
  \label{tope41_intro}
\end{center}
\end{figure}

Let us describe  now  some of the properties  of $\ANA(b)$.

A point $b=(b_i)\in \R^N$  is called regular (with respect to the sequence  of linear forms $\mu_i$)
if  a subset of $k$  equations among     the equations $\{\mu_i=b_i\}$  do not have a common solution if $k>d$.
We define a tope $\tau$ to be  a connected component of the open set of  regular points $b$ in $\R^N$.
Topes are separated by hyperplanes which we call walls.

Let $b^0\in \R^N$ be regular. Recall that we assume that  $\q(b^0)$ is compact.
In this case, each vertex of the polytope  $\q(b^0)$ belongs to exactly $d$ facets,
in other words the polytope $\q(b^0)$ is simple.
Loosely speaking, the shape of the polytope $\q(b)$  does not change
when $b$ remains close to  $b^0$. The facets of $\q(b)$
remain parallel to those of $\q(b^0)$, while its vertices depend linearly  on $b$. When $b$ crosses a wall, the
shape of $\q(b)$ changes.

Let $h(y)$ be  a polynomial function  on $\R^d$. The  integral
$$
\int_{\q(b)}h(y)dy,
$$
and  the  discrete sum
$$
\sum_{y\in \q(b)\cap \Z^d}h(y)
$$
are classical topics.
In particular, if $h$ is the constant function $1$,
 these quantities are respectively the volume of the polytope $\q(b)$ and the number of integral points in the polytope $\q(b)$.
 It is well-known that
the function  $ b\to \int_{\q(b)}h(y)dy $ is given on each tope  by a  polynomial function of $b$.
Moreover, if we assume that the
linear forms  $\mu_i$ are rational, the discrete sum $b\to  \sum_{y\in
\q(b)\cap \Z^d}h(y) $ is given on each tope by a quasi-polynomial function of $b$.
These results follow for instance from
Brion's theorem of decomposing a polytope as a sum of its tangent cones at
vertices, \cite{Brion88}, \cite{brion-vergne-97-residue}.
When the parameter $b$
crosses a wall of the tope $\tau$, the integral  $ b\to \int_{\q(b)}h(y)dy $ is  given by a different
polynomial, the discrete sum by a different quasi-polynomial.
Their
wall-crossing variations have been computed by Paradan, in a more
general context of Hamiltonian geometry, using transversally
elliptic operators, \cite{paradan2004}.

The function $\varchenko$  which we construct in this article depends  on the tope $\tau$ which contains the starting value $b^0$, and we will study its dependance with respect to $\tau$.
 Therefore, we  write  $\varchenko(\tau)(x)$ and
 $\ANA(\tau,b)(y)=\varchenko(\tau)(b_1-\langle\mu_1,y\rangle,\ldots, b_N-\langle\mu_N,y\rangle)$ instead of $\varchenko(x)$ and $\ANA(b)(y)$ from now on. The function
$y\mapsto \ANA(\tau, b)(y)$ enjoys the following properties.

\noindent $\bullet$ When $b$ is in the closure $\overline \tau$  of the tope
$\tau$,  $\ANA(\tau, b)$ coincides with
 the characteristic function $\suppresschi{[\q(b)]}$ of $\q(b)$.

\noindent $\bullet$  The function $ \ANA(\tau, b)(y)$ is a linear
combination with integral coefficients of characteristic functions of bounded
faces of various dimensions of the arrangement of hyperplanes
$\langle\mu_i,y\rangle = b_i, \;\; 1\leq i\leq N$.

 \noindent $\bullet$ The
integral
$$
\int_{\R^d}\ANA(\tau, b)(y)e^{\langle\xi,y\rangle}dy
$$
is an analytic function of $(\xi,b)\in (\R^d)^*\times \R^N$.
For  $b\in \overline \tau$, it coincides with
$\int_{\q(b)} e^{\langle\xi,y\rangle}dy$. If $h(y)$ is a polynomial  function,
then $b\mapsto \int_{\R^d}\ANA(\tau, b)(y)h(y) dy
$ is a polynomial function of $b\in \R^N$ which coincides with $
\int_{\q(b)}h(y)dy $ when $b\in \overline \tau$ .

\noindent $\bullet$ Moreover, if we assume  that the $\mu_i$ are
rational, the discrete sum
$$
\sum_{y\in \Z^d}\ANA(\tau, b)(y)h(y)
$$
is a quasi-polynomial function of $b$, (see Definition  \ref{quasi} of quasi-polynomial functions).
 It coincides with $\sum_{y\in
\q(b)\cap \Z^d}h(y)$ for $b$ in the initial tope and even in a
neighborhood of its closure (see the precise statement in Corollary
\ref{th:continuity-on-closed-tope} ).

For instance, let us look again at the closed interval $[0,b]$.
For $b\in \N$,  the number of integral points in $[0,b]$ is given by the polynomial function $b+1$.
For a negative integer $b<0$, the value $b+1$
is  indeed equal to $(-1)$ times the number of integral points in the open interval $b<y<0$ .

The key idea is to define $\ANA(\tau, b)$ as a signed sum
of closed affine cones,  shifted when $b$ varies, so that
their vertices depend linearly on the parameter $b$. We use
 decompositions of a polytope $\p$ as a
signed sum of cones, such as the Brianchon-Gram decomposition, (see
for instance \cite{brion-vergne-97-lattice}).
\begin{theorem}[Brianchon-Gram decomposition]\label{th:brianchon-gram}
Let $\p\subset \R^d$ be a polytope. For each  face $\f$ of $\p$, let
$\t_{\rm aff}(\p,\f)\subseteq \R^d $ be the affine tangent cone to $\p$ at
the face $\f$. Then
\begin{equation}\label{eq:brianchon-gram}
\suppresschi{[\p]}=\sum_{\f\in \CF(\p)}(-1)^{\dim
\f}\suppresschi{[\t_{\rm aff}(\p,\f)]},
\end{equation}
where  $ \CF(\p)$ is the set of  faces of $\p$.
\end{theorem}
Here, for a set $E\subset \R^d$, we denote by  $\suppresschi{[E]}$
the function on $\R^d$ which is the characteristic function of the set $E$.

For regular values of $b$, our construction of
$\ANA(\tau, b)$ coincides with  the parallel transport of
Varchenko \cite{varchenko}, the idea of which is
 quite simple. For instance,
write the Brianchon-Gram formula  for the closed interval $0\leq y\leq b$,
$$
[0\leq y\leq b] =[ y\leq b]+[y\geq 0] -[\R].
$$
If the vertex $b$  moves  to the left,   crosses the origin and becomes negative,
the right hand side of the Brianchon-Gram formula  becomes first,
 for $b=0$, the characteristic function of the point $0$, then for $b<0$,
the characteristic function of the open interval $b<y<0$  with a minus sign.

 Actually,  instead of the Brianchon-Gram
decomposition, Varchenko uses the
 polarized decomposition into semi-closed cones at
vertices which he obtains  in \cite{varchenko}.
However, we go beyond \cite{varchenko} in several
ways. First, as we already mentioned, we introduce (and compute) the ``precursor" function $\varchenko(\tau)$, a sum of characteristic functions of semi-open quadrants,
which gives rise to
 $\ANA(\tau, b)$ for all $b$. Moreover,
 we compute explicitly the wall-crossing variation
$$
\suppresschi{[\q(b)]}-\ANA(\tau, b)
$$
when $b$ belongs to a tope adjacent to the starting tope
$\tau$.
Actually, we compute the wall crossing variation at the level of the ``precursor" function $\varchenko(\tau)$ itself.

Finally, we show that ``analytic continuation"  of  the faces of the polytope $\q(b^0)$
occurs naturally, when one wants to compute $\sum_{y\in \q(b_0)\cap \Z^d}e^{\ll \xi,y\rr }$
for a   degenerate value of $\xi$.

\bigskip

Let us  now summarize the results of this article. We need some  notations.
It is   more convenient  to work  in the framework
of partition polytopes. So, let us first recall how one goes from the framework of
linear inequalities $\ll \mu_i,y\rr \leq b_i$ to  that of  partition polytopes.
A partition polytope $\p(\Phi, \lambda)$ is determined by a sequence
$\Phi=(\phi_j)_{1\leq j\leq N} $ of elements of a vector space $F$ (of dimension $r$)
and an element $\lambda\in F$, as follows:
\begin{definition}\label{def:partitionpolytope}
$$\p(\Phi, \lambda)=\{x\in \R^N ; \sum_{j=1}^N x_j\phi_j=\lambda,
\;\; x_j\geq 0.\}$$
\end{definition}

We  assume that the cone $\c(\Phi)$ generated by the $\phi_j$'s, is
salient and that $\Phi$ generates $F$. Thus the set $\p(\Phi, \lambda)$
is compact whenever $\lambda\in \c(\Phi)$ (if $\lambda$ is not in $\c(\Phi)$, then  $\p(\Phi, \lambda)$ is empty.)
The polytope $\p(\Phi,\lambda)$ is, by definition,
the intersection of  the affine subspace $$V(\Phi,\lambda)= \{x\in
\R^N ; \sum_{j=1}^N x_j\phi_j=\lambda\}$$ with the standard quadrant
$$
Q:=\{x\in \R^N;x_j\geq 0\}.
$$
A wall in $F$ is a hyperplane generated by $r-1$ linearly independent elements of $\Phi$.
An element $\lambda\in F$ is called $\Phi$-regular, if $\lambda$ does not lie on any wall.
If $\lambda\in \c(\Phi)$ is regular,  the polytope $\p(\Phi,\lambda)$
is a simple polytope of dimension $d=N-r$ contained in the affine space $V(\Phi,\lambda)$.

Consider the map $M:\R^N\to F$  given by $M(x)=\sum_i x_i\phi_i$.
Let $V\subset \R^N$ be the kernel of $M$.
 $$
 V= \{x\in
\R^N ; \sum_{j=1}^N x_j\phi_j=0\}.
$$
So $V$ has dimension  $d=N-r$.

If $E$ is a subset of $\R^N$, we denote now by $[E]$ the function on $\R^N$ which is the characteristic function of $E$.
Thus if $E$ is a subset of $V$, its characteristic function in $V$ is identified with $[E]\cap [V]$.

If $\lambda=M(b)=\sum_i b_i\phi_i$, the map
\begin{equation} \label{eq:y_mapsto_x}
x\to x+b
 \end{equation}
 is an isomorphism between $V$ and the affine space $V(\Phi,\lambda)$ .

 Let $\mu_i$ be the linear form $-x_i$ restricted to $V$.
 The bijection $V\to V(\Phi,\lambda)$    maps the polytope $\q(b)=\{y\in V; \; \langle\mu_i,y\rangle\leq b_i\}$
 onto $\p(\Phi,\lambda)$.
 Indeed, the point $(y_1+b_1,\ldots, y_N+b_N)$ is in $\p(\Phi,\lambda)$ if and only if $-y_i\leq b_i$.

Moreover,  $b$ is regular with respect to the sequence of linear forms
$\mu_i$ on $V$ if and only if $\lambda =M(b)$ is $\Phi$-regular in $F$.
A connected component  of the set of $\Phi$-regular elements of  $F$ will be called a $\Phi$-tope.
Thus a subset $\tau\subset F$ is a $\Phi$-tope if and only if
 $M^{-1}(\tau)\subset \R^N$ is a connected component
of the set of regular parameters, i.e. a tope with respect to  $(\mu_i)$.

It is clearly  equivalent to study the variation of the polytope $\q(b)$ when $b$ varies,
or the variation of  the partition polytope $\p(\Phi,\lambda)$,
when $\lambda$ varies.
In this framework,  the inequations $x_j\geq 0$ are fixed, while the affine space $V(\Phi,\lambda)$ varies.
For example,
Fig. \ref{intervalagain}  shows  the interval $[0,b]$, in blue, now realized as
 $\{x_1\geq 0, x_2\geq 0, x_1+x_2=b\}$.
The analytic continuation $\ANA(\tau,b)$ for $b<0$ is colored in red on this figure,
where a minus sign is assigned to red.
\begin{figure}[!h]
\begin{center}
 \includegraphics[width=2 in]{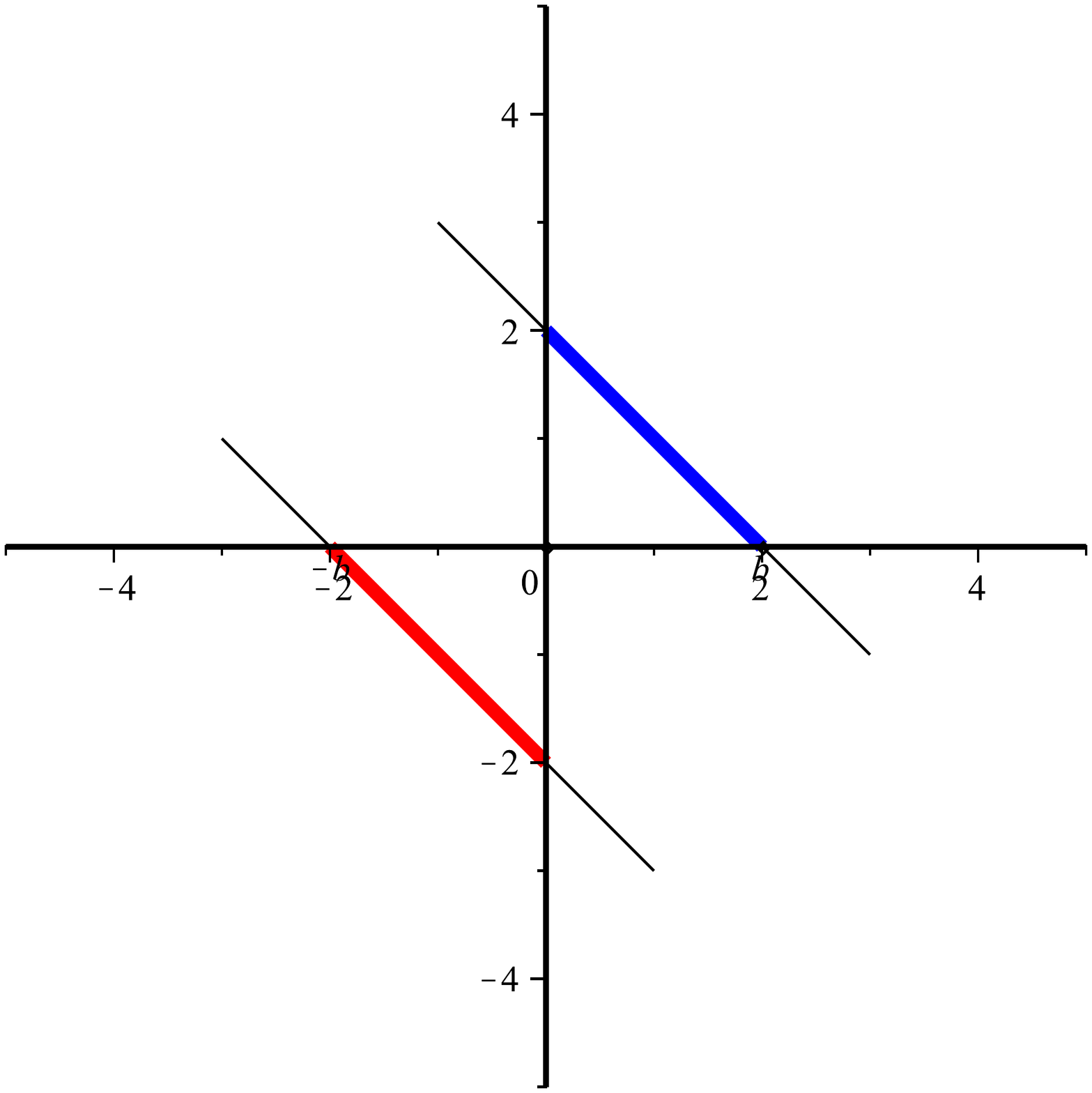}
\caption{$F=\R, \Phi=(1,1)$.}
\label{intervalagain}
\end{center}
\end{figure}

We fix a $\Phi$-tope $\tau\subset F$, and consider $\lambda\in \tau$.
Recall the combinatorial description of the faces of the partition polytope
$\p(\Phi,\lambda)$.
We denote by $\CG(\Phi,\tau)$ (resp.
$\CB(\Phi,\tau)$) the set of
 $I\subseteq \{1,\dots, N\}$ such that $\{\phi_i,i\in I\}$ generates $F$
(resp. is a basis  of $F$) and such that $\tau$ is contained in the
cone generated by $\{\phi_i,i\in I\}$. The set of faces (resp.
vertices) of $\p(\Phi,\lambda)$ is in one-to-one correspondence with
$\CG(\Phi,\tau)$ (resp. $\CB(\Phi,\tau)$). The face which
corresponds to $I$ is
\begin{eqnarray}\label{def:faceI}
  \f_I(\Phi,\lambda) = \{x\in \R_{\geq 0}^N, \sum_{j=1}^N x_j\phi_j=\lambda, \;\; x_j=0 \mbox{ for }  j\in I^c\}.
  \end{eqnarray}
The affine tangent cone to $\p(\Phi,\lambda)$ at the
face $\f_I(\Phi,\lambda)$ is
\begin{equation}\label{eq:tangent-cone}
 \t_{I}(\Phi,\lambda)=\{x\in\R^N, \sum_{j=1}^N x_j\phi_j=\lambda, x_j\geq 0 \mbox {  for  } j\in I^c
 \}.
\end{equation}
If $\lambda$ is in $\c(\Phi)$, but is not in the tope $\tau$, then
the partition polytope $\p(\Phi,\lambda)$ is not empty, but its
faces are no longer in one-to-one correspondence with
$\CG(\Phi,\tau)$, (see Fig.\ref{chamber24_intro}).
Nevertheless,  the cone in (\ref{eq:tangent-cone}) makes sense
\textbf{for every }$\lambda\in F$: it remains
``the same cone" $\{x\in V;x_j\geq 0,j\in I^c\}$  up to a shift,
under the map   $V(\Phi,\lambda)\to V$ (see Formula (\ref{moveIg})).

We introduce now the main character
of this story,  the  function on $\R^N$ previously denoted by $\varchenko(\tau)$.
 \begin{definition}\label{geometric_brianchon_gram_function} The
 Geometric Brianchon-Gram
function is
\begin{equation*}\label{eq:geometric_brianchon_gram_function}
\varchenko(\Phi,\tau)=\sum_{I\in \CG(\Phi,\tau)}(-1)^{|I|-\dim
F}\prod_{j\in I^c}\suppresschi{[x_j\geq 0]}.
\end{equation*}
\end{definition}

Let us compute this function for the case of
$\Phi=(1,1)$ in $F=\R$.
Then
 $$\varchenko(\Phi,\tau)=[x_1\geq 0]+[x_2\geq 0]-[\R^2]$$ is equal to
 $$[x_1\geq 0,x_2\geq 0]-[x_1<0,x_2<0],$$
  the characteristic function of the closed positive quadrant minus
  the characteristic function of the open negative quadrant,  (Fig. \ref{varchenkodim1}).
\begin{figure}[!h]
\begin{center}
 \includegraphics[width=2 in]{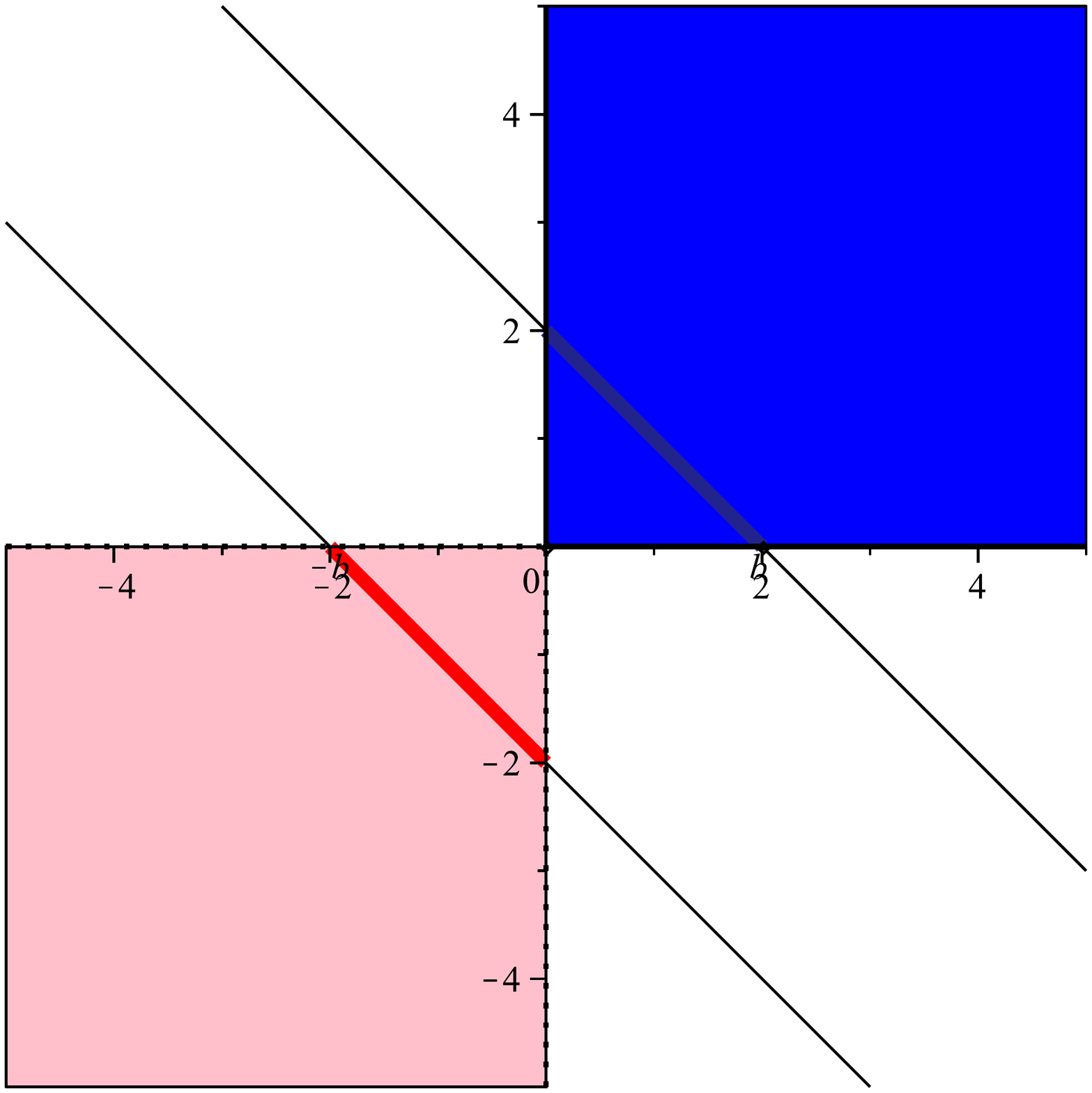}
 \caption{The function $\varchenko(\Phi,\tau)$ for $\Phi=(1,1)$}
 \label{varchenkodim1}
 \end{center}
\end{figure}

If $\lambda\in \tau$, the Brianchon-Gram theorem implies
\begin{equation}\label{eq:Xgeom_and_BGtheoremintro}
\varchenko(\Phi,\tau)\suppresschi{[V(\Phi,\lambda)]}=\suppresschi{[\p(\Phi,\lambda)]},
\end{equation}
the characteristic function of the partition polytope $\p(\Phi,\lambda)$.
However, the function
$\varchenko(\Phi,\tau)\suppresschi{[V(\Phi,\lambda)]}$ is defined for any
$\lambda\in F$. It is  a signed sum of characteristic functions
of closed cones intersected with the affine space $V(\Phi,\lambda)$.

For instance, in the case of $\Phi=(1,1)$, taking the product of $\varchenko(\Phi,\tau)$
with the characteristic function of the affine line $x_1+x_2=b$, we clearly recover the analytic continuation
 pictured in Fig. \ref{intervalagain}.

One of our first results (and our main technical tool)  (Theorem \ref{th:polarized-sum}) is the fact
that the Brianchon-Gram combinatorial function  $X(\Phi,\tau)(p,q)$
coincides with the analogous function associated with
any Lawrence-Varchenko polarized decomposition of a  polytope  into semi-closed cones
at vertices \cite{lawrence91}, \cite{varchenko}.

From this result, we deduce  that the function
$\varchenko(\Phi,\tau)\suppresschi{[V(\Phi,\lambda)]}$
is the signed sum of characteristic functions of semi-open polytopes,
in particular the support of this function is bounded for any $\lambda\in F$ (Corollary \ref{co:bounded}).

Reverting to the  framework of linear inequalities, we define now  $\ANA(\tau, b)$ to be the
inverse image of $\varchenko(\Phi,\tau)$
under the map  $v\to v+b$ from $V$ to $\R^N$.
For $b\in \tau$, $\ANA(\tau,b)$ is the characteristic function
of the polytope $\q(b)$.  For any  value of $b$,  it follows from the definition that $\ANA(\tau,b)$
is  the signed sum of the characteristic functions
of the tangent cones to the faces of the initial polytope $\q(b_0)$, with $b_0\in \tau$,
followed``by continuity ".
The above qualitative result  implies that  $\ANA(\tau,b)$
  is a signed sum of bounded faces
  of various dimensions of the mirage $\mu_i=b_i$.
  It is easy to see that $\ANA(\tau,b)$ enjoys the analyticity properties  stated above.

\bigskip

Our main result  is a wall crossing formula which we prove in a purely combinatorial context.

As the space $\R^N$ is the disjoint union of the semi-closed
quadrants $ Q_{\rm neg}^B: =\{x=(x_i); x_i
< 0 \mbox{ for } i\in B ,\,\,   x_i\geq 0 \mbox{ for } i\in B^c  \} $, we  write
$\varchenko(\Phi,\tau)$ in terms of the characteristic functions of these
quadrants.

 We introduce  the following polynomial in the
variables $p_i$ and $q_i$.
\begin{definition}  Let $\tau$ be a $\Phi$-tope.
The Combinatorial Brianchon-Gram  function associated to the pair
$(\Phi,\tau)$ is
\begin{equation}\label{Brianchon-Gram-combinatoireintro}
 X(\Phi, \tau)(p,q)= \sum_{I\in \CG(\Phi,\tau)}(-1)^{|I|-\dim
F}\prod_{i\in I^c}p_i \prod_{i\in I}(p_i+q_i) .
\end{equation}
\end{definition}
We recover $\varchenko(\Phi,\tau)$ when we substitute
$p_i=\suppresschi{[x_i\geq 0]}$  and $q_i=\suppresschi{[x_i< 0]}$
 in $X(\Phi, \tau)(p,q)$ (so that $p_i+q_i=1$).

For example, when $\Phi=(1,1)$, we have
$$
X(\Phi,\tau)=p_1(p_2+q_2)+p_2(p_1+q_1)-(p_1+q_1)(p_2+q_2)=p_1p_2-q_1q_2.
$$

The  polynomial $ X(\Phi,\tau)$ enjoys remarkable properties.
Let us say that the quadrant $Q_{\rm neg}^B$ is $\Phi$-bounded,
if the intersection of its closure $\overline{Q_{\rm neg}^B}$ with $V$ is reduced to $0$.
Equivalently,
the intersection of $Q_{\rm neg}^B$ with the affine space
$V(\Phi,\lambda)$ is bounded for any $\lambda\in F$.

We have
$$
X(\Phi,\tau)(p,q)=\sum_{B} z_B \prod_{i\in B}p_i \prod_{i\in B^c }q_i.
$$
where, for any subset $B\subseteq \{1,\dots, N\}$ such that $z_B\neq 0$,
the associated quadrant  $Q_{\rm neg}^B$ is $\Phi$-bounded. 
The coefficients $z_B$ are in $\Z$ and we give an algorithmic formula for them.

As we will observe in the last section, the decomposition in $\Phi$-bounded quadrants of $X(\Phi,\tau)$ is an analogue
of the fact that the $\overline \partial$ cohomology spaces
of a compact complex manifold are finite dimensional.

Our main result is Theorem \ref{th:wall_crossing}, where
we compute the function $X(\Phi,\tau_1)-X(\Phi,\tau_2)$, when  $\tau_1$ and $\tau_2$ be two adjacent topes
 (meaning that the intersection of their  closures  is contained  in a wall  and spans this wall).

We will not state the  formula for $X(\Phi,\tau_1)-X(\Phi,\tau_2)$ in this introduction,
but let us just mention a significant corollary,
the wall-crossing formula for the polytope $\p(\Phi,\lambda)$.
Let $A $ be
the set of $i\in \{1,\dots,N\}$ such that $ \phi_i$ belongs to  the
open side of $ H $   which contains $\tau_1$ (hence $-\phi_i$
belongs to the side of $\tau_2$).
Let $$
\p_{\rm flip}(\Phi,A,\lambda)=\{x\in V(\Phi,\lambda); x_i<0\,\, {\rm if}\,\, \, i\in A, x_i\geq 0 \,\,{\rm if}\,\, i\notin A\}.
$$
Thus $\p_{\rm flip}(\Phi,A,\lambda)$ is a semi-closed bounded polytope  in $V(\Phi,\lambda)$.

\begin{theorem} Let $\tau_1$ and $\tau_2$ be adjacent topes.
If $\lambda\in \tau_2$, we have
$$
\varchenko(\Phi,\tau_1)\suppresschi{[V(\Phi,\lambda)]}=\suppresschi{[\p(\Phi,\lambda)]}
+(-1)^{|A|} \suppresschi{[\p_{\rm flip}(\Phi,A,\lambda)]}.
$$
\end{theorem}
This formula is clearly inspired by the results of Paradan \cite{paradan2004}.
In turn, we show that it implies the convolution formula of Paradan
which expresses the jump of the number of lattice points of the partition polytope
in terms of the number of lattice points of some lower dimensional polytopes associated to $\Phi\cap H$.

The above formula implies that, after crossing a wall,  
the analytic continuation of the original polytope $\p(\Phi,\lambda)$ is  the signed sum of two polytopes,
among which one, but no more than one, may be empty.
 As illustrated in Section \ref{exampleB2}, we see the new polytope
$\p_{\rm flip}(\Phi,A,\lambda)$ starting to show his nose when $\lambda$ crosses the wall.
 To be  precise, the wall $H$ must separate  two chambers, not just two topes,
 (as explained in Remark \ref{trivial_wall_crossing}) in order for the new polytope
 $\p_{\rm flip}(\Phi,A,\lambda)$ to be not empty.

When $F$ is  provided with a lattice $\Lambda$, and the $\phi_i$'s are in $\Lambda$,
the data $(\Phi,\lambda)$ parameterize  a toric variety together  with a line bundle.
The zonotope
$$
\b(\Phi):=\{\sum_{i=1}^N t_i \phi_i;0\leq t_i\leq 1\}
$$
plays  an important role in the  ``continuity properties" of our formulae in the discrete case, where,
for a tope $\tau$,
the ``neighborhood" of  $\tau \cap \Lambda$ is  the fattened tope  $(\tau-\b(\Phi))\cap \Lambda$.
In Section \ref{Integrals_and_discrete_sums},
 where we study discrete sums over partition polytopes,
we recover the quasi-polynomiality  over fattened topes  which was previously
obtained  in  \cite{dahmen-micchelli-1988},
\cite{szenes-vergne-2002}, \cite{deconcini-procesi-vergne-dahmen-micchelli-2008},
as well as wall crossing formulae.
Remarkably, the proofs  which we give  in the present  article are
based only on the Brianchon-Gram decomposition of a polytope and some set theoretic computations.

\bigskip

Our original motivation for the present work was to understand
Brion's formula when specialized at a degenerate point.
Let $\p\subset V$ be a full-dimensional polytope in a
vector space $V$ equipped with a lattice $V_\Z$.
 Consider the discrete sum
$$
S(\p)(\xi)=\sum_{x\in V_\Z\cap \p}e^{\ll\xi,x\rr}.
$$
Brion's theorem expresses the analytic function $S(\p)(\xi)$ as the sum
\begin{eqnarray}\label{eq:brion-generating-function}
S(\p)(\xi)&=&\sum_{s\in \CV(\p)}S(s+\c_s)(\xi).
\end{eqnarray}
Here $s$ runs over the set of vertices of $\p$,
and $s+\c_s$ is the tangent cone at $\p$ at the vertex $s$.
Now the function $S(s+\c_s)(\xi)$ is a meromorphic function of $\xi$.
Its  poles are the points $\xi\in V^*$  such that $\xi$ vanishes on some edge generator of the cone $\c_s$
 (or equivalently,
such that $\xi$ takes the same value at the vertex $s$ and some adjacent vertex $s'$ of $\p$).

It is well known that if $\xi$ is regular with respect to $\p$,
 (i.e. $\langle\xi,s\rangle\neq \langle\xi,s'\rangle$ for adjacent vertices),
Brion's formula is the combinatorial translation
of the localization formula in equivariant cohomology \cite{MR685019},
in a case where the fixed points are isolated.
The case where $\xi$ is not regular corresponds to the case where the variety of  fixed points has components
of positive dimension.
We obtain indeed the combinatorial translation of the localization formula in this degenerate case.
The vertices must be replaced by the faces on which $\xi$ is constant which are maximal with respect to
this property. For such a face $\f$, the tangent cone must be replaced by the transverse cone to $\p$ along
$\f$. However, the formula  is ``nice"  only under some conditions
 (satisfied for example when the  polytope $\p$ is simple).
The formula
involves  the ``analytic continuation" of the face $\f$
obtained by slicing the polytope $\p$ by  affine subspaces parallel to $\f$,
Fig.\ref{fig:Brianchon-Gram continuation of a face}

Finally, in the last section, we sketch  the relation of this work
with the cohomology of line bundles over a toric variety.
In the case where the $\phi_i$'s generate a lattice in $F$,
a $\Phi$-tope $\tau$ gives rise to a toric variety $M_\tau$.
Then,
 the value of the function $\varchenko(\Phi,\tau)$ computed at  a point
 $m\in \Z^N\subset \R^N$ is the multiplicity of the character $m$
 in the alternate sum of the cohomology groups of the line bundle $L_\lambda$ on $M_\tau$
which corresponds to $\lambda=\sum_i m_i\phi_i$.
In other words, the function  $\varchenko(\Phi,\tau)$ induces on each affine space $V(\Phi,\lambda)$ the constructible function associated by Morelli \cite{MR1234308} to the line bundle $L_\lambda$ on $M_\tau$.

 The continuity result (Corollary \ref{th:continuity-on-closed-tope}) implies that the function $\lambda\to \dim H^0(M_\tau,\lambda)$ is a quasi-polynomial  on the fattened tope $(\tau-\b(\Phi))\cap \Lambda$.
  We give some examples of computations in the last section.

These results have been presented by the second author M.V. in  the Workshop: Arrangements of Hyperplanes held in Pisa in June 2010.
 M.V. thanks C. De Concini, H. Schenck and  M. Wachs
  for numerous discussions on posets, cohomology of line bundles on toric varieties,
  during this special period, and thanks the Centro de Giorgi for providing such a stimulating atmosphere.
 The idea of this  article  arose  while both authors were enjoying a
Research in Pairs stay at Mathematisches
Forschungsinstitut Oberwolfach in March/April 2010.
The support of MFO is   gratefully acknowledged.

We thank P. Johnson for drawing our attention to Varchenko's work and
to  the paper \cite{cavalieri-2010} where applications of Varchenko's work
to  wall crossing formulae for Hurwicz numbers are obtained.

\section*{List of notations}

\[\begin{array}{ll}
\ANA(b), \ANA(\tau,b)   &\mbox{a function on  $\R^d$,}\\
                         & \mbox{the  analytic continuation of a polytope}\\
{[E]}   &\mbox{characteristic function of a set $E\subseteq R^d$ or $E\subseteq \R^N$}\\
 F      &\mbox{r-dimensional real vector space; $\lambda\in F$}\\
 \Phi   &\mbox{a sequence of  N non zero vectors $\phi_i$ in $F$ }\\
e_i     & \mbox{canonical basis of $\R^N$}\\
x_i     & \mbox{coordinates functions on $\R^N$}\\
M       &\mbox{the map $\R^N\to F; M(e_i)=\phi_i$}\\
V(\Phi) ,V  & \mbox{$\{x\in\R^N; \sum_i x_i\phi_i=0\}$}\\
d& N-r; \mbox{the dimension of $V$}\\
V(\Phi,\lambda) & \mbox{$\{x\in\R^N; \sum_i x_i\phi_i=\lambda\}$}\\
\p(\Phi,\lambda)   & \mbox{$\{x\in\R^N; x_i\geq 0; \sum_i x_i\phi_i=\lambda\}$}\\
\mbox{Partition polytope}  &\mbox{a polytope $\p(\Phi,\lambda)$}\\
\Lambda    &\mbox{lattice in $F$;  $\lambda\in \Lambda$}\\
\mbox{ $k(\Phi)$}  &\mbox{the function $\lambda\to {\rm cardinal}\,\,\p(\Phi,\lambda)$}\\
\mbox{Partition function} &\mbox{the function $k(\Phi)$}\\
Q&  \mbox{ the standard quadrant $\{x\in\R^N; x_i\geq 0\}$}\\
I,J,K, A, B&\mbox{subsets of $\{1,2,\ldots,N\}$}\\
I^c&\mbox{complementary subsets to $I$ in $\{1,2,\ldots,N\}$}\\
\Phi_I& \mbox{$(\phi_i,i\in I)$}\\
\c(\Phi), \c(\Phi_I) &\mbox{cone generated by $\Phi,  \Phi_I$}\\
\a(K)   &\mbox{the cone in $\R^N$ defined as }\\
        & \{ x\in \R^N;  x_j\geq 0 \mbox{ for } j\in K^c\}\\
\a_0(K)&\mbox{the cone $V\cap \a(K)$}\\
 \t_K(\Phi,\lambda)&\mbox{$\a(K)\cap V(\Phi,\lambda)$}\\
\mbox{ $\Phi$-basic subset $I$ }&\mbox{ a subset $I$ such that $\phi_i$, $i\in I$, is a basis of $F$}\\
\mbox{$\Phi$-generating subset $I$ }&\mbox{ a subset $I$ such that $\phi_i$, $i\in I$, generates  $F$}\\
  \mathcal{B}(\Phi) &\mbox{the set of $\Phi$-basic subsets}\\ \mathcal{G}(\Phi) &\mbox{the set of  $\Phi$-generating subsets}\\
\rho_{\Phi,I}&\mbox{$\rho_{\Phi,I}: \R^N\to V(\Phi)$ with kernel $\oplus_{i\in I} \R e_{i}$}\\
g_j^I&\mbox{$\rho_{\Phi,I}(e_j), j\in I^c$}\\
{\rm wall} \, H &\mbox{hyperplane in $F$ generated by $r-1$ vectors in $\Phi$}\\
\mbox {regular } \lambda & \mbox{$\lambda$ does not belong to any wall $H$}\\
{\rm tope}\, \tau   & \tau\subset F, \mbox{a connected component}\\
                    &\mbox{of the set of regular elements}\\
\mathcal{B}(\Phi,\tau) &\mbox{the set  of $\Phi$-basic subsets $I$ such that $\tau\subset \c(\Phi_I)$}\\
\mathcal{G}(\Phi,\tau)  &\mbox{the set of  $\Phi$-generating subsets $I$ such that $\tau\subset \c(\Phi_I)$}\\
{\rm arrangement} \CH(\lambda) &\mbox{ the collection of the  hyperplanes $x_i=0$ in $V(\Phi,\lambda)$}
\end{array}\]

\[\begin{array}{ll}
\mbox{vertex $s$}   &\mbox{ of the arrangement} \CH(\lambda); \mbox{$s$ belongs to} \\
                    &\mbox{  $d$ hyperplanes of $\CH(\lambda)$}\\
s_I(\Phi, \lambda)  &\mbox{ the vertex of  $\CH(\lambda)$ such that $s_j=0$ for $j\in I^c$}\\
\f_I(\Phi, \lambda), \f_I   &\mbox{the face of $\p(\Phi,\lambda)$  indexed by $I$; defined by }\\
                            & \p(\Phi,\lambda)\cap\{x_j=0, j\in I^c\}\\
 \t_{\aff}(\p, \f)&\mbox{tangent affine cone to a polytope $\p$ at the face $\f$}\\
\varchenko(\Phi,\tau)&\mbox{$\sum_{I\in \CG(\Phi,\tau)}(-1)^{|I|-\dim
F}\prod_{j\in I^c}[x_j\geq 0]$}\\
X(\Phi, \tau)(p,q)&\mbox{ $\sum_{I\in \CG(\Phi,\tau)}(-1)^{|I|-\dim
F}\prod_{j\in I^c}p_j \prod_{i\in I}(p_i+q_i)$ }\\
w_B&\mbox{$\prod_{j\in B^c}p_j \prod_{i\in B}q_i$}\\
W&\mbox{space of polynomials with basis $w_B$}\\
\geom&\mbox{ substituting $p_i=[x_i\geq 0]$, $q_j=[x_j<0]$ in $w_B$}\\
\b(\Phi)    &\mbox{the zonotope generated by } \Phi; \\
            & \{\sum_{i=1}^N t_i \phi_i;0\leq t_i\leq 1\} \\
Q_{\rm neg}^B   & \mbox{$\{x=(x_i),  x_i < 0 $ for $ i\in B$ ; $x_i\geq 0 $ for $i\in B^c\}$}\\
\Phi_{{\rm flip}}^B &\mbox{ the sequence } (\sigma_i \phi_i, 1\leq i\leq N), \mbox{ where } \\
                    & \mbox{ $\sigma_i=-1$ if $i\in B$;  $\sigma_i=1$ if $i\notin B$}\\
\widetilde{\c}(\Phi_{{\rm flip}}^B)&\mbox{$\{\sum_i x_i \phi_i , x\in Q_{\rm neg}^B\}$}\\
  \widetilde{\c}_\Z(\Phi_{{\rm flip}}^B)&\mbox{$\{\sum_i x_i \phi_i, x\in
Q_{\rm  neg}^B\cap \Z^N\}$}\\
 \beta& \mbox{linear form on $\R^N$}\\
 K^{c,+}_\beta&\mbox{$\{ j\in K^c;\langle\beta, g_j^K \rangle > 0\}$}\\
K^{c,-}_\beta&\mbox{$\{ j\in K^c;\langle\beta, g_j^K \rangle < 0\}$}\\
\a(K,\beta)&\mbox{$\{ x\in \R^N;
x_i\geq 0\,\,, i\in {K_\beta^c}^+;
x_i< 0 \,\,, i\in {K_\beta^c}^- \}$}\\
\a_0(K,\beta)&\mbox{the cone $V\cap \a(K,\beta)$}\\
Y(\Phi,\tau,\beta)&\mbox{$
\sum_{K\in \CB(\Phi,\tau)}(-1)^{|{K^c_\beta}^-|}\prod_{i\in
{K^c_\beta}^+}p_i \prod_{i\in  {K^c_\beta}^-}q_i \prod_{i\in K
}(p_i+q_i)$}\\
\p(\Phi,A,\lambda)&\mbox{$\{x\in V(\Phi,\lambda), \;\;  x_i>0 \mbox{ for }
i\in A, x_i\geq 0 \mbox{ for } i\in A^c\}$}\\
\p_{\rm flip}(\Phi,A,\lambda)&\mbox{$\{x\in V(\Phi,\lambda)\;\; x_i<0  \mbox{ for }  i\in A, x_i\geq 0
\mbox{ for } i\in A^c\}$}\\
\end{array}\]

\newpage
\section{Definition of the analytic continuation}

\subsection{Some cones related to a partition polytope}

In this article, there will be plenty of cones.
A cone will always  be an affine polyhedral convex cone.
A cone  will be called  flat if it contains  an  affine line,
otherwise, it will be  called salient.

  Let $F$  be a real vector space of dimension $r$, and let
  $  \Phi=(\phi_1,\ldots, \phi_N)$
 be a sequence of  $N$ non zero elements of  $F$.
We assume that $\Phi$ generates $F$ as a vector space.

The standard basis of $\R^N$  is denoted by  $e_i$ with dual basis the linear forms $x_i$.
We denote by $M:\R^N \to F$ the surjective map
which sends the vector $e_i$ to the vector $\phi_i$.
The kernel of $M$ is a subspace of dimension $d=N-r$ which will be denoted by $V(\Phi)$ or simply $V$ when $\Phi$ is understood.
$$
V(\Phi):=\{x\in \R^N; \sum_i x_i\phi_i=0\}.
$$

We denote by $Q$  the standard quadrant
$$
Q:=\{x\in \R^N; x_i\geq 0\}.
$$
The cone $\c(\Phi)$ generated by $\Phi$ is  the image of $Q$ by $M$.
Assume that  the cone $\c(\Phi)$ is salient.
In other words, there exists a linear form $a\in F^*$
such that $\langle a,\phi_i \rangle > 0$ for all $1\leq i\leq N$.
This is also equivalent to the fact that $V\cap Q=0$.

If $I$ is a subset of $\{1,2,\ldots, N\}$, we denote by $I^c$ the complementary subset
to $I$ in   $\{1,2,\ldots, N\}$.
\begin{definition}\label{def:coneI}
If $I$ is a subset of $\{1,2,\ldots, N\}$, let
$$
\a(I)=\{ x\in \R^N;
x_j\geq 0 \mbox{ for } j\in I^c\}
$$
and let
$$
\a_0(I)=V\cap \a(I)=\{ x\in V;\; x_j\geq 0 \mbox{ for } j\in I^c\}
$$
be the intersection of $V$ with the cone $\a(I)$.
\end{definition}
Thus $\a(I)$ is the product of the positive quadrant in the variables $I^c$,
with a vector space of dimension $|I|$.
The cone $\a(I)$ is called an angle by Varchenko.
It is never salient, except if $I=\emptyset$.
With this notation, the positive quadrant  $Q$ is  $\a(\emptyset)$.

We now analyze the cone  $\a_0(I)\subseteq V$.
A subset $I\subseteq \{1,2,\ldots, N\}$ such that  $\{\phi_i,i\in I\}$ is  a basis of $F$ will be called
$\Phi$-basic. We denote by $\mathcal B(\Phi)$ the set of $\Phi$-basic subsets.
 A subset $I\subseteq \{1,2,\ldots, N\}$ such that  $\{\phi_i,i\in I\}$ generates  $F$ will be called
$\Phi$-generating. We denote by $\mathcal G(\Phi)$ the set of $\Phi$-generating subsets.

Let $I$ be $\Phi$-basic. Then the cardinal of $I^c$ is $d=N-r$ and
the restrictions  to $V$ of the linear forms $x_j$, with $j\in I^c$, form a basis of $V^*$.
Hence  $\a_0(I)$ is a   cone of dimension $d$ in $V$ with $d$ generators,
in other words a simplicial cone of full dimension in the vector space $V$.
Let us describe the edges of the simplicial  cone $\a_0(I)$.
 We have
  $$
  \R^N=V(\Phi)\oplus (\oplus_{i\in I}\R e_i),
  $$
  and  we denote by $\rho_{\Phi,I}$ the corresponding linear projection
$\R^N \to V$.
For $j\in I^c$, we write $\phi_j=\sum_{i\in I} u_{i,j} \phi_i$.

\begin{lemma}\label{lemma:later}
Let $I$ be $\Phi$-basic. For $j\in I^c$, let
 $$
g_j^I= \rho_{\Phi,I}(e_j)=e_j-\sum_{i\in I}u_{i,j} e_i.
$$
Then the $d$ vectors $g_j^I$ are the generators  of the edges of the simplicial cone $\a_0(I)$.
\end{lemma}

Now, let $I$ be a generating subset.
Then the restrictions to $V$ of the linear forms $x_j$, $j\in I^c$,
are linearly independent elements of  $V^*$.
The  cone $\a_0(I)$ is again the  product of a simplicial cone of dimension $|I^c|$ by  a vector space of dimension $|I|-r$
More precisely, if $K$ is any $\Phi$-basic subset contained in $I$,
the cone $\a_0(I)$ is the product of the cone generated by $\rho_{\Phi,K}(e_j), j\in I^c,$
 by  the vector space generated by   $\rho_{\Phi,K}(e_i)$ with $i\in I \setminus K$.

\subsection{Vertices and faces of a partition polytope}

Recall that, for $\lambda\in F$, we denote by $V(\Phi,\lambda)\subset \R^N$ the affine subspace
$$\{x\in \R^N; \sum_i x_i\phi_i=\lambda\}.$$
The intersections  of the coordinates hyperplanes $\{x_i=0\}$ with $V(\Phi,\lambda)$
form an arrangement  $\mathcal H(\lambda)$ of $N$ affine hyperplanes of $V(\Phi,\lambda)$.

By definition, a vertex of this arrangement is a point $s\in V(\Phi,\lambda)$
 such that $s$ belongs to at least $d$ independent hyperplanes.
The arrangement $\mathcal H(\lambda)$ is called regular if no vertex belongs to more than $d$ hyperplanes.
A $\Phi$-wall $H$ is a hyperplane of $F$ spanned by $r-1$ linearly independent elements of $\Phi$.
Thus $\mathcal H(\lambda)$ is regular
 if and only if $\lambda$ does not belong to any $\Phi$-wall, that is, if $\lambda$ is regular.

By definition, a face  of the arrangement $\mathcal H(\lambda)$
  is the set of elements  $x\in V(\Phi,\lambda)$ which satisfy
  a subset of the  set of relations $\{x_i\geq 0, x_j\leq 0, x_k=0\}$.

Recall that the partition polytope $\p(\Phi, \lambda)$
 is the intersection of the affine space $V(\Phi,\lambda)$ with the positive quadrant $Q$.
 Thus it is  a bounded face of the arrangement of hyperplanes $\mathcal H(\lambda)$.

If  $I\subset \{1,\dots,N\}$ is $\Phi$-basic, then $\lambda$ has a unique decomposition
$\lambda=\sum_{i\in I}x_i \phi_i$.
If $\lambda$ is regular,   $x_i\neq 0$  for all $i$.
\begin{definition}
Let $I$ be  a $\Phi$-basic subset, let
 $\lambda=\sum_{i\in I}x_i \phi_i$.
Then  $s_I(\Phi, \lambda)$ is the vertex of  the arrangement $\mathcal H(\lambda)$ defined
by $s_I(\Phi, \lambda)=(s_i)$ where
$s_i=x_i$ if $i\in I$, and $s_j=0$ if $j\in I^c$.
\end{definition}
Observe that $s_I(\Phi, \lambda)$ depends linearly on $\lambda$.

If $\lambda$ is regular,   the vertices of the arrangement
$\mathcal H(\lambda)$ are in one to one correspondence  $I\mapsto s_I(\Phi, \lambda)$
with the set $\mathcal B(\Phi)$ of $\Phi$-basic subsets of $\{1,\dots,N\}$.

\begin{definition}
For $I$ a subset of $\{1,2,\ldots,N\}$,
 define
 $$
 \t_I(\Phi,\lambda)=\a(I)\cap V(\Phi,\lambda).
 $$
 \end{definition}
If $I$ is a $\Phi$-basic subset, then the  cone $\t_I(\Phi,\lambda)$
  is the shift $s_I(\Phi, \lambda)+\a_0(I)$
of the \textbf{fixed simplicial cone} $\a_0(I)$
 by the vertex $s_I(\Phi, \lambda) $ which depends linearly of $\lambda$.
 \begin{equation}\label{moveI}
\t_I(\Phi,\lambda)=\a(I)\cap V(\Phi,\lambda)=s_I(\Phi, \lambda)+\a_0(I).
\end{equation}
Similarly, if $I$ is a $\Phi$-generating subset, choose a $\Phi$-basic subset  $K$ contained in $I$,
then the cone $\t_I(\Phi,\lambda)$
  is the shift  of the \textbf{fixed cone}
$\a_0(I)$  by the vertex $s_K(\Phi,\lambda) $ which depends linearly of $\lambda$.
\begin{equation}\label{moveIg}
\t_I(\Phi,\lambda)=\a(I)\cap V(\Phi,\lambda)=s_K(\Phi,\lambda)+\a_0(I).
\end{equation}

So, one can  say  that the set
 $\t_I(\Phi,\lambda)$ varies analytically with $\lambda$, whenever $I$ is a generating subset.
 At least it ``keeps the same shape".
 This is not the case when $I$ is not generating, for example when $I=\emptyset$.
 Indeed $\t_\emptyset(\Phi,\lambda)$ is the partition polytope $\p(\Phi,\lambda)$,
 and it certainly does not vary ``analytically".

We now analyze the faces of the partition polytope $\p(\Phi,\lambda)$
and the corresponding tangent cones.

If $\tau$ is a $\Phi$-tope,  we denote by
$\CB(\Phi,\tau)\subseteq \CB(\Phi)$ the set of basic subsets
$I$ such that $\tau$ is contained in the cone $\c(\phi_I)$ generated by the $\phi_i, i\in I$.
In other words, the equation
$\lambda=\sum_{i\in I} x_i\phi_i$ can be solved with positive $x_i$. Equivalently,
the corresponding vertex $s_I(\Phi, \lambda)$ belongs to the polytope $\p(\Phi,\lambda)$.
Thus when $\lambda$ is regular, there is a one-to-one correspondence between the elements $I\in \mathcal B(\Phi,\tau)$ and the vertices of the polytope $\p(\Phi,\lambda)$.

When $\lambda$ belongs to the closure of a tope $\tau$,
every vertex  of $\p(\Phi,\lambda)$ is still of
the form $s_I(\Phi, \lambda)$ with $I\in \mathcal B(\Phi,\tau)$,
but two $\Phi$-basic subsets can give rise to the same vertex.

Let $I\in \mathcal B(\Phi,\tau)$.
Assume that $\lambda$ is regular, so that all coordinates $s_i$ of $s_I(\Phi, \lambda)$ with $i\in I$ are positive.
Then it is clear that the tangent cone  to $\p(\Phi,\lambda)$ at the vertex $s_I(\Phi, \lambda)$
is the cone determined by the inequations $x_i\geq 0$ for $i\in I^c$,
while the sign of  the coordinates $x_i$ with $i\in I$ are arbitrary,
In other words, it is the simplicial affine cone  $\t_I(\Phi,\lambda)$

We denote by $\CG(\Phi,\tau)\subseteq \CG(\Phi)$
the set of generating subsets $I$ such that $\tau$
is contained in the cone $\c(\phi_I)$ generated by the $\phi_i, i\in I$.

If $I\in \CG(\Phi,\tau)$,
the intersection of $\p(\Phi,\lambda)$ with $\{x_j=0, j\in I^c\}$
is a face $\f_I(\Phi,\lambda)$ of dimension $|I|-r$ of the polytope $\p(\Phi,\lambda)$.
The vertices of this face are the points
 $s_K(\Phi,\lambda)$ corresponding to all the  $\Phi$-basic subsets $K$ contained in $I$.
 The affine tangent cone $\t_{\aff}(\p(\Phi,\lambda),f_K(\Phi,\lambda))$
  to the polytope $\p(\Phi,\lambda)$ along the face $f_K(\Phi,\lambda)$ is
 $$
 \t_{\aff}(\p(\Phi,\lambda), f_K(\Phi,\lambda))=\t_I(\Phi,\lambda)=\a(I)\cap V(\Phi,\lambda).
 $$

\subsection{The  Brianchon-Gram function}
Summarizing, for $\lambda\in \tau$, there is a one-to one correspondence
between the set of faces of the polytope $\p(\Phi,\lambda)$ and the set $\mathcal G(\Phi,\tau)$.
The Brianchon-Gram theorem implies, for $\lambda\in \tau$,
$$
\suppresschi{[\p(\Phi,\lambda)]}=\left(\sum_{I\in \mathcal G(\Phi,\tau)} (-1)^{|I|-\dim F} \suppresschi{[\a(I)]}\right) \suppresschi{[V(\Phi,\lambda)]}.
$$
When $\lambda$ varies,  the right hand side is obtained
by intersecting a number of fixed cones in $\R^N$   with the varying affine
space $V(\Phi,\lambda)$.
It is  natural to introduce the  function on $\R^N$
\begin{eqnarray}\label{eq:geometric_brianchon_gram_}
\nonumber \varchenko(\Phi,\tau)&=&  \sum_{I\in \CG(\Phi,\tau)}(-1)^{|I|-\dim
F}\suppresschi{[\a(I)]}\\
&=& \sum_{I\in \CG(\Phi,\tau)}(-1)^{|I|-\dim
F}\prod_{j\in I^c}\suppresschi{[x_j\geq 0]}.
\end{eqnarray}
that is,  the Geometric Brianchon-Gram  function which we mentioned  in the introduction.

For  $\lambda\in \tau$, we have
\begin{equation}\label{eq:Xgeom_and_BGtheorem}
\varchenko(\Phi,\tau)\suppresschi{[V(\Phi,\lambda)]}=\suppresschi{[\p(\Phi,\lambda)]},
\end{equation}
 the characteristic function of the partition polytope $\p(\Phi,\lambda)\subset \R^N$.


Let us now consider the function $\varchenko(\Phi,\tau)\suppresschi{[V(\Phi,\lambda)]}$  for any $\lambda\in F$.

By Equations (\ref{moveI}) and (\ref{moveIg}), we have
$$\varchenko(\Phi,\tau)\suppresschi{[V(\Phi,\lambda)]}=
\sum_{I\in \CG(\Phi,\tau)}(-1)^{|I|-\dim
F}\suppresschi{[s_K(\Phi,\lambda)+\a_0(I)]}.$$

Here, for each $I\in \CG(\Phi,\tau)$, we choose  $K\subset I$, a basic subset contained in  $I$.

We thus see that
$\varchenko(\Phi,\tau)\suppresschi{[V(\Phi,\lambda)]}$
 is constructed as follows.
Start from the polytope $\p(\Phi,\lambda_0)$ with $\lambda_0\in \tau$,
write the characteristic function of $\p(\Phi,\lambda_0)$ as
the alternate sum of its tangent cones at faces, and when moving $\lambda$ in the whole space $F$,
follow these cones by moving their vertex linearly in function of $\lambda$.
As all the sets $I$ entering in the formula for $\varchenko(\Phi,\tau)$
are generating, the individual pieces $\a(I) \cap V(\Phi,\lambda)=s_K(\Phi,\lambda)+\a_0(I)$  keep the same shape.

It is clear that the support of the function $\varchenko(\Phi,\tau)\suppresschi{[V(\Phi,\lambda)]}$
is a union  of faces of various dimensions of  the arrangement $\mathcal H(\lambda)$.
We will show that it is 
is a union of bounded faces of this arrangement,   for any $\lambda\in F$  (Corollary \ref{co:bounded}). 
\begin{remark}\label{topes and chambers}
Chambers rather than topes are relevant to wall crossing.
However, we preferred to use topes, because topes are naturally
related to the whole set of vertices of the arrangement  $\CH(\lambda)$.
A chamber is a connected component
of the complement in $F$ of the union of all the {\bf cones} spanned by $(r-1)$-elements of $F$.
Chambers are bigger than topes, the closure of a chamber is a  union of closures of topes.
See Figure \ref{fig:3topes}.
But if $\tau_1$ and $\tau_2$ are contained in the  same chamber, we have $\CG(\Phi,\tau_1)=\CG(\Phi,\tau_2)$,  hence
$X(\Phi,\tau_1)=X(\Phi,\tau_2).$
 \end{remark}

\begin{figure}[!h]
\begin{center}
  \includegraphics[width=1.5 in]{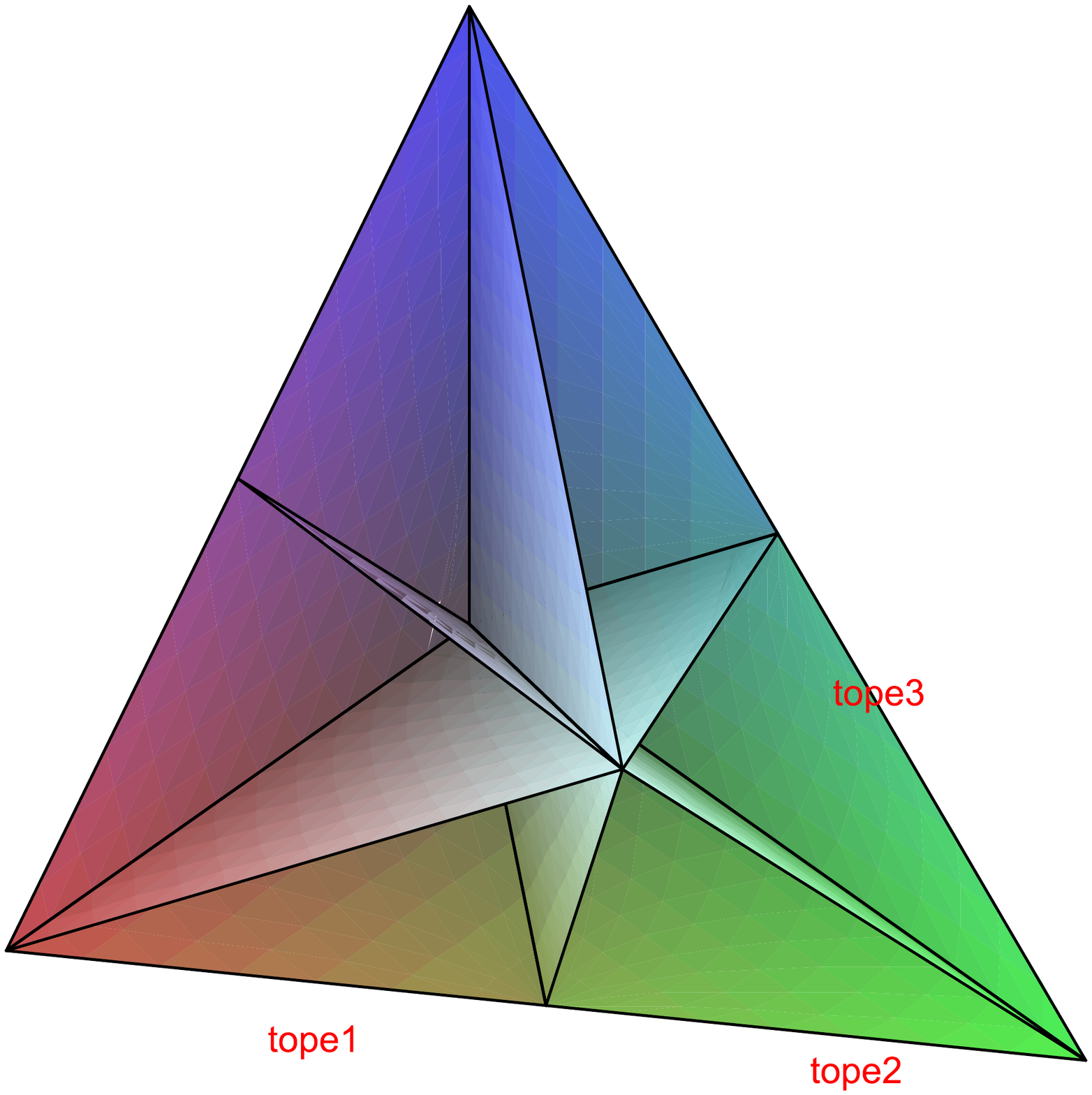}\includegraphics[width=1.5 in]{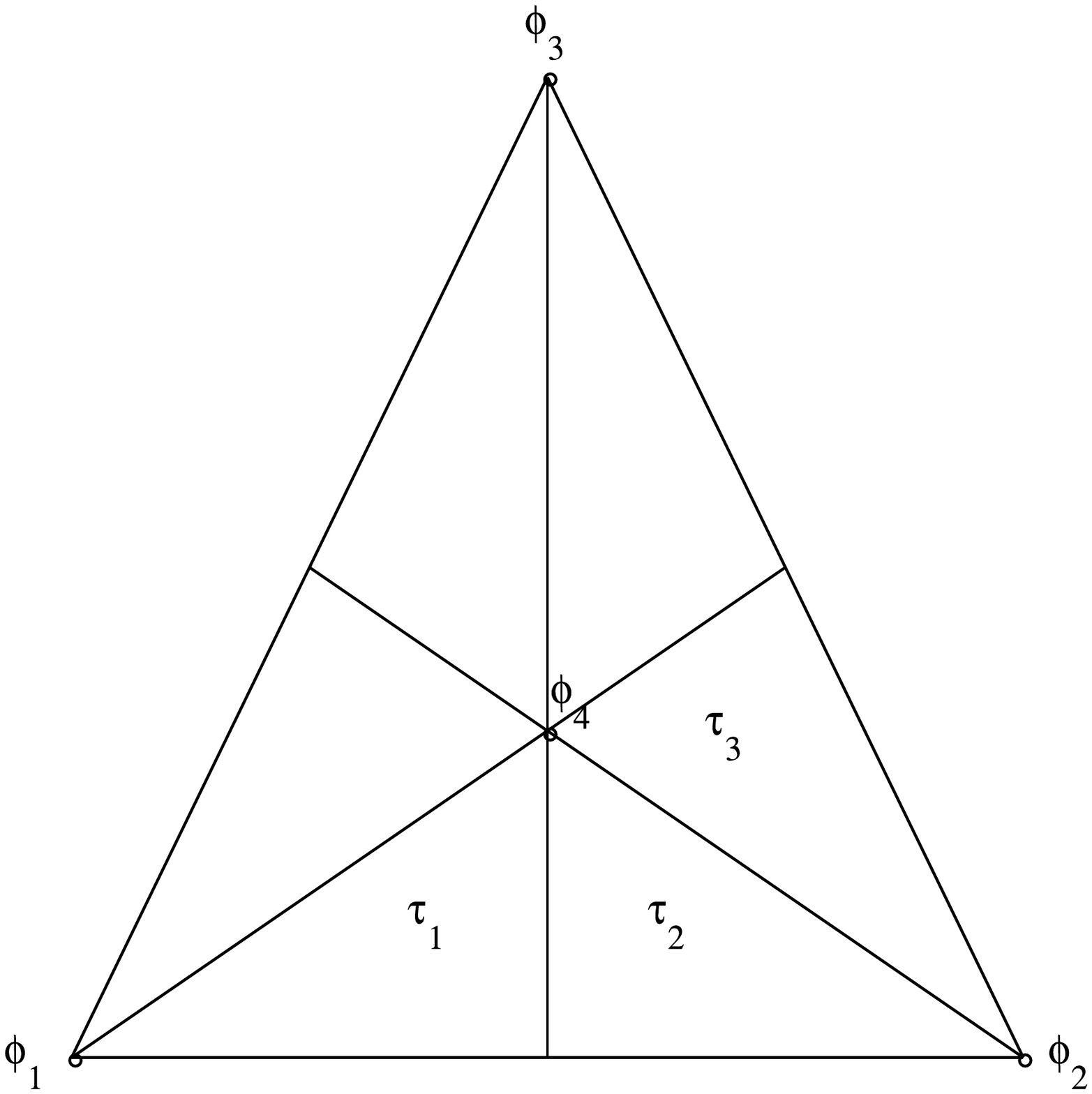}
  \includegraphics[width=1.5 in]{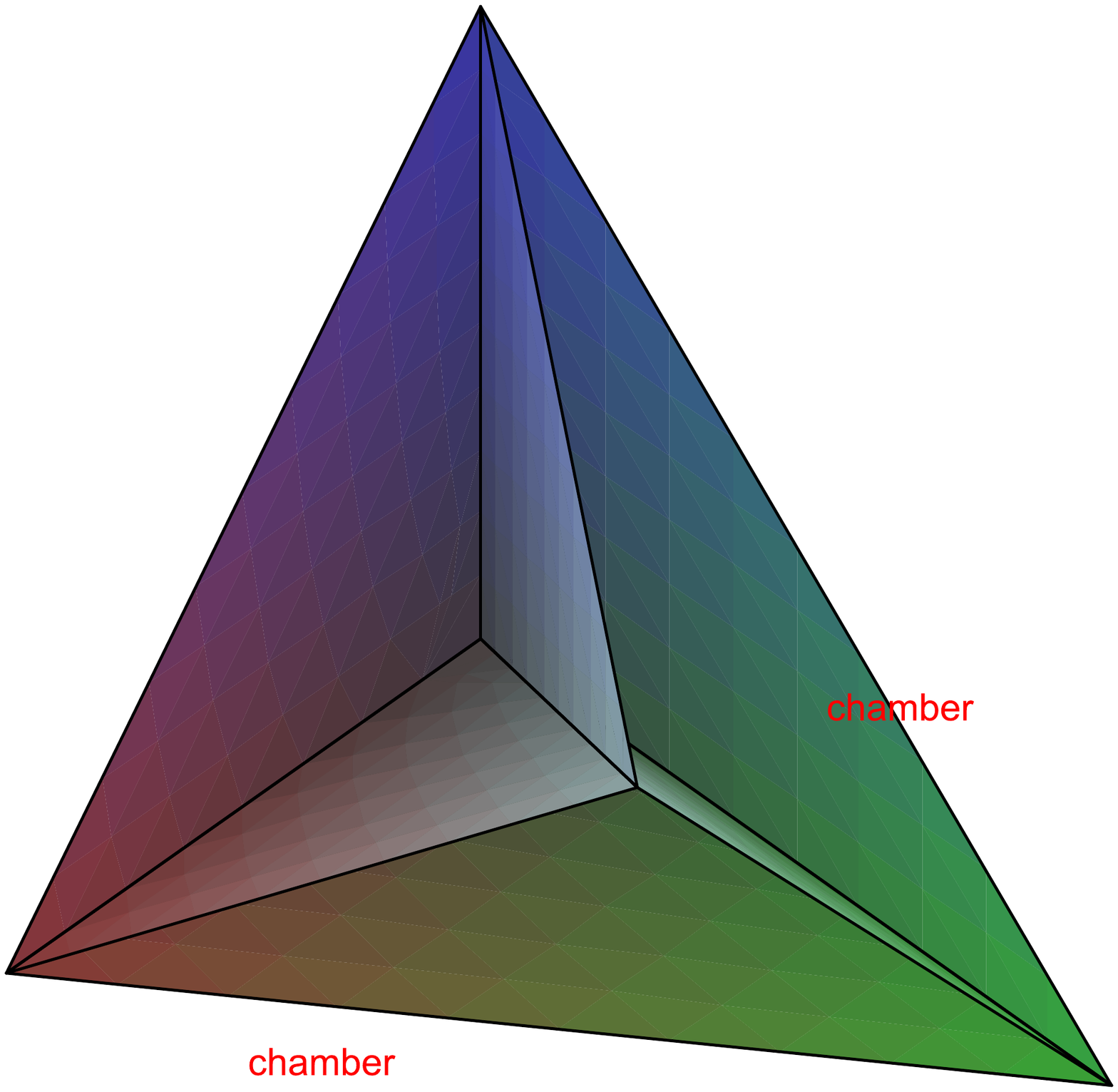}
  \caption{ Left, topes for $\Phi=(\phi_1,\phi_2,\phi_3,\phi_1+\phi_2+\phi_3)$.
  Right, chambers.}
  \label{fig:3topes}
  \end{center}
\end{figure}

\section{Signed sums of quadrants}

\subsection{Continuity properties of the Brianchon-Gram  function}

Recall that we defined in the introduction   the following polynomial in the
variables $p_i$ and $q_i$.
\begin{definition}  Let $\tau$ be a $\Phi$-tope.
The Combinatorial Brianchon-Gram  function associated to the pair
$(\Phi,\tau)$ is
\begin{equation}\label{Brianchon-Gram-combinatoire}
 X(\Phi, \tau)(p,q)= \sum_{I\in \CG(\Phi,\tau)}(-1)^{|I|-\dim
F}\prod_{j\in I^c}p_j \prod_{i\in I}(p_i+q_i) .
\end{equation}
\end{definition}

If the tope $\tau$ is not contained in $\c(\Phi)$, the set
$\CG(\Phi,\tau)$ is empty and $X(\Phi,\tau)=0$.
Otherwise, if $\tau\subset \c(\Phi)$, the sum  defining $X(\Phi,\tau)$
is indexed by all the faces of the  polytope $\p(\Phi,\lambda_0)$ (for any choice of $\lambda_0\in \tau$).

We recover $\varchenko(\Phi,\tau)$ when we substitute
$\suppresschi{[x_i\geq 0]}$ for $p_i$ and $\suppresschi{[x_i< 0]}$ for $q_i$ in $X(\Phi, \tau)(p,q)$ (so that $p_i+q_i=1$).

The Combinatorial Brianchon-Gram function is a particular element of
the  space $W$ below.
\begin{definition} Let $W$ be the subspace of
$\Q[p_1,\dots, p_N,q_1,\dots, q_N]$ which consists of linear
combinations of the monomials
$$
w_B= \prod_{j\in B^c}p_j \prod_{i\in B}q_i
$$
where $B$ runs over the subsets of $\{1,\dots,N\}$.
\end{definition}
Thus we have
\begin{equation}\label{eq:coeffs_combinatorial_BG}
X(\Phi,\tau)=\sum_B z(\Phi,\tau,B) w_B,
\end{equation}
with coefficients $z(\Phi,\tau,B)\in \Z$.

\noindent \emph{Remark}: For the subset $B=\{1,2,\ldots,N\}$, the coefficient  $z(\Phi,\tau,B)$  is $(-1)^d=(-1)^{N-r}$.

\begin{example}[The standard
knapsack]\label{standard_simplex} Let $F=\R$, $\phi_i=1$ for
$i=1,\dots,N$, and  $\tau=\R_{>0}$. From the usual
inclusion-exclusion relations, we get
\begin{equation}\label{standard-simplex}
X(\Phi,\tau)= p_1\cdots p_N -(-1)^N q_1\cdots q_N.
\end{equation}
\end{example}

An element in $W$ gives a function on $\R^N$ by the following substitution.
\begin{definition}
We denote by $\geom$ the map  from $W$ to the space of functions
on $\R^N$ defined by substituting
$\suppresschi{[x_i\geq 0]}$ for $p_i$ and $\suppresschi{[x_j< 0]}$ for $q_j$.
\end{definition}
Later, we will use other substitutions.

\bigskip

We  prove now some  ''continuity"  properties of the Combinatorial
Brianchon-Gram function when $\lambda$ reaches the closure of the tope $\tau$. Actually, these properties are shared by
any element of the space $W$ which satisfies the hypothesis of
Proposition \ref{th:quadrants} below.
We first introduce some definitions and prove an easy lemma.

\begin{definition}
The zonotope $\b(\Phi)$ is the subset of $F$ defined by
$$\b(\Phi)=\{\sum_{i=1}^N t_i \phi_i;0\leq t_i\leq 1\}.$$
 \end{definition}

When $\tau$ is a tope contained in $\c(\Phi)$,
the domain $\tau-\b(\Phi):=\{x-y,x\in \tau,y\in \b(\Phi)\}$
will play a crucial role in  ``continuity properties" of our functions.
Remark that $\tau-\b(\Phi)$  is a fattening of $\tau$
which  contains the closure of the tope $\tau$.
Usually the set of integral points in $\tau-\b(\Phi)$ is larger than the set of integral points in $\overline \tau$.

\begin{definition}\label{def:flipped_Phi}
 For $B\subseteq \{1,\dots,N\}$,

 $\bullet$
$Q_{\rm neg}^B\subset \R^N$ is the semi-closed quadrant
$$
Q_{\rm neg}^B=\{x=(x_i),  x_i < 0 \mbox{ for } i\in B , x_i\geq 0 \mbox{ for }
i\in B^c\},
$$

$\bullet$
$\Phi_{\rm flip}^B$ is  the sequence $[\sigma_i \phi_i, 1\leq i\leq N]$, where
$\sigma_i=-1$ if $i\in B$ and $\sigma_i=1$ if $i\notin B$.

$\bullet$
$\widetilde{\c}(\Phi_{\rm flip}^B)\subset F$ is the semi-closed cone
$$
\widetilde{\c}(\Phi_{\rm flip}^B)=\{\sum_i x_i \phi_i , x\in Q_{\rm neg}^B\}
$$

and

  $$ \widetilde{\c}_\Z(\Phi_{\rm flip}^B)=\{\sum_i x_i \phi_i, x\in
Q_{\rm neg}^B\cap \Z^N\}.
$$
\end{definition}

With this notation, the standard quadrant  is $$Q=Q_{\rm neg}^{\emptyset}.$$

Remark that the closure of the semi-closed cone
$\widetilde{\c}(\Phi_{\rm flip}^B)$ is the closed cone $\c(\Phi_{\rm flip}^B)$.

We recall the following lemma.

\begin{lemma}\label{lemma:salient }
The following conditions are equivalent:

(i) The cone $\c(\Phi_{\rm flip}^B)$ is salient

(ii) $\overline Q_{\rm neg}^B\cap V=\{0\}$

(iii) For any $\lambda\in F$, $V(\Phi,\lambda)\cap \overline Q_{\rm neg}^B $ is bounded.

\end{lemma}

\begin{lemma}\label{easy}
Let $\tau\subset \c(\Phi)$ be a tope and $\overline \tau$ its closure.
Let $B$ be a subset of $\{1,2,\ldots, N\}$.
Assume that  the semi-open cone $\widetilde{\c}(\Phi_{\rm flip}^B)$ and the
tope $\tau$ are disjoint.
Then

\noindent(i)
$\tau$
is disjoint from the closed cone  $\c(\Phi_{\rm flip}^B)$.

\noindent(ii)
The closure $\overline{\tau}$ of $\tau$
is disjoint from the semi-open cone $\widetilde{\c}(\Phi_{\rm flip}^B)$.

\noindent (iii)  ${\tau}-\b(\Phi)$ is disjoint from
$\widetilde{\c}_\Z(\Phi_{\rm flip}^B)$.

\end{lemma}

\begin{proof}
(i) Assume that   the semi-open cone $\widetilde{\c}(\Phi_{\rm flip}^B)$ and the
tope $\tau$ are disjoint. As $\tau$ is open, it is disjoint from the closure
${\c}(\Phi_{\rm flip}^B)$ of $\widetilde{\c}(\Phi_{\rm flip}^B)$.

(ii)
Choose  $z$ small in $\tau$. As $\tau\subset \c(\Phi)$,
we can write $z=\sum_{a\in A} \epsilon_a \phi_a$ with $A$ a subset of $\{1,2,\ldots, N\}$ and $\epsilon_a>0$.
As $\tau$ is a cone, we may assume the $\epsilon_a$ very small.
Let $\lambda\in \overline{\tau}$.
Then $\lambda+z\in \tau$.
Now, if $\lambda$ belongs also to $  \widetilde{\c}(\Phi_{\rm flip}^B)$,
we may write $\lambda= \sum_{i=1}^N x_i\phi_i$
with $x_i<0$ if $i\in B$ and $x_i\geq 0$ if $i\in B^c$
and we see that $\lambda+z$ is still in
$\widetilde{\c}(\Phi_{\rm flip}^B)$  if $\epsilon_a$ are sufficiently small. This
contradicts the fact that $\widetilde{\c}(\Phi_{\rm flip}^B)\cap \tau$ is
empty.  So (ii) is proven.

Let us prove (iii). Assume that there exist $(n_i)\in \Z^N$,
with $n_i<0$ for $i\in B$ and $n_i\geq 0$ for $i\notin B$,
such that $\sum_i n_i \phi_i\in {\tau}-\b(\Phi)$.  Thus there exist $(t_i)$
with $0\leq t_i\leq 1$, for $i=1,\dots,N$ , and $\lambda \in {\tau}$,
such that $\sum_i (n_i+t_i) \phi_i  =\lambda$.
As $n_i$ are integers, we have $n_i\leq -1$ hence $n_i+t_i\leq 0$ for $i\in B$.
We have also $n_i+t_i\geq 0$ for $i\notin B$.
It follows that $\lambda\in {\tau}\cap {\c}(\Phi_{\rm flip}^B)$.
This contradicts (ii).
\end{proof}
The following proposition states continuity properties on closures and beyond.
\begin{proposition}\label{th:quadrants} Let $\tau\subset \c(\Phi)$  be a tope and
let $Z=\sum_{B} z_B w_B \in W $  be such that
\begin{equation}\label{Z}
\sum_B z_B
\suppresschi{[Q_{\rm neg }^B]}\;\suppresschi{[V(\Phi,\lambda)]}=\suppresschi{[\p(\Phi,\lambda)]} \mbox{
for every } \lambda\in \tau.
\end{equation}
Then

\noindent (i) $z_\emptyset=1$.

\noindent (ii) The equation
$$\sum_B z_B
\suppresschi{[Q_{\rm neg }^B]}\;\suppresschi{[V(\Phi,\lambda)]}=\suppresschi{[\p(\Phi,\lambda)]}$$ still holds
for every  $\lambda\in \overline{\tau}$.

\noindent (ii) For $\lambda\in \tau-\b(\Phi)$, we have
$$
\sum_B z_B
\suppresschi{[Q_{\rm neg}^B]}
\suppresschi{[V(\Phi,\lambda)\cap \Z^N]}=\suppresschi{[\p(\Phi,\lambda)\cap \Z^N]}.
$$
\end{proposition}

\begin{proof}
 Let $\lambda\in\tau$ and
$x\in \p(\Phi,\lambda)$. Then $x\in Q=Q_{\rm neg}^{\emptyset}$. As the quadrants
$Q_{\rm neg}^B$ are pairwise disjoint, $x\notin Q_{\rm neg}^B$ for $B\neq \emptyset$
hence (\ref{Z}) implies (i).

Next, let  $B\neq \emptyset$. Let $\lambda\in\tau$. Assume there is
an  $x\in Q_{\rm neg}^B\cap V(\Phi,\lambda)$.
   Then $x\notin
\p(\Phi,\lambda)$, thus  (\ref{Z}) implies that $z_B=0$. Hence, if
$z_B\neq 0$,  the semi-open cone $\widetilde{\c}(\Phi_{\rm neg}^B)$ and the
tope $\tau$ are disjoint.
We can then apply Lemma \ref{easy}. As
$\overline{\tau}$
is disjoint from $\widetilde{\c}(\Phi_{\rm neg}^B)$, we see that $V(\Phi,\lambda)$ does not intersect any of the $Q_{\rm neg}^B$ with $z_B\neq 0$ and $Q_{\rm neg}^B$ different of $Q$.  This implies (ii).
  In the same way, we obtain (iii).
\end{proof}
\begin{example}[See Fig.\ref{fig:intervalle_3_sommets}] \label{ex:intervalle_3_sommets_1} Let $N=3$, $\dim F=2$,
$\Phi=(\phi_1, \phi_2, \phi_3= \phi_1+\phi_2)$. If $\tau_1$ is the
open cone generated by $(\phi_1,\phi_3)$, we have
$X(\Phi,\tau_1)=(p_1+q_1)(p_2+q_2)p_3+ (p_1+q_1)p_2
-(p_1+q_1)(p_2+q_2)(p_3+q_3)=  p_1p_2p_3 -  p_1q_2q_3 + q_1p_2p_3
-q_1q_2q_3$. We can check that $X(\phi,\tau_1)$ satisfies the
properties (ii) and (iii) of Proposition \ref{th:quadrants} on Fig. \ref{fig:intervalle_3_sommets}.
\end{example}
\begin{figure}[!h]
\begin{center}
  \includegraphics[width=2in]{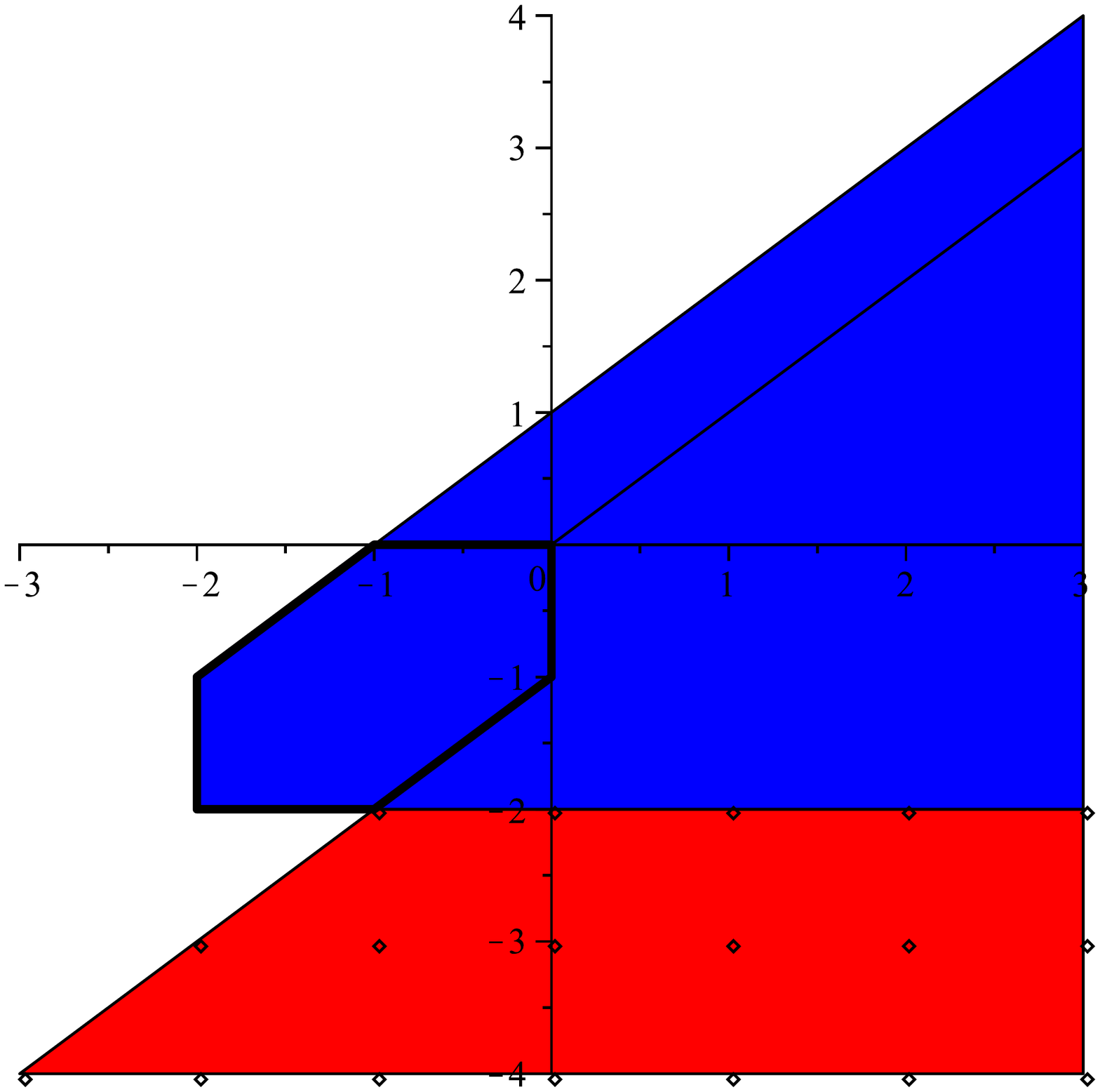}
  \caption{(Example \ref{ex:intervalle_3_sommets_1}).
  In blue, $-\b(\Phi)$ is the closed hexagon, $\tau$ and  ${\tau}-\b(\Phi)$  are open sets. For $B=(2,3)$,
  $\widetilde{\c}_Z(\Phi_{\rm neg}^B)$ is the set of lattice points in the red
  zone,   $m\phi_1-n\phi_3$ with $m\geq 1$ and $n\geq 2$.}
  \label{fig:intervalle_3_sommets}
\end{center}
\end{figure}

\begin{corollary}\label{th:continuity-on-closed-tope} (Continuity on the closure of a tope $\tau$.)
Let $\Phi=(\phi_j)_{1\leq j\leq N} $ be a sequence of non zero elements
of a vector space $F$, generating F, and  spanning a salient cone and  let $\tau\subset \c(\Phi)$ be a tope
relative to $\Phi$ .

\noindent(i) If $\lambda$ belongs to the closure $\overline{\tau}$
of the tope $\tau$, then we  have the equality of characteristic functions of sets
$$  \varchenko(\Phi,\tau)\suppresschi{[V(\Phi,\lambda)]}=\suppresschi{[\p(\Phi,\lambda)]}.$$

\noindent(ii) If $\lambda\in \tau-\b(\Phi)$, we still have the equality of characteristic functions of  sets of lattice points
$$
\varchenko(\Phi,\tau)\suppresschi{[V(\Phi,\lambda)]}\,\suppresschi{[\Z^N]}
=\suppresschi{[\p(\Phi,\lambda)\cap \Z^N]}.
$$
\end{corollary}
\begin{remark}
There are other elements $Z=\sum z_B w_B$ which satisfy (\ref{Z}).
The simplest one is $Z=w_\emptyset=p_1\cdots p_N$. However it does
not enjoy the analytic  properties of  $\varchenko(\Phi,\tau)$, see
Theorem \ref{th:polynom_integral_and_discrete}.
\end{remark}

\subsection{Polarized sums}\label{polarized_sums}

We now introduce another function on $\R^N$ related to the polarized decomposition of a polytope
as a signed sums of polarized semi-closed cones at the vertices.

In addition to the data of the previous section, we use here a
linear form $\beta $ on $\R^N$, regular with respect to $\Phi$ in
the following sense.

Let $I$ be a basic subset for $\Phi$,
and recall the description of the cone $\a_0(I)$ given in Lemma \ref{lemma:later} with generators $g_j^I=\rho_{\Phi,I}(e_j)$:

$$\a_0(I)=\sum_{j\in I^c}\R_{\geq 0} g_j^I.$$

We assume that $\beta$ is such that
its restriction to $V$ does not vanish on any edge $g_j^I$ of the simplicial cones $\a_0(I)$ when $I$ varies in $\mathcal B(\Phi).$
That is
$\langle\beta, \rho_{\Phi,I} (e_j) \rangle\neq 0$ for all
$I\in \CB(\Phi)$ and $j\in I^c$.

We associate to $\beta$ the ``polarized" cone
 $$
 \a_0(I,\beta)=\sum_{j; \langle\beta, g_j^I \rangle>0} \R_{\geq 0}g_j^I +
 \sum_{j; \langle\beta, g_j^I \rangle<0} \R_{< 0}g_j^I .
 $$
 This is the cone obtained by reversing   the direction of some of the generators of the simplicial cone $\a_0(I)$, in order that $\beta$ takes positive value on all of them. Note however the delicate condition on signs.

Now all the cones  $\a_0(I,\beta)$ are contained in the half space of $V$ determined by $\beta\geq 0$.

For each $K\in \CB(\Phi)$, we denote by
${K^c_\beta }^+$ (resp. ${K^c_\beta}^-$) the set of $j\in K^c$
such
that $\langle\beta, \rho_{\Phi,K} e_j \rangle > 0$ ( resp. ²
$\langle\beta, \rho_{\Phi,K} e_j \rangle < 0$).

\begin{definition}\label{def:coneIbeta}
If $K$ is a subset of $\{1,2,\ldots, N\}$, we denote by

 $$
\a(K,\beta)=\{ x\in \R^N;
x_i\geq 0 \mbox{ for } i\in {K_\beta^c}^+,
x_i< 0 \mbox{ for } i\in {K_\beta^c}^- \}.$$

\end{definition}

Thus the set $\a(K,\beta)$ is the product of three terms: the closed quadrant in the variables in ${K_\beta^c}^+$, 
, the opposite of the open quadrant in the variables in ${K_\beta^c}^-$, and a vector space in the variable in $K$.

If  $K\in \CB(\Phi)$, the cone $\a(K,\beta)\cap V(\Phi,\lambda)$ 
is the translate by the vertex $s_K(\Phi,\lambda)$ of the semi-open cone $\a_0(K,\beta)$ of dimension $d$.
In particular  $\a(K,\beta)\cap V(\Phi,\lambda)$ is contained 
in the half space $s_K(\Phi,\lambda)+{\beta\geq 0}\cap V(\Phi,\lambda)$ of $V(\Phi,\lambda).$

If $\lambda\in \tau$ and $K\in \CB(\Phi,\tau)$,
the  cone   $\a(K,\beta)\cap V(\Phi,\lambda)$   is obtained 
by reversing some of the generators of the tangent cone to the polytope $\p(\Phi,\lambda)$ 
at the vertex $s_K(\Phi, \lambda)$   
so that $\beta$ takes positive values on them. We say that it is the polarized tangent cone.

Recall  the Lawrence-Varchenko polarized decomposition of $\p(\Phi,\lambda)$, (actually, we will give a proof below).
$$
\suppresschi{[\p(\Phi,\lambda)]}=\sum_{K\in \mathcal B(\Phi,\tau)} (-1)^{|{K_\beta^c}^-|}\suppresschi{[\a(K,\beta)\cap V(\Phi,\lambda)]}.
$$
Again this equality is obtained by intersecting a number of fixed cones in $\R^N$
with the varying affine subspace $V(\Phi,\lambda)$. Therefore it is  natural to define the following function.

\begin{definition}\label{Y} The Combinatorial Lawrence-Varchenko
function is the following element of $W$:
\begin{multline}Y(\Phi,\tau,\beta)=
\sum_{K\in \CB(\Phi,\tau)}(-1)^{|{K^c_\beta}^-|}\prod_{i\in
{K^c_\beta}^+}p_i \prod_{i\in  {K^c_\beta}^-}q_i \prod_{i\in K
}(p_i+q_i).
\end{multline}
\end{definition}

If the tope $\tau$ is not contained in $\c(\Phi)$, then $Y(\Phi,\tau,\beta)=0$.
Otherwise, if $\tau\subset \c(\Phi)$, the sum  defining $Y(\Phi,\tau,\beta)$ is
 indexed by all the vertices of  the polytope $\p(\Phi,\lambda_0)$ (for any choice of $\lambda_0\in \tau$).

If we replace $p_i$ by the characteristic function of $x_i\geq 0$ and $q_i$ by the characteristic function of $x_i<0$ ($p_i+q_i=1$), we obtain a function $Y_{{\rm geom}}(\Phi,\tau,\beta)$ on $\R^N$.

 By construction,
 if $\lambda\in\tau$, the product
$\suppresschi{[V(\Phi,\lambda)]}Y_{{\rm geom}}(\Phi,\tau,\beta)$ is the signed
sum of polarized semi-closed cones at the vertices of the (simple)
partition polytope $\p(\Phi,\lambda)$.
  Lawrence-Varchenko's  theorem  can be restated as
  $\suppresschi{[V(\Phi,\lambda)]}Y_{{\rm geom}}(\Phi,\tau,\beta)=\suppresschi{[\p(\Phi,\lambda)]}$
  while Brianchon-Gram's theorem is
$\suppresschi{[V(\Phi,\lambda)]}X_{{\rm geom}}(\Phi,\tau)=\suppresschi{[\p(\Phi,\lambda)]}$
for any $\lambda\in \tau$.

The Lawrence-Varchenko decomposition of a simple polytope can be
derived from the Brianchon-Gram one, by grouping some faces with a
common vertex, \cite{lawrence91}.
It is remarkable that their
combinatorial precursors  actually  {coincide} as elements of the
space $W$, as we show in the next theorem.

\begin{theorem}\label{th:polarized-sum}
 Let $\Phi=(\phi_j)_{1\leq j\leq N} $ be a sequence of non zero elements
of a vector space $F$, generating F, and  spanning a salient cone,  and let $\tau\subset \c(\Phi)$ be a $\Phi$-tope. Let
$X(\Phi,\tau)$ be the Combinatorial Brianchon-Gram function. For any
linear form $\beta$  which is regular with respect to $\Phi$,
let  $Y(\Phi,\tau,\beta)$ be the Combinatorial Lawrence-Varchenko
function. Then
$$ Y(\Phi,\tau,\beta)= X(\Phi,\tau).
$$
\end{theorem}
\begin{proof}
In the sum $X(\Phi,\tau)$, for a given $K\in \CB(\Phi,\tau)$, we
group together the  $I\in \CG(\Phi,\tau)$  such that $ K\subseteq I$
and   $ \langle\beta,\rho_{\Phi,K} e_i \rangle <0, \mbox{ for every
} i\in I\setminus K$. We denote the set of these $I$ by
$\CG(\Phi,\tau)_\beta^K$.

\begin{lemma} Let $I\in \CG(\Phi,\tau)$ and $K\in \CB(\Phi,\tau)$ such that $ K\subseteq I$ . Then
$I\in\CG(\Phi,\tau)_\beta^K$  if and only if
 for any $\lambda\in \tau$,  on the face $\f_I(\Phi,\lambda)$  of
$\p(\Phi,\lambda)$ which is indexed by $I$, the linear form $\beta$
reaches its maximum at the vertex $s_K$ indexed by $K$.
\end{lemma}

\begin{proof}
Let $\lambda\in\tau$.    Let $x=\sum_{i\in I}x_i e_i \in \f_I(\Phi,\lambda)$, with
$x\neq s_K$. Then $x-s_K$ is the projection $\rho_{\Phi,K}(x)=
\sum_{i\in I\setminus K}x_i \rho_{\Phi,K}(e_i) $. Hence
$$
\left\langle \beta, x-s_K \right\rangle =\sum_{i\in I\setminus K}x_i
\langle \beta,\rho_{\Phi,K}(e_i)\rangle .
$$
Assume that $ \langle\beta,\rho_{\Phi,K} e_i \rangle <0, \mbox{ for
every } i\in I\setminus K$. As $x_i\geq 0$ and $x_i>0$ for at least
one index $i\in I\setminus K$, we have $\left\langle \beta, x-s_K
\right\rangle < 0$.

Conversely, let $i\in I\setminus K$. Take an $x=\sum_{k\in K}x_k e_k + x_i e_i \in \f_I(\Phi,\lambda)$
with $x_i>0$. Then by assumption we have $\left\langle \beta, x-s_K \right\rangle<0$, hence
 $x_i  \langle\beta,\rho_{\Phi,K} e_i \rangle <0$.
\end{proof}
It follows that, when $K$ runs over  $\CB(\Phi,\tau)$ (the set of
vertices of $\p(\Phi,\lambda)$), the subsets
$\CG(\Phi,\tau)_\beta^K$ form a partition of $\CG(\Phi,\tau)$ (the
set of faces). Therefore, in order to prove Theorem
\ref{th:polarized-sum}, there remains to show that for every $K\in
\CB(\Phi,\tau)$, we have
\begin{multline}
\sum_{I\in \CG(\Phi,\tau)_\beta^K}(-1)^{|I|-\dim
F}\prod_{i\in I^c}p_i \prod_{i\in I}(p_i+q_i) =\\
 (-1)^{|{K^c_\beta}^-|}\prod_{i\in
{K^c_\beta}^+}p_i \prod_{i\in  {K^c_\beta}^-}q_i \prod_{i\in K
}(p_i+q_i).
\end{multline}
We factor out $\prod_{i\in K
}(p_i+q_i)$. We need to prove
\begin{multline}\label{eq:check_combinatorial_equality}
\sum_{I\in \CG(\Phi,\tau)_\beta^K}(-1)^{|I|-\dim F}\prod_{i\in
I^c}p_i \prod_{i\in I\setminus K}(p_i+q_i) =\\
 (-1)^{|{K^c_\beta}^-|}\prod_{i\in
{K^c_\beta}^+}p_i \prod_{i\in  {K^c_\beta}^-}q_i.
\end{multline}
We make several observations. First, $\CG(\Phi,\tau)_\beta^K $ is
precisely  the set of $I\subseteq \{1,\dots,N\}$ such that
$K\subseteq I$ and $I\setminus K \subseteq {K^c_\beta}^-$. Moreover,
for each $I\in \CG(\Phi,\tau)_\beta^K$, the set of indices $I^c
\bigsqcup (I\setminus K)$ is exactly the complement $K^c$ and  $|I|-
\dim F= |I\setminus K|$,  as  $\dim F= |K|$. Let $B\subseteq K^c$. A
given monomial $ \prod_{i\in K^c\setminus B}p_i \prod_{i\in B}q_i$
appears on the left hand side of
(\ref{eq:check_combinatorial_equality}) with coefficient
\begin{equation*}\label{eq:moebius}
\sum_{\{I\in \CG(\Phi,\tau)_\beta^K, \, B\subseteq I\setminus K
\}}(-1)^{|I\setminus K|}.
\end{equation*}
By the usual inclusion-exclusion relations applied to the subsets
$I\setminus K$ of $K^c$, this sum is equal to
$(-1)^{|{K^c_\beta}^-|}$ if $B= {K^c_\beta}^-$ and to $0$ otherwise.
\end{proof}
\begin{example}
In Example \ref{standard_simplex}  of the standard knapsack  with
$N=3$, we take $\langle\beta,x\rangle= x_1 + \frac{x_2}{2} +\frac{x_3}{3}$. We obtain $Y(\Phi,\tau,\beta)=
(p_1+q_1)q_2q_3+p_1(p_2+q_2)p_3-p_1q_2(p_3+q_3)$. It is indeed equal
to $X(\Phi,\tau)= p_1p_2p_3+q_1q_2q_3$.
\end{example}

\begin{example}\label{ex:intervalle_3_sommets_2}(continues Example
\ref{ex:intervalle_3_sommets_1}). The subspace $V$ is
generated by the vector $e_1+e_2-e_3$. Hence, the projections
$\rho_{(\Phi,K)} (e_i)$ are $\rho_{\Phi,(1,2)} (e_3)= e_3-e_1-e_2$,
$\rho_{\Phi,(1,3)} (e_2)= \rho_{\Phi,(2,3)} (e_1)= e_1+e_2-e_3$. We
can take $\langle\beta,x\rangle= x_1+ x_2+ x_3$. Let  $\tau_1$ be
the cone generated by $(\phi_1,\phi_3)$. We obtain
\begin{eqnarray}\label{eq:intervalle_3_sommets_2}
Y(\Phi,\tau_1,\beta)& =& -(p_1+q_1)(p_2+q_2)q_3 +
(p_1+q_1)p_2(p_3+q_3)\\
& =& p_1p_2p_3 + p_1q_2p_3 - q_1p_2q_3 -q_1q_2q_3.
\end{eqnarray}
Comparing with Example \ref{ex:intervalle_3_sommets_1}, we check
that $Y(\Phi,\tau_1,\beta)= X(\Phi,\tau_1) $.
\end{example}
\begin{corollary}\label{co:bounded}
For any $\lambda\in F$,
$\varchenko(\Phi,\tau)\suppresschi{[V(\Phi,\lambda)]}$ is
as a signed sum of bounded polytopes.
\end{corollary}
\begin{proof}
Choose any $\beta$ regular, then the function
$\varchenko(\Phi,\tau)\suppresschi{[V(\Phi,\lambda)]}$
is equal to $Y_{{\rm geom}}(\Phi,\tau,\beta)\suppresschi{[V(\Phi,\lambda)]}$. 
Taking $m(\beta)$ to be the minimum of the values 
$\langle \beta,s_K(\Phi,\lambda)\rangle$ over the $K\in \mathcal B(\Phi,\lambda)$, 
we see that the support of the function
$\varchenko(\Phi,\tau)\suppresschi{[V(\Phi,\lambda)]}$  
is contained in the half space $\{\langle \beta,x\rangle\geq m(\beta)\}$ of $V(\Phi,\lambda)$.
As this equality holds for any regular linear form $\beta$, this implies that the support  of
$\varchenko(\Phi,\tau)\suppresschi{[V(\Phi,\lambda)]}$ is bounded.
\end{proof}

\section{Wall-crossing}
\subsection{Combinatorial wall-crossing}
In this section  we prove the main theorem of this article: a formula for  $X(\Phi,\tau_2)- X(\Phi,
\tau_1)$, when  $\tau_1$ and $\tau_2$ are  adjacent topes.

 The
computation  comes out nicely when we use the polarized expression
$Y(\Phi,\tau,\beta)$ as a sum over $\CB(\Phi,\tau)$ (Theorem
\ref{th:polarized-sum}), because it is easy to  analyze how
$\CB(\Phi,\tau)$ changes as we cross the wall $H$ between the two
topes.

We recall that two topes $\tau_1$ and $\tau_2$ are called adjacent if the intersection of their closures spans a wall $H$.
We denote by $\Phi\cap H$ the subsequence of $\Phi$ formed by the elements $\phi_i$ belonging to $H$.

\begin{lemma}\label{flip} Let $\tau_1$ and $\tau_2$  be adjacent   $\Phi $-topes
 such that $\tau_1\subset \c(\Phi)$.
 Let $K\in \CB(\Phi,\tau_1)$ such that $K\notin \CB(\Phi,\tau_2)$.

\noindent(i) For all $k\in K$ but one, say $k_1$, we have $\phi_k\in
H$. The vector  $\phi_{k_1}$ is in the open side  of $ H $ which
contains $\tau_1$.

\noindent(ii) Let $\tau_{12}$ be the unique tope of $\Phi\cap H$
such that $\overline{\tau_1}\cap\overline{\tau_2}\subset
\overline{\tau_{12}}$. Then $\tau_{12}$ is contained in the cone
generated by  the vectors $\phi_k$ for $ k\in K, k\neq k_1 $.
\end{lemma}
\begin{proof}
Up to renumbering, we can assume that  $K=\{1,\dots,r\}$.
Let $x_1,\dots,x_r$ be the  coordinates on $F$ relative to this basic subset.
If $K\notin \CB(\Phi,\tau_2)$, at least one of these coordinates,
 say $x_1$, is $<0$ on $\tau_2$. Then the wall $H$ must be the hyperplane
$ \{x_1=0\}$.  (i) follows immediately.

The proof of (ii) is also easy.

\end{proof}

Recall Definition \ref{def:flipped_Phi} of flipped systems $\Phi_{\rm flip}^A$.

\begin{lemma}\label{bases_topes_adjacents}
Let $\tau_1$ and $\tau_2$  be adjacent    $\Phi $-topes such that
$\tau_1\subset \c(\Phi)$. Let $H$ be their common wall. Let $A $ be
the set of $i\in \{1,\dots,N\}$ such that $ \phi_i$ belongs to  the
open side of $ H $   which contains $\tau_1$, (hence $-\phi_i$
belongs to the side of $\tau_2$). Then $\CB(\Phi_{\rm flip}^A,\tau_2)$ is equal
to the symmetric difference $\CB(\Phi,\tau_1)\vartriangle
\CB(\Phi,\tau_2)$. More precisely
$$
K\in \CB(\Phi,\tau_1), K\notin \CB(\Phi,\tau_2)  \Leftrightarrow
K\in\CB(\Phi_{\rm flip}^A,\tau_2), K\cap A\neq \emptyset,
$$
$$K\in
\CB(\Phi,\tau_2), K\notin \CB(\Phi,\tau_1)\Leftrightarrow
K\in\CB(\Phi_{\rm flip}^A,\tau_2), K\cap A=\emptyset.
$$
Moreover, the cone $\c(\Phi_{\rm flip}^A)$ is salient , and $\tau_2$ is
contained in at least one of the cones $\c(\Phi)$ or $\c(\Phi_{\rm flip}^A)$.
\end{lemma}
\begin{proof}
It follows easily from Lemma \ref{flip} (i) and the definition of
$\Phi_{\rm flip}^A$.
\end{proof}

It will be convenient to have a
notation.
\begin{definition}
Let $\tau_1, \tau_2$ be adjacent $\Phi$-topes. We denote by
$A(\Phi,\tau_1,\tau_2)$ the set of $i\in\{1,\dots,N\}$ such that
$\phi_i$ belongs to the open side of the common wall which contains
 $\tau_1$ .
\end{definition}

\begin{theorem}\label{th:wall_crossing}

Let $\Phi=(\phi_j)_{1\leq j\leq N} $ be a sequence of non zero elements
of a vector space $F$, generating F, and  spanning a salient cone.
 Let $\tau_1$ and $\tau_2$ be adjacent $\Phi $-topes
such that $\tau_1\subset \c(\Phi)$.
 Let $H$ be their common wall.
Let $A $ be the set of $i\in \{1,\dots,N\}$ such that $ \phi_i$ is
in the open side  of $ H $ which contains $\tau_1$. Let $\Phi_{\rm flip}^A$ be
the sequence $\sigma^A_i \phi_i$, where $\sigma^A_i=-1$ if $i\in A$
and $\sigma^A_i=1$ if $i\notin A$.  Let $X(\Phi,\tau_1)$,
$X(\Phi,\tau_2)$ and $X(\Phi_{\rm flip}^A,\tau_2)$ be the corresponding
Combinatorial Brianchon-Gram polynomials $\in \Z[p_i,q_i]$. Let
$\flip_A $ be the ring homomorphism  from $\Z[p_i,q_i]$ to itself
defined by exchanging $p_i$ and $q_i$ for $i\in A$.  Then we have
the wall-crossing formula
\begin{equation}\label{eq:wall-crossing}
X(\Phi,\tau_1)= X(\Phi,  \tau_2)- (-1)^{|A|}\flip_A X(\Phi_{\rm flip}^A,
\tau_2).
\end{equation}
\end{theorem}
\begin{remark}
If  the tope  $\tau_2$ is not contained in $\c(\Phi)$, (res
$\c(\Phi_{\rm flip}^A)$),  then $\CG(\Phi, \tau_2)$, (resp. $\CG(\Phi_{\rm flip}^A,
\tau_2)$), is empty, hence $X (\Phi, \tau_2)=0$, (resp $X (\Phi_{\rm flip}^A,
\tau_2)=0$).
\end{remark}

\begin{remark}\label{trivial_wall_crossing}
By Remark \ref{topes and chambers}, the actual jumps occur only on walls between chambers. This in agreement with this formula: indeed if $H$ is not a wall between chambers, the tope $\tau_{12}$ is not contained in $\c(\Phi\cap H)$ and $\tau_2$ is not contained in $\c(\Phi_{\rm flip}^A)$.
\end{remark}

\begin{remark}

The theorem is trivially true if $\tau_1\nsubseteq \c(\Phi)$.
Indeed, in this case, we have $X(\phi,\tau_1)=0$, and
$A(\Phi,\tau_1,\tau_2)=\emptyset$, so that the right hand side of
(\ref{eq:wall-crossing}) is $X(\phi,\tau_2)-X(\phi,\tau_2)=0$.
\end{remark}
\begin{proof}[Proof of Theorem \ref{th:wall_crossing}]
 Let $\beta$ be a linear form on $\R^N$ which is regular for $\Phi$.
 Let $\beta^A$ be the linear form defined by
$$
\langle\beta^A, e_i\rangle= \sigma^A_i \langle\beta, e_i\rangle
$$
where $\sigma^A_i=-1$ if $i\in A$ and $\sigma^A_i=1$ if $i\in A^c$.
We have, for every $i$,
\begin{equation}\label{eq:flipped_beta}
\langle\beta^A, \rho_{\Phi_{\rm flip}^A,K}e_i\rangle= \sigma^A_i \langle\beta,
\rho_{\Phi,K}e_i\rangle.
\end{equation}
It follows, in particular, that  $\beta^A$ is regular for $\Phi_{\rm flip}^A$.

First we will prove the following relation between the polarized
sums
\begin{equation}\label{eq:sautY}
 Y(\Phi,\tau_2,\beta)- Y (\Phi, \tau_1,\beta)=(-1)^{|A|}\flip_A Y(\Phi_{\rm flip}^A,\tau_2,\beta^A).
\end{equation}
Then we obtain (\ref{eq:wall-crossing})  by applying  Theorem
\ref{polarized_sums}. We write
\begin{multline}
Y(\Phi,\tau_2,\beta)- Y (\Phi, \tau_1,\beta)= \\ \sum_{K\in
\CB(\Phi,\tau_2)}(-1)^{|{K^c_\beta}^-|}\prod_{i\in {K^c_\beta}^+
}p_i \prod_{i\in {K^c_\beta}^- }q_i\prod_{i\in K}(p_i+q_i)\\
 - \sum_{K\in
\CB(\Phi,\tau_1)}(-1)^{|{K^c_\beta}^-|}\prod_{i\in {K^c_\beta}^+
}p_i \prod_{i\in {K^c_\beta}^- }q_i\prod_{i\in K}(p_i+q_i)
\end{multline}
The terms for which  $K\in \CB(\Phi,\tau_1)\cap \CB(\Phi,\tau_2)$
cancel out.  For the other terms, we apply  Lemma
\ref{bases_topes_adjacents}.  Take a $K$ in  $
\CB(\Phi,\tau_1)\vartriangle \CB(\Phi,\tau_2)= \CB(\Phi_{\rm flip}^A,\tau_2)$.
Using (\ref{eq:flipped_beta}),
 we check easily that the unique $K$-term in
$Y(\Phi,\tau_2,\beta)-Y(\Phi,\tau_1,\beta) $ is
 equal to the $K$-term in \\
 $(-1)^{|A|}\flip_A Y(\Phi_{\rm flip}^A,
\tau_2,\beta^A)$.
\end{proof}

\begin{example}\label{ex:zonotopeB2_suite}
$N=4$ and
$\Phi=(\phi_1,\phi_2,\phi_3=\frac{1}{2}(\phi_2-\phi_1),\phi_4=\frac{1}{2}(\phi_1+\phi_2))$.
The tope  $\tau_1$ is  the open cone generated by $\phi_2$ and
$\phi_4$. The adjacent tope  $\tau_2$ is the open cone generated by
$\phi_4$ and $\phi_1$, see Fig. \ref{fig:B2topes}.
Then $\phi_2$ and $\phi_3$ lie on the
$\tau_1$-side of the common wall, so that $A=\{2,3\}$ and $\Phi_{\rm flip}^A=
(\phi_1,-\phi_2,-\phi_3,\phi_4)$. We obtain
\begin{eqnarray*}
X(\Phi,\tau_1)&=&
p_{{1}}p_{{2}}p_{{3}}p_{{4}}-p_{{1}}q_{{2}}q_{{3}}p_{{4}}-q_{{1}}p_{{2}}p_{{3
}}q_{{4}}+q_{{1}}q_{{2}}q_{{3}}q_{{4}}, \\
 X(\Phi,\tau_2) &=& p_{{1}}p_{{2}}p_{{3}}p_{{4}}+p_{{1}}q_{{2}}q_{{3}}q_{{4}}+q_{{1}}p_{{2}}p_{{3
}}p_{{4}}+q_{{1}}q_{{2}}q_{{3}}q_{{4}}, \\
X(\Phi_{\rm flip}^A,\tau_2)&=&
p_{{1}}p_{{2}}p_{{3}}p_{{4}}+p_{{1}}p_{{2}}p_{{3}}q_{{4}}+q_{{1
}}q_{{2}}q_{{3}}p_{{4}}+q_{{1}}q_{{2}}q_{{3}}q_{{4}}, \\
 \flip_A X(\Phi_{\rm flip}^A, \tau_2) &=&
p_{{1}}q_{{2}}q_{{3}}p_{{4}}+p_{{1}}q_{{2}}q_{{3}}q_{{4}}+q_{{1}}p_{{2}}p_{{3
}}p_{{4}}+q_{{1}}p_{{2}}p_{{3}}q_{{4}}, \\
 &=& X(\Phi,\tau_2)-X(\Phi,\tau_1) .
\end{eqnarray*}
\end{example}
\begin{figure}[!h]
\begin{center}
  \includegraphics[width=1.5 in]{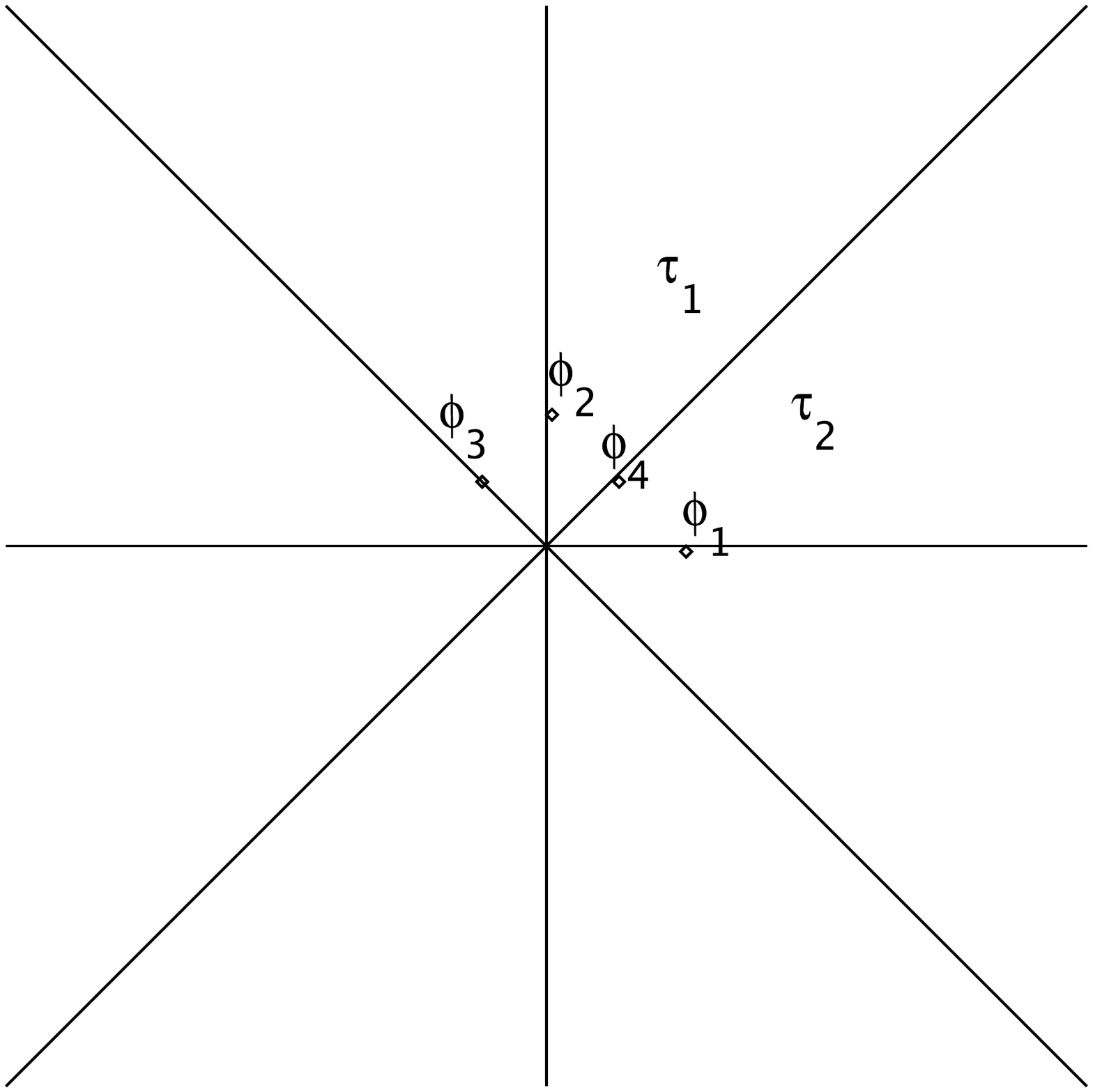}
  \caption{Topes $\tau_1$,$\tau_2$  for $\Phi=(\phi_1,\phi_2,\phi_3=\frac{1}{2}(\phi_2-\phi_1),\phi_4=\frac{1}{2}(\phi_1+\phi_2))$.}
  \label{fig:B2topes}
  \end{center}
\end{figure}
Later, in order to give a formula for the decomposition of $\varchenko(\Phi,\tau)$ in quadrants, we
will need to iterate the wall-crossing through a sequence of
consecutively adjacent topes. To help understand what we obtain, let us first cross two
consecutive walls.
\begin{corollary}\label{th:wallcrossing2steps}
Let  $\tau_1, \tau_2,\tau_3$ be pairwise different $\Phi$-topes such
that $\tau_2$ is adjacent to  $\tau_1$ and to $\tau_3$. Let
$A_1=A(\Phi,\tau_1,\tau_2)$,  $A_2=A(\Phi,\tau_2,\tau_3)$ and let
 $A_{1,2}$ be the symmetric difference $A_1\vartriangle
 A(\Phi_{\rm flip}^{A_1},\tau_2,\tau_3)$. Then the  three cones 
 $\c(\Phi_{\rm flip}^{A_1}), \c(\Phi_{\rm flip}^{A_2})$ and $\c(\Phi_{\rm flip}^{
A_{1,2}})$ are salient .

We have
\begin{multline}\label{eq:wallcrossing2steps}
  X(\Phi,\tau_1)=X(\Phi,\tau_3)-(-1)^{|A_1|}\flip_{A_1}X(\Phi_{\rm flip}^{A_1},\tau_3)\\
  -(-1)^{|A_2|}\flip_{A_2}X(\Phi_{\rm flip}^{A_2},\tau_3)
  +(-1)^{|A_{1,2}|}\flip_{A_{1,2}}X(\Phi_{\rm flip}^{ A_{1,2}},\tau_3).
\end{multline}
\end{corollary}
\begin{proof}
We apply the wall-crossing theorem  to $(\Phi,\tau_1,\tau_2)$. We
obtain
$$
 X(\Phi,\tau_1)=
 X(\Phi,\tau_2)-(-1)^{|A_1|}\flip_{A_1}X(\Phi_{\rm flip}^{A_1},\tau_2).
$$
In the right hand side,  we transform each  term by crossing the
wall from $\tau_2$ into $\tau_3$. First,
$$
 X(\Phi,\tau_2)=
 X(\Phi,\tau_3)-(-1)^{|A_2|}\flip_{A_2}X(\Phi_{\rm flip}^{A_2},\tau_3).
$$
In order to apply the wall-crossing to $(\Phi_{\rm flip}^{A_1},\tau_2,\tau_3)$,
we observe that the sign rule implies
\begin{eqnarray*}
(\Phi_{\rm flip}^{A_1})^{A(\Phi_{\rm flip}^{A_1},\tau_2,\tau_3)}&=&\Phi_{\rm flip}^{A_1\vartriangle
A(\Phi_{\rm flip}^{A_1},\tau_2,\tau_3)},\\
\flip_{A_1}\circ\flip_{A(\Phi_{\rm flip}^{A_1},\tau_2,\tau_3)}&=&\flip_{A_1\vartriangle
A(\Phi_{\rm flip}^{A_1},\tau_2,\tau_3)}.
\end{eqnarray*}
Moreover,
$(-1)^{|A_1|}(-1)^{A(\Phi_{\rm flip}^{A_1},\tau_2,\tau_3)}=(-1)^{A_1\vartriangle
A(\Phi_{\rm flip}^{A_1},\tau_2,\tau_3)}$. Hence we obtain
\begin{multline*}
-(-1)^{|A_1|}\flip_{A_1}X(\Phi_{\rm flip}^{A_1},\tau_2)=\\
-(-1)^{|A_1|}\flip_{A_1}X(\Phi_{\rm flip}^{A_1},\tau_3)+(-1)^{|A_{1,2}|}\flip_{A_{1,2}}X(\Phi_{\rm flip}^{
A_{1,2}},\tau_3).
\end{multline*}
This proves (\ref{eq:wallcrossing2steps}). The cones are salient by
the very definition of the flipped systems $\Phi_{\rm flip}^{A_i}$ and
$(\Phi_{\rm flip}^{A_1})^{A(\Phi_{\rm flip}^{A_1},\tau_2,\tau_3)}$.
\end{proof}

We now cross a number of walls to go from a tope $\tau$ to another tope $\nu$. A signed subset of $\{1,2,\ldots, N\}$ is a list $[\epsilon, I]$ where $I$ is a subset of $\{1,2,\ldots, N\}$ and $\epsilon=\pm 1$ a sign.
\begin{definition}
Let $\tau$ and $\nu$  be two topes, and let us choose a sequence
  $\tau_k$,
 $k=1,\dots,\ell$ of topes such that $\tau_{k+1}$
is adjacent to $\tau_k$ for every $1\leq k\leq \ell-1$, and $\tau_1=\tau, \tau_\ell=\nu$.

 For every
sequence $K=(1\leq k_1<\cdots<k_s\leq \ell-1)$,
 let $A_K\subseteq
\{1,\dots,N\}$ be the subset defined recursively as follows.  If
$s=0$, that is $K=\emptyset$, then $A_\emptyset=\emptyset$. If
$K=(k_1,k_2,\dots,k_s)$ with $s\geq 1$, let
$A=A_{(k_1,\dots,k_{s-1})}$, then
$$
A_K=A \vartriangle A(\Phi_{\rm flip}^A,\tau_{k_{s}},\tau_{k_{s}+1}).
$$

We define the list $\CA(\nu,\tau)$  to be the list of signed subsets  $[(-1)^{|K|}, A_K]$ of $\{1,2,\ldots, N\}$ so obtained.
\end{definition}

\begin{remark}
The list $\CA(\nu,\tau)$ depends of the choice of path of adjacent topes from $\tau$ to $\nu$ , but we do not indicate this in the notation.

\end{remark}

We obtain the following result, if there are $\ell-1$ wall crossings to go from $\tau$ to $\nu$.

\begin{corollary}\label{th:wallcrossing_k_steps}
 Let $\nu$ be a tope.  Then for every $[\epsilon,A]\in \CA(\nu,\tau)$,  the cone $\c(\Phi_{\rm flip}^{A})$ is salient. Furthermore,
we have
\begin{equation*}
 X(\Phi,\tau)=\sum_{[\epsilon,A]\in \CA(\nu,\tau)} \epsilon
 \flip_{A}X(\Phi_{\rm flip}^{A},\nu).
\end{equation*}
\end{corollary}
\begin{proof} The recursion rule means that,  when we travel through the
sequence of topes $\tau_i, i=1,\dots,\ell $, and apply Formula
(\ref{eq:wall-crossing}),  we choose the
flipped term when we cross the wall between $\tau_{k_i}$ and
$\tau_{k_i+1}$, and the unflipped term for the other walls. For
instance, $A_{\{1\}}=A(\Phi,\tau_1,\tau_{2})$, and
$A_{\{1,2\}}=A_{\{1\}}\vartriangle
A(\Phi_{\rm flip}^{A_{\{1\}}},\tau_2,\tau_3)$, in agreement with the two-step
wall-crossing formula. The general case is immediate by induction.
\end{proof}

\subsection{Semi-closed partition polytopes}
In order to state the geometric consequences of the above
combinatorial wall-crossing formulas, we introduce some semi-closed
partition polytopes, to which the Brianchon-Gram theorem extends
naturally.
\begin{definition}\label{flipped-polytope}
$$
\p(\Phi,A,\lambda)=\{x\in V(\Phi,\lambda), \;\;  x_i>0 \mbox{ for }
i\in A, x_i\geq 0 \mbox{ for } i\in A^c.\}
$$
\end{definition}
When $\lambda$ is regular, the closure of $\p(\Phi,A, \lambda)$ is the partition polytope
$\p(\Phi,\lambda)$.
\begin{definition}For $A\subseteq \{1,\dots,N\}$
we denote by $\geom_A$ the map  from $W$ to the space of functions
on $\R^N$ defined by substituting $1-p_i$ for $q_i$,  then
$\suppresschi{[x_i\geq 0]}$ for $p_i$ if $i\notin A$ and $\suppresschi{[x_i> 0]}$
for $p_i$ if $i\in A$.
 \end{definition}

When $A$ is the empty set, the substitution $\geom_{\emptyset}$ coincide with the usual substitution $\geom$ defined before.
When we consider  non empty subsets $A\subseteq \{1,\dots,N\}$, we
obtain an extension of the  Brianchon-Gram theorem to these
semi-closed polytopes.

\begin{proposition}\label{semi-closed-brianchon-gram}
Let $A\subseteq \{1,\dots,N\}$.  For  $ \lambda\in \tau$, we have
\begin{equation}
\geom_A X(\Phi,\tau)\, \suppresschi{[V(\Phi,\lambda)]}=\suppresschi{[\p(\Phi,A,
\lambda)]}.
\end{equation}
\end{proposition}
\begin{proof}
When $A=\emptyset$, it is exactly the Brianchon-Gram theorem. We
proceed by induction on the cardinality of $A$.  If $A\neq
\emptyset$, we can assume that $N\in A$, up to renumbering. Let
$A'=A\setminus \{N\}$.

We write $ \suppresschi{[\p(\Phi, A, \lambda)]}=\suppresschi{[\p(\Phi, A',\lambda
)]}+ \left(\suppresschi{[\p(\Phi,A,\lambda )]}- \suppresschi{[\p(\Phi, A',\lambda
)]}\right)$. We have
$$
\suppresschi{[\p(\Phi,A',\lambda )]}- \suppresschi{[\p(\Phi, A,\lambda
)]}=\suppresschi{[\p(\Phi, A',\lambda )]}\suppresschi{[x_N=0]}.
$$

 Let us show that we have
\begin{equation}\label{eq:recurrence_Geometric_Brianchon_Gram_function}
(\geom_{A'} X(\Phi,\tau)-\geom_A X(\Phi,\tau))\suppresschi{[V(\Phi,\lambda)]}=
\suppresschi{[\p(\Phi, A',\lambda )]}\suppresschi{[x_N=0]}.
\end{equation}
We first prove (\ref{eq:recurrence_Geometric_Brianchon_Gram_function}) in the case
where $A=\{N\}$, hence $A'=\emptyset$.

We observe that the right hand side of
(\ref{eq:recurrence_Geometric_Brianchon_Gram_function}) is the
characteristic function of the face of $\p(\Phi, \lambda )$ defined by
$x_N=0$. If we identify the hyperplane $\{x_N=0\}$ with $\R^{N-1}$,
 this face is the partition polytope  $\p(\Phi', \lambda )$ corresponding to
$\Phi'=(\phi_i), 1\leq i\leq N-1$ and $\lambda\in F$.

We now look at the left hand side of
(\ref{eq:recurrence_Geometric_Brianchon_Gram_function}).
We see that
\begin{multline} \label{eq:brianchon-gram_facet}
\geom_{\emptyset} X(\Phi,\tau)-\geom_{\{N\}} X(\Phi,\tau)=\\
\left(\sum_{\stackrel{I\in \CG(\Phi,\tau),}{  I^c\ni N}}(-1)^{|I|-\dim F}
\prod_{i\in I^c, i\neq N}\suppresschi{[x_i\geq 0]}\right)\suppresschi{[x_N=0]},
\end{multline}
because the terms indexed by the subsets $I$ such that $N\notin I^c$
cancel out in the difference.

If $\Phi'$ does not generate $\Phi$, we see that both sides of the equation
(\ref{eq:recurrence_Geometric_Brianchon_Gram_function}) are  equal to $0$.
 Indeed as $\lambda$ is regular, it cannot be contained
 in the smaller dimensional space generated by $\Phi'$,
 and every generating subset in $\CG(\Phi,\tau)$ contains the index $N$.

Now assume that $\Phi'$ generates $F$. The
$\Phi$-tope $\tau$ is contained in a unique $\Phi'$-tope $\tau'$.
The set $\CG(\Phi',\tau')$ consists  precisely of the subsets
$I'\subseteq \{1, \dots, N-1\}$ such that $I'\in \CG(\Phi,\tau)$.

Therefore the right hand side of
(\ref{eq:brianchon-gram_facet}) is the Brianchon-Gram decomposition
of the facet $\p(\Phi', \lambda )$. Thus  we have proved
 (\ref{eq:recurrence_Geometric_Brianchon_Gram_function}) in the case where $A=\{N\}$.

The general case when $A' \neq \emptyset$ is similar. We have now
\begin{multline*}
\geom_{A'} X(\Phi,\tau)-\geom_A X(\Phi,\tau)=\\
\left(\sum_{\stackrel{I\in \CG(\Phi,\tau),}{  I^c\ni N}}(-1)^{|I|-\dim F}
\prod_{i\in I^c \cap A'} \suppresschi{[x_i> 0]} \prod_{\stackrel{i\in I^c \cap A'^c,}{
i\neq N}}\suppresschi{[x_i\geq 0]} \right)\suppresschi{[x_N=0]}.
\end{multline*}
By the induction hypothesis, the right hand side of this equality
is the Brianchon-Gram decomposition of the semi-closed polytope \\
$\p(\Phi', A',\lambda )=\p(\Phi, A',\lambda )\cap \{x_N=0\} $.
\end{proof}

\begin{remark}
The formula is not necessarily true on the boundary of $\tau$,
as shown by the trivial example $\Phi=(\phi_1)$, $A=\{1\}$, $\lambda=0$.
\end{remark}
It will be useful to rephrase Proposition
\ref{semi-closed-brianchon-gram} in the terms which arise in the
combinatorial  wall-crossing Theorem \ref{th:wall_crossing}.

\begin{definition}
    For $A\subseteq\{1,\dots,N\}$ such that
    $\c(\Phi_{\rm flip}^A)$ is salient and $\lambda\in F$, let
\begin{multline}\label{eq:flipped_polytope}
\p_{\rm flip}(\Phi,A,\lambda)=\\
\{x\in \R^N\; ; \sum_i x_i\phi_i =\lambda, x_i<0  \mbox{ for }  i\in A, x_i\geq 0
\mbox{ for } i\notin A\}.
\end{multline}
\end{definition}

\begin{proposition}\label{geometric_flip}
Let $A\subseteq\{1,\dots,N\}$ be such that the cone $\c(\Phi_{\rm flip}^A)$ is
salient. Let $\tau$ be a $\Phi$-tope. Then, for $\lambda\in \tau$ we
have
\begin{equation}\label{eq:geometric_flip}
\geom_\emptyset \flip_A
X(\Phi_{\rm flip}^A,\tau)\suppresschi{[V(\Phi,\lambda)]}=\suppresschi{[\p_{\rm flip}(\Phi,A,\lambda)]}.
\end{equation}
\end{proposition}
\begin{proof}
For any  polynomial $Z\in\C[p_i,q_i]$,   we have
\begin{equation}\label{eq:geomflipA}
\geom_\emptyset\flip_A(Z)= \geom_A(Z)\circ \sigma^A
\end{equation}
where $\sigma^A x = (\sigma^A_i x_i) $ with $\sigma^A_i=-1$ if $i\in
A$ and $\sigma^A_i=1$ if $i\notin A$.  Moreover we have
$$
\suppresschi{[V(\Phi, \lambda)]}= \suppresschi{[V(\Phi_{\rm flip}^A, \lambda)]}\circ
\sigma^A.
$$
Thus
\begin{multline}\label{eq:geomA}
\geom_\emptyset \flip_A
X(\Phi_{\rm flip}^A,\tau)\suppresschi{[V(\Phi,\lambda]}=\\
\bigl(\geom_A X(\Phi_{\rm flip}^A,\tau)\suppresschi{[V(\Phi_{\rm flip}^A, \lambda)]}\bigr)\circ
\sigma^A.
\end{multline}
We  apply Proposition \ref{semi-closed-brianchon-gram} to the
sequence $\Phi_{\rm flip}^A$ and the $\Phi_{\rm flip}^A$-tope $\tau$. We obtain that the
right hand side of (\ref{eq:geomA}) is equal to
$$
\suppresschi{[\p(\Phi_{\rm flip}^A,A,\lambda]}\circ \sigma^A .
$$
By definition of $\p(\Phi_{\rm flip}^A,A,\lambda)$, this is precisely  the
characteristic function of the set of $x$ such that $\sigma^A_i
x_i=-x_i>0$ for $i\in A$ and $\sigma^A_i x_i= x_i\geq 0$ for
$i\notin A$,   and $\sigma^A x\in V(\Phi_{\rm flip}^A,
\lambda)$, i.e. $x\in \p_{\rm flip}(\Phi,A,\lambda)$.
\end{proof}

\subsection{Decomposition in  quadrants}
Recall that when $\tau$ and $\nu$ are two topes,
we have defined  a list $\CA(\nu,\tau)$ of signed subsets $[\epsilon,A]$ of $\{1,2,\ldots,N\}$.
\begin{theorem} \label{th:salient}
Let $z(\Phi,\tau,B), B\subset\{1,\dots,N\}$ be the
collection of coefficients of the Combinatorial Brianchon-Gram
polynomial associated to the $\Phi$-tope $\tau$.
\begin{equation*}
X(\Phi,\tau)=\sum_B z(\Phi,\tau,B)\prod_{i\notin B}p_i\prod_{i\in
 B}q_i.
 \end{equation*}

\noindent (i) If   $B=\emptyset$,   then  $z(\Phi,\tau,B)=1$ while if $B=\{1,2,\ldots,N\}$,  then $z(\Phi,\tau,B)=(-1)^d$.

\noindent (ii) If   $ z(\Phi,\tau,B)\neq 0$,  then   the cone $\c(\Phi_{\rm flip}^B)$ is
salient .

\noindent (iii) More precisely, if $z(\Phi,\tau,B)$ is not equal to $0$,
choose $x=(x_i)\in Q_{\rm neg}^B$ such that $\lambda=\sum_i x_i\phi_i$ is a regular element in $F$ and let $\nu$ be the tope containing $\lambda$. Then $B$ occurs in the list  in $\mathcal A(\nu,\tau)$ and we have
\begin{equation}\label{eq:calculcoeff}
 z(\Phi,\tau,B)=(-1)^{|B|}\sum_{\{[\epsilon,B]\in \CA(\nu,\tau)\}}\epsilon.
\end{equation}
\end{theorem}
\begin{proof}
We already remarked (i).

For the choice of $x$ as in (iii),
the coefficient $z_B=z(\Phi,\tau,B)$ is the value of $\geom X(\Phi,\tau)[V(\Phi,\lambda)]$ at such a point $x$.
We now apply Corollary \ref{th:wallcrossing_k_steps} and
Proposition \ref{geometric_flip} using the tope $\nu$, where $\lambda$ belongs.

We obtain that the value $z_B$ is the signed sum of the values of
$ \p_{\rm flip}(\Phi,A,\lambda)= [Q_{\rm neg}^A]\cap [V(\Phi,\lambda)]$ at $x$.
By definition, this value is zero if the set of $i$ with $x_i<0$ is different from $B$. Otherwise, is equal to $1$.
\end{proof}
\begin{example}
 Let $\tau_1$  and $\tau_2$ be the adjacent topes and  $A=A(\Phi,\tau_1,\tau_2)$.  If
$\tau_2\subset\c(\Phi_{\rm flip}^A)$, then $z(\Phi,\tau,A)=-(-1)^{|A|} $. An
example of this situation is the standard knapsack, (Example
\ref{standard_simplex}), where there are exactly two topes,
$\R_{>0}$ and $\R_{<0}$. The theorem implies that the Brianchon-Gram
polynomial is $p_1\cdots p_N-(-1)^N q_1\cdots q_N$, as we found
directly.
\end{example}

\begin{example}
If there is no subset $K$ in the sum (\ref{eq:calculcoeff}), or if
there are more than one,  then the coefficient $z(\Phi,\tau,B)$ may
be $0$ although the cone $\c(\Phi_{\rm flip}^B)$ is salient . Let us take
Example \ref{ex:intervalle_3_sommets_1}. The direct
computation gave $X(\Phi,\tau_1)=p_1p_2p_3 -  p_1q_2q_3 + q_1p_2p_3-q_1q_2q_3$. We see that there are no terms corresponding to
$B=\{2\}$ and $B=\{1,3\}$. For $B=\{2\}$, the tope
$\tau_2=\stackrel{\circ}{\c}(\phi_1,-\phi_2)$ is contained in
$\c(\phi_1,-\phi_2,\phi_3)$, hence we take the sequence
$(\tau_1,\tau_2)$. For $K=(1)$ we have $A_{(1)}=\{2,3\}\neq B$, so
there is no $K$ such that $A_K=B$.  For $B=\{1,3\}$, we need a
sequence of three topes, $(\tau_1,\tau_2,\tau_3)$ where now
$\tau_2=\stackrel{\circ}{\c}(\phi_2,\phi_3)$ and
$\tau_3=\stackrel{\circ}{\c}(-\phi_1,\phi_2)\subset \c(\Phi_{\rm flip}^B)$. The
$K$ such that $A_K=\{1,3\}$ are $K=(2)$ and $K=(1,2)$, which indeed
lead to  opposite signs in the sum (\ref{eq:calculcoeff}).
\end{example}

\subsection{Geometric wall-crossing}
By "intersecting" $\geom X(\Phi,\tau)$ with $[V(\Phi,\lambda)]$, we translate the results on $X(\Phi,\tau)$ in geometric terms.
\begin{corollary}[We keep the notations of Theorem \ref{th:wall_crossing}]\label{th:saut_geometrique}
Let $\tau_1$ be a  $\Phi$-tope. For $\lambda$ in the adjacent tope
$\tau_2$, we have the geometric wall-crossing for.mula
\begin{equation}\label{eq:geometric_wall-crossing}
\varchenko(\Phi,\tau_1)\suppresschi{[V(\Phi,\lambda)]}=\suppresschi{[\p(\Phi,
\lambda)]} -(-1)^{|A|}\suppresschi{[\p_{\rm flip}(\Phi,A,\lambda)]}.
\end{equation}
\end{corollary}
\begin{proof}
We apply the map $\geom$ on  both sides of the
combinatorial wall-crossing formula (\ref{eq:wall-crossing}), then
we multiply by the characteristic function $\suppresschi{[V(\Phi,\lambda)]}$.
We obtain, by definition,
$$
\varchenko(\Phi,\tau_1)[V(\Phi,\lambda)]=[\p(\Phi,\lambda]- (-1)^{|A|}\geom\flip_A X(\Phi_{\rm flip}^A,
\tau_2) \suppresschi{[V(\Phi,\lambda)]},
$$
hence (\ref{eq:geometric_wall-crossing})  by applying the semi-closed
Brianchon-Gram formula, as stated in Proposition
\ref{geometric_flip}, to the tope $\tau_2$.
\end{proof}

\begin{corollary}\label{th:virtual_polytopes}
For any $\lambda\in F$, the function
\begin{equation}\label{eq:virtual_polytopes}
\varchenko(\Phi,\tau)\suppresschi{[V(\Phi,\lambda)]}=\sum_{\{B,\lambda\in
\c(\Phi_{\rm flip}^B)\}}z(\Phi,\tau,B)[\p_{\rm flip}(\Phi,B,\lambda)]
\end{equation}
is a linear
combination with integral coefficients of semi-closed partition
polytopes.
\end{corollary}

\subsection{An example}\label{ex:tetragon}
We return to Example \ref{ex:zonotopeB2_suite}, see Fig. \ref{fig:B2topes}, with
$\Phi=(\phi_1,\phi_2,\phi_3=\frac{1}{2}(\phi_2-\phi_1),\phi_4=\frac{1}{2}(\phi_1+\phi_2))$.
\label{exampleB2}
 \begin{figure}[!h]
\begin{center}
  \includegraphics[width=1.5 in]{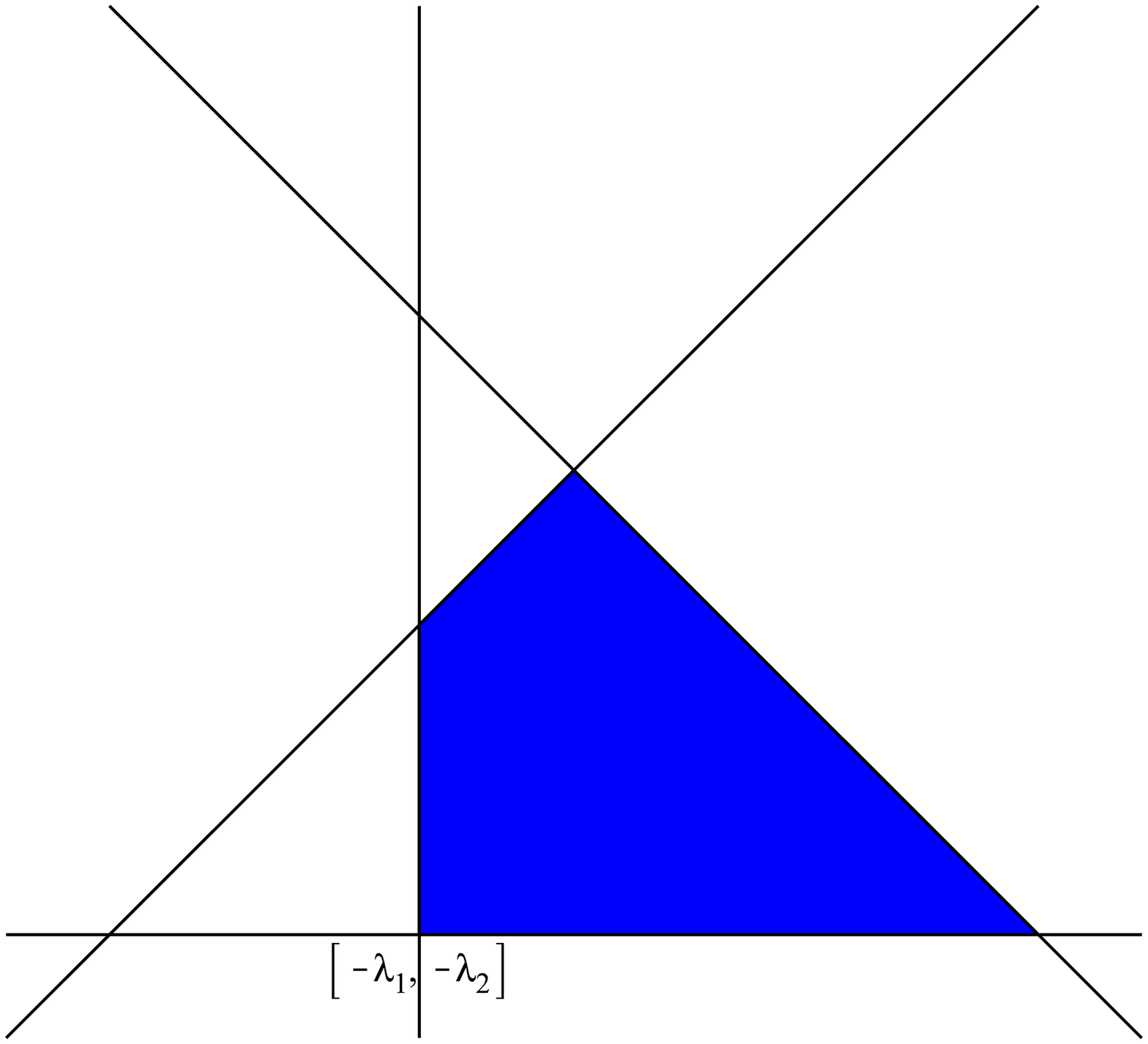}
  \includegraphics[width=1.5 in]{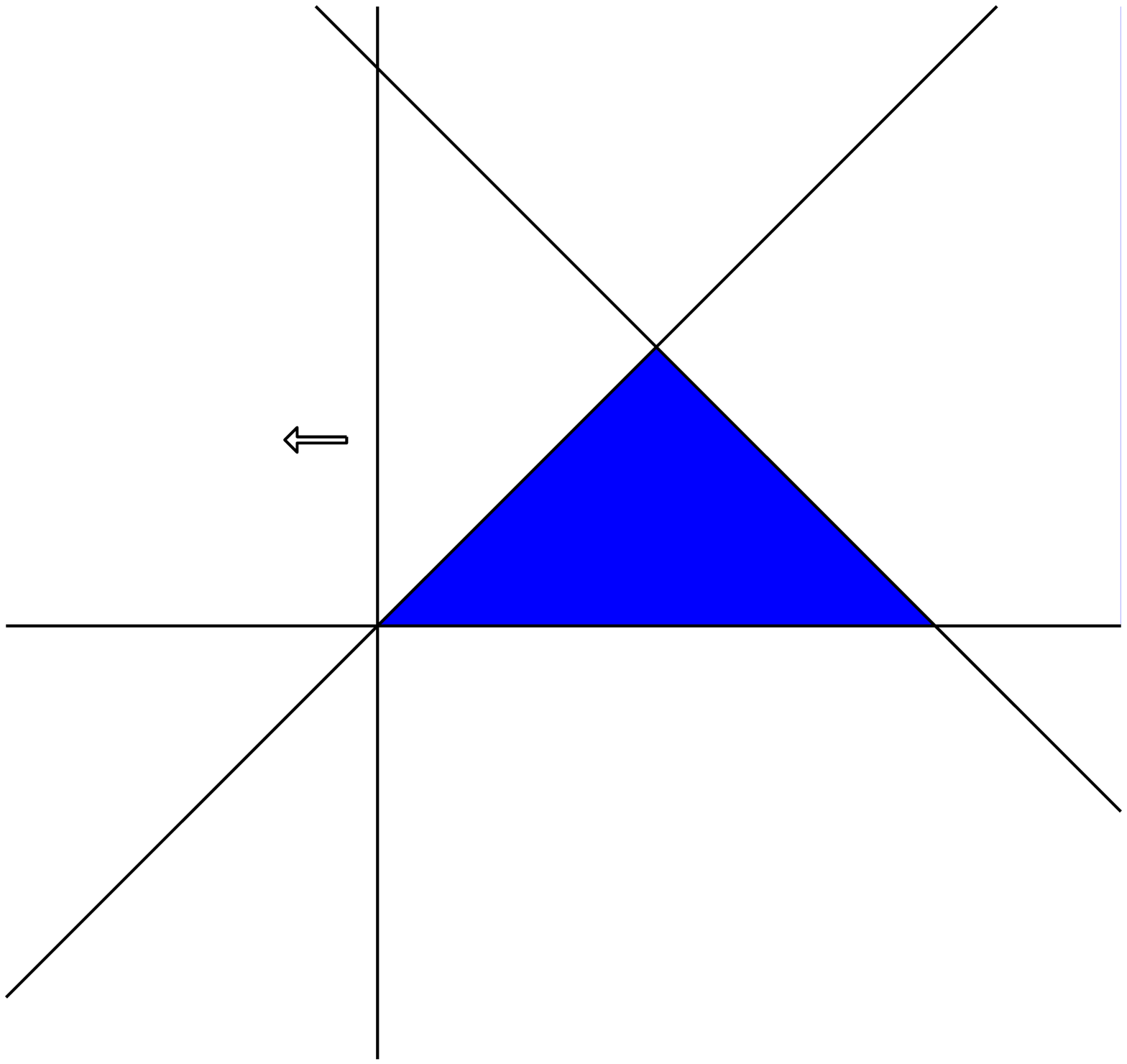}
  \caption{$\lambda_2>\lambda_1>0$, (tope $(\phi_2,\phi_4)$), then $\lambda_2=\lambda_1>0$, (wall $(\phi_4)$).}
  \label{chamber24}
  \end{center}
\end{figure}
\begin{figure}[!h]
\begin{center}
  \includegraphics[width=1.5 in]{chamber41.pdf}
  \includegraphics[width=1.5 in]{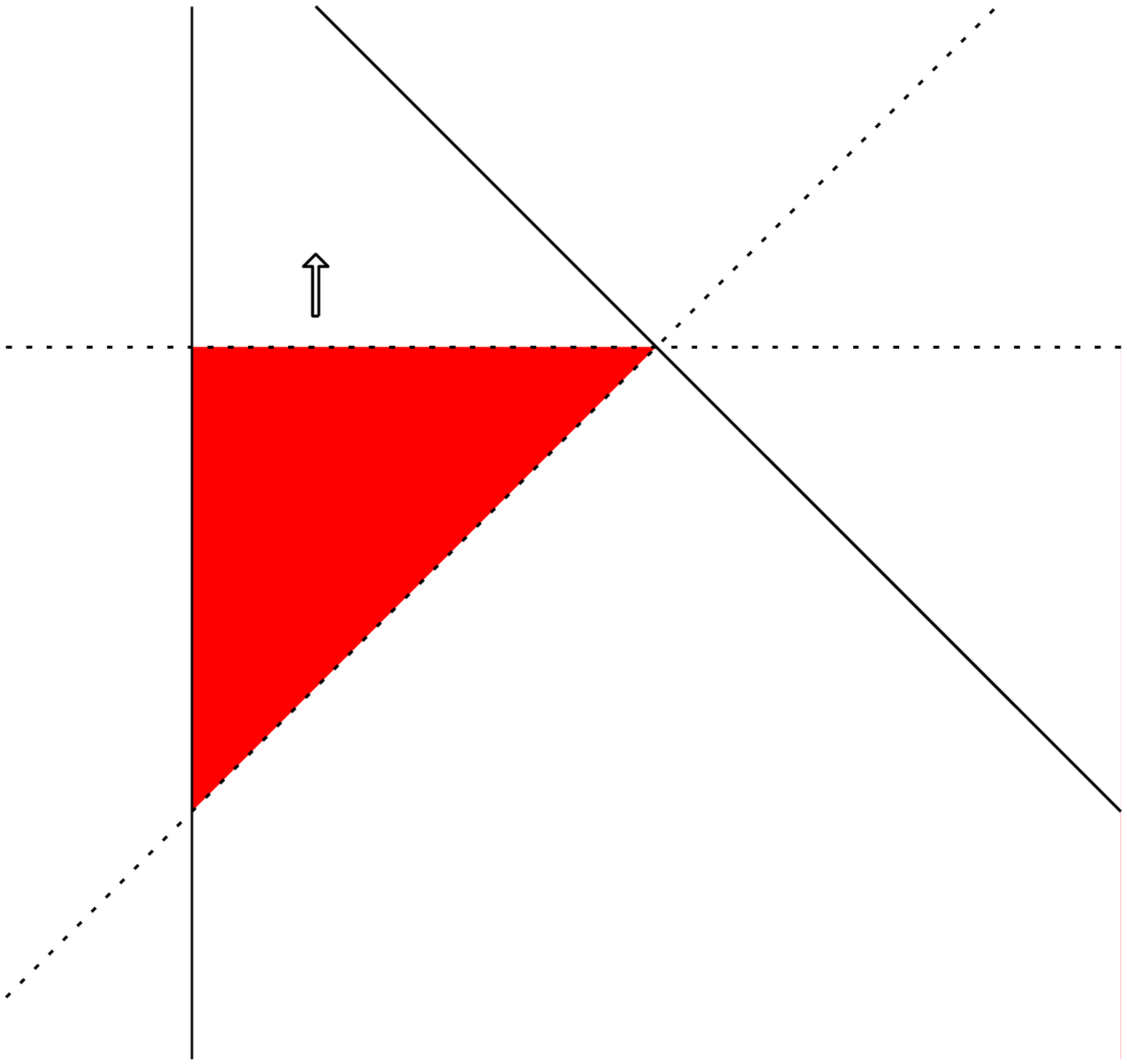}
   \caption{ $\lambda_1>\lambda_2>0$, (tope $(\phi_4,\phi_1)$), then $\lambda_1>0=\lambda_2$, (wall $(\phi_1)$).}
   \label{tope41}
\end{center}
\end{figure}
\begin{figure}[!h]
\begin{center}
  \includegraphics[width=1.5 in]{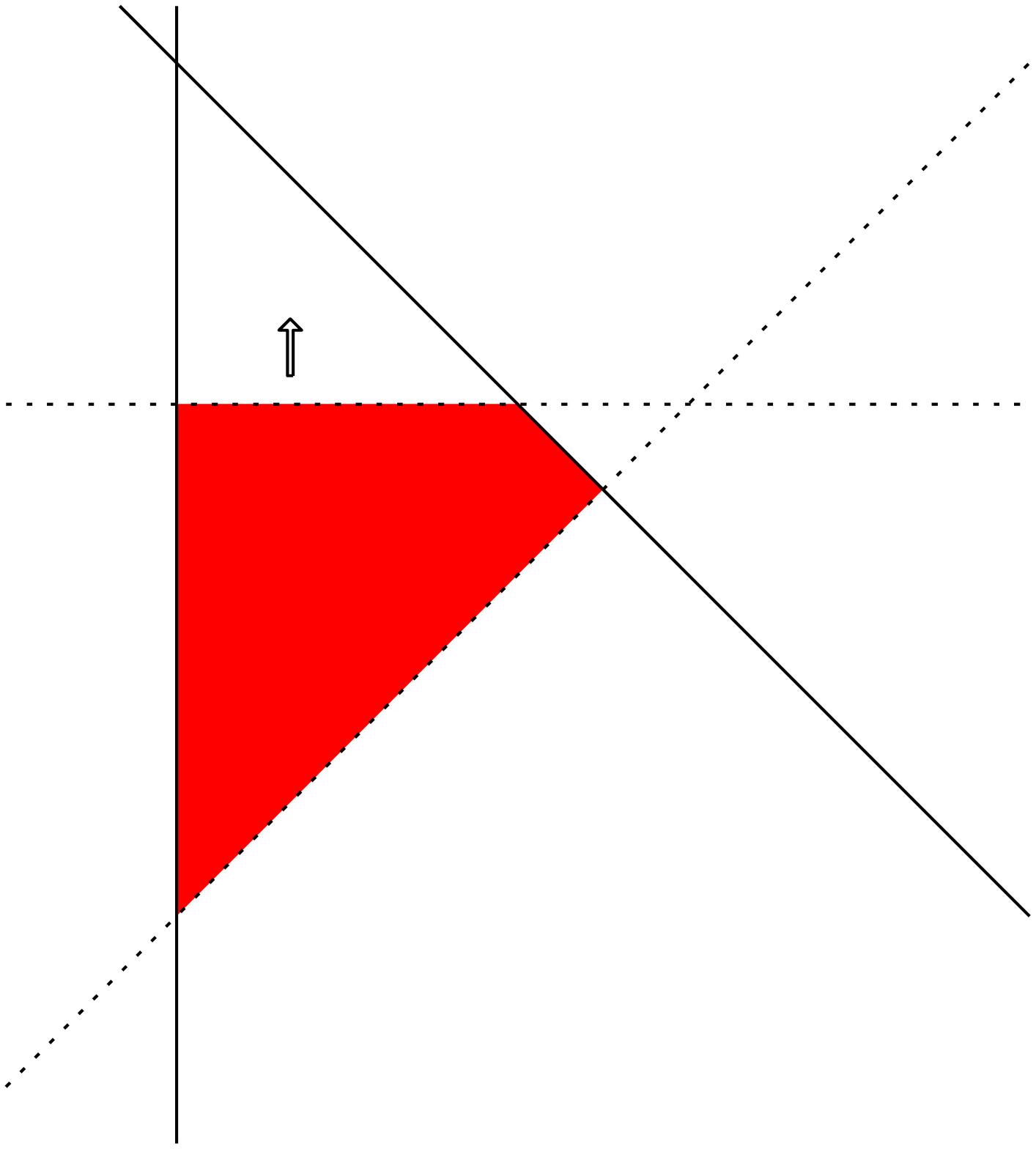}\includegraphics[width=1.5 in]{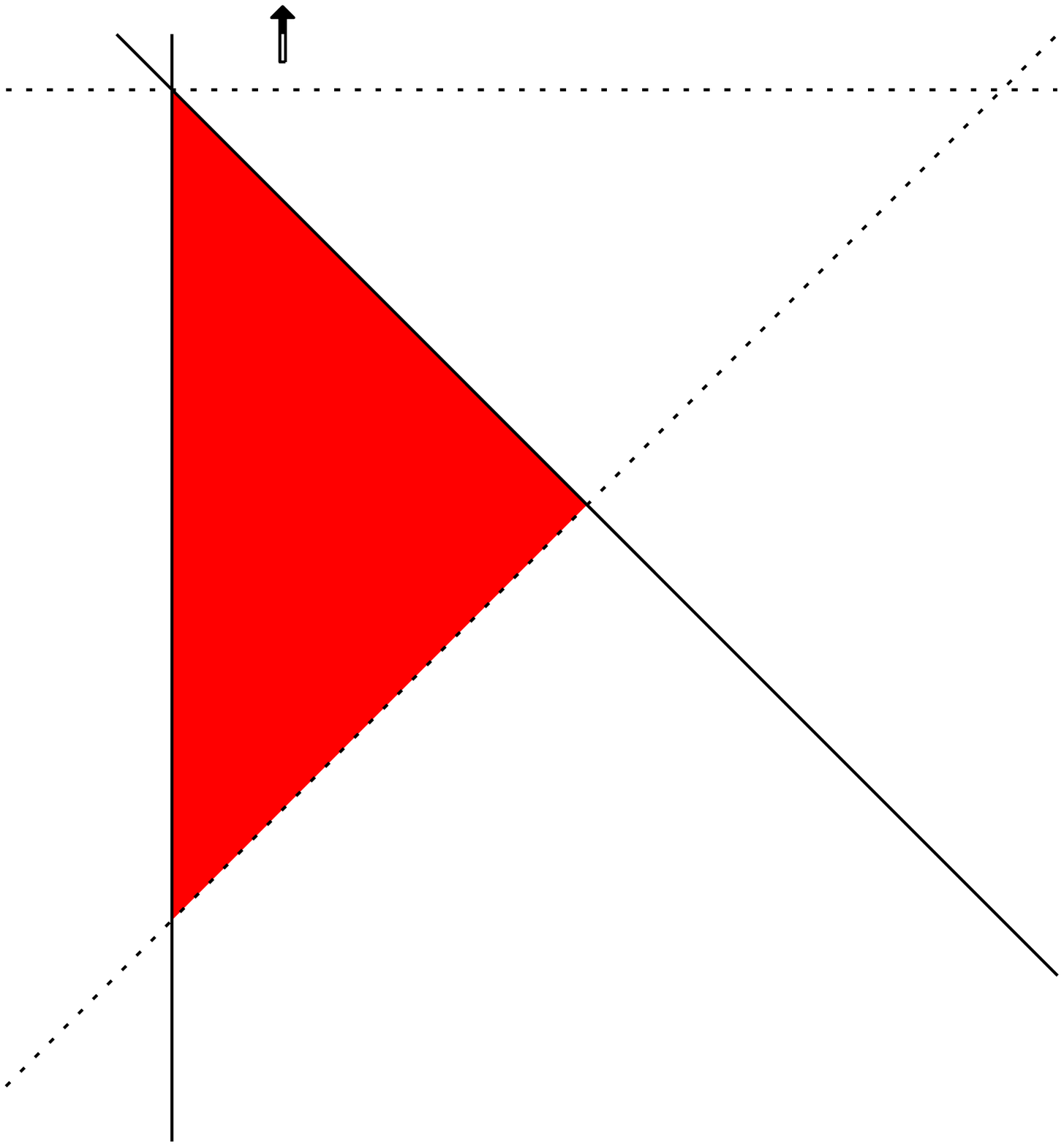}
  \caption{$\lambda_1>-\lambda_2>0$, (tope $(\phi_1,-\phi_3)$), then $\lambda_1=-\lambda_2>0$, (wall $(-\phi_3)$).}
  \label{chamber1minus3}
\end{center}
\end{figure}
\begin{figure}[!h]
\begin{center}
  \includegraphics[width=1.5 in]{chamberminus3minus2.pdf} \includegraphics[width=1.5 in]{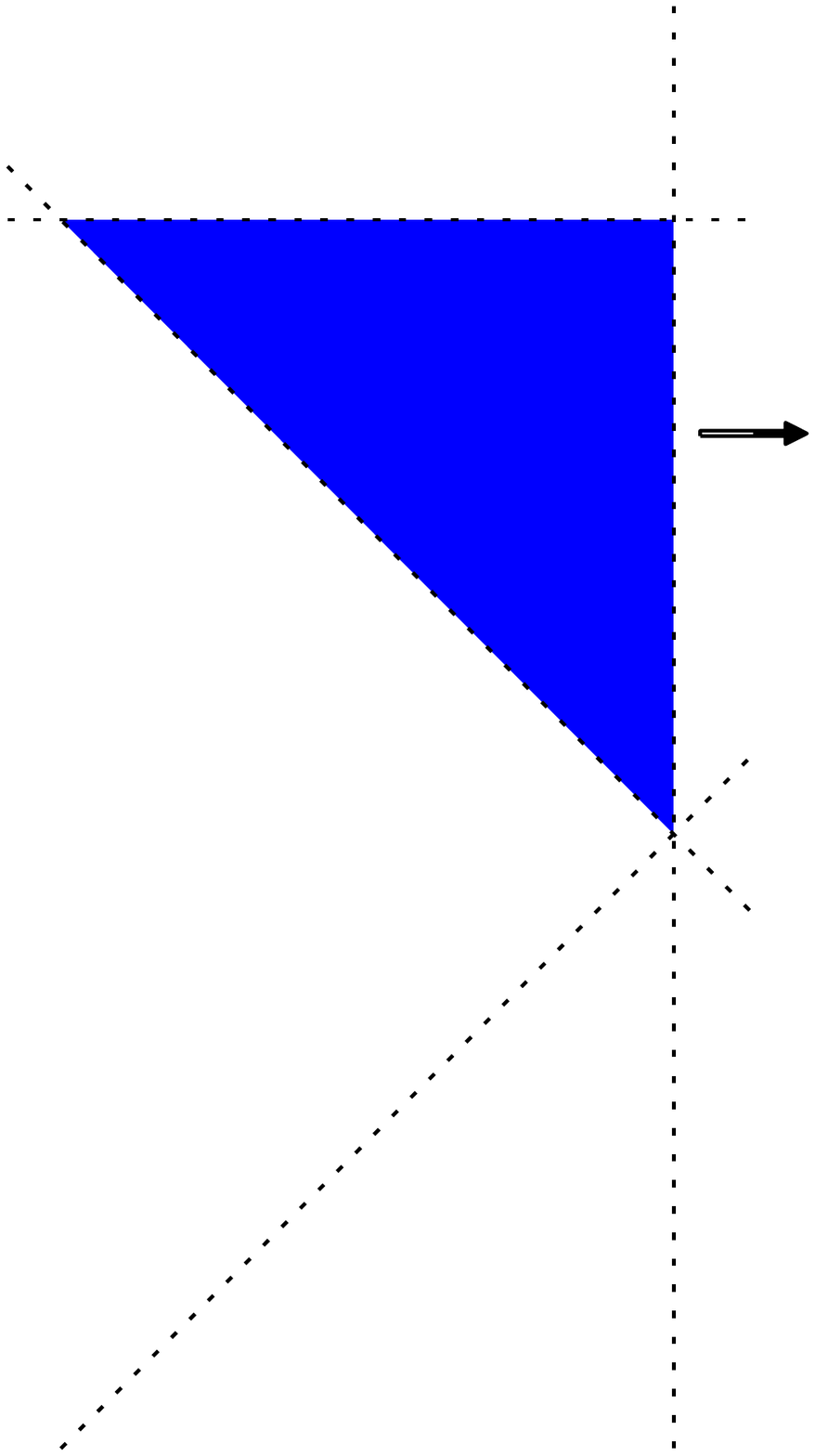}
  \caption{$-\lambda_2>\lambda_1 >0$, (tope $(-\phi_3,-\phi_2)$), then $- \lambda_2>0=\lambda_1$,  (wall $-\phi_2$).}
  \label{chamberminus3minus2}
\end{center}
\end{figure}
\begin{figure}[!h]
\begin{center}
  \includegraphics[width=1.5 in]{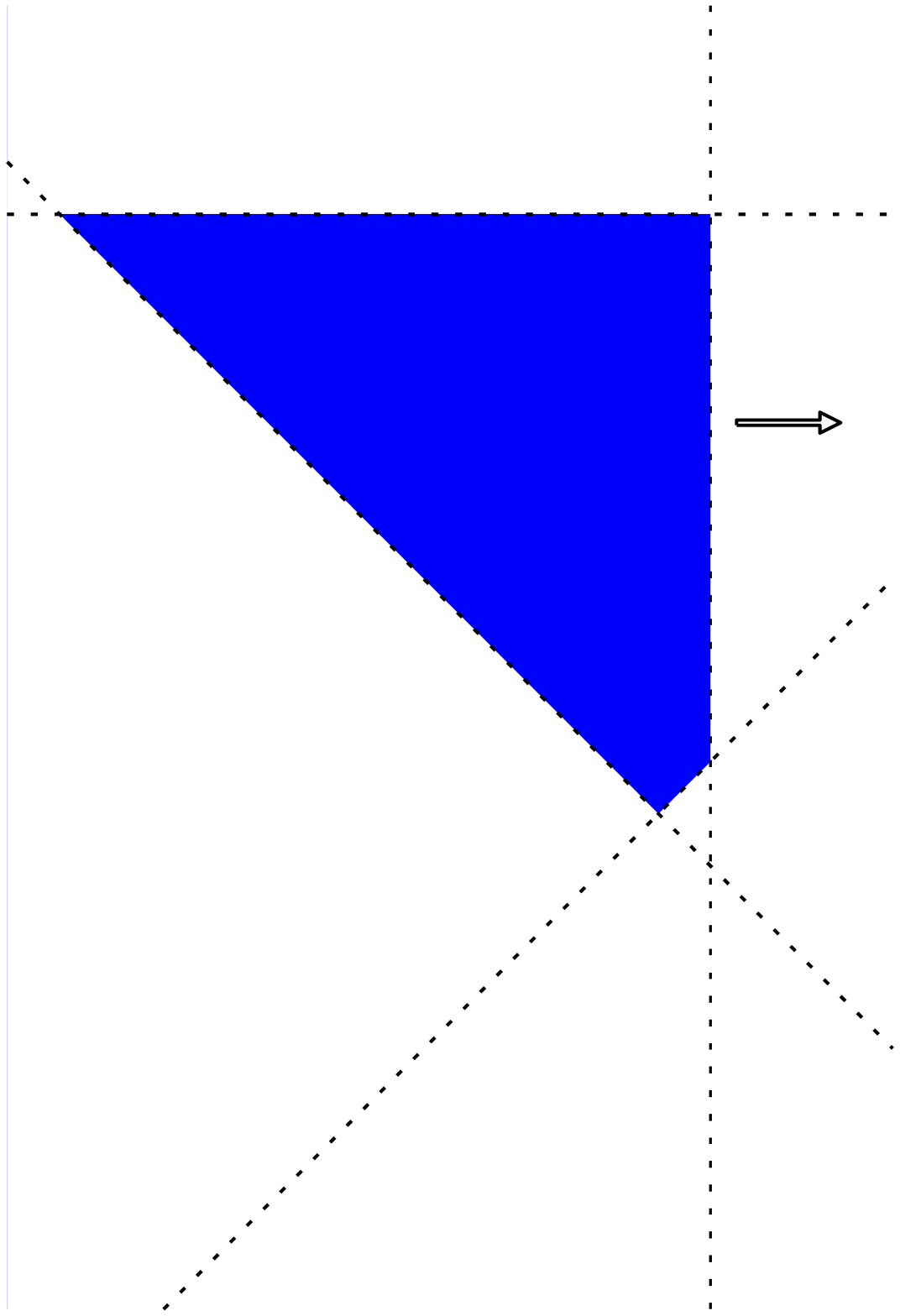}
 \caption{$\lambda_1<\lambda_2<0$,  (tope $(-\phi_2,-\phi_4)$). The polygon is now the interior of the opposite
 of the initial tetragon of Fig. \ref{chamber24},  cf \cite{varchenko}.}
 \label{chamberminus2minus4}
\end{center}
\end{figure}
For $\lambda=\lambda_1\phi_1+\lambda_2\phi_2$, we parametrize  the
$2$-dim subspace $V(\Phi,\lambda)\subset \R^4$ by   $(y_1,y_2)
\mapsto (y_1+\lambda_1, y_2+\lambda_2, y_1-y_2, -(y_1+y_2))$. We
start with $\lambda$  in the  tope $\tau_1$ generated by $\phi_2$
and $\phi_4$, i.e. $\lambda_2>\lambda_1>0$. Then $\p(\lambda)$ corresponds under the
parameterization to the tetragon in Fig. \ref{chamber24},
defined by the inequations
\begin{eqnarray*}
y_1+\lambda_1  &\geq & 0\\
y_2+\lambda_2  &\geq & 0\\
y_1 - y_2 &\geq & 0\\
y_1+y_2 &\leq& 0
\end{eqnarray*}
We describe its analytic continuation, as the parameter $\lambda$
visits the topes, one after the other.
 In the figures, the polytopes which are counted positively  are
coloured in blue, those wich are counted negatively are coloured in
red. Semi-openness is indicated with dashed lines.
 When $\lambda$
moves to the right and reaches the wall generated by $\phi_4$, the
tetragon $\p(\lambda)$ transforms into a triangle
(Fig. \ref{chamber24}). When $\lambda$ enters the adjacent tope $\tau_2$
generated by $(\phi_1,\phi_4)$, i.e. $\lambda_1>\lambda_2>0$,
 the wall-crossing polytope  appears (Fig. \ref{tope41}). It is (with
 sign $-1$) the  semi-closed triangle
\begin{eqnarray*}
y_1+\lambda_1  &\geq & 0\\
y_2+\lambda_2  &< & 0\\
y_1 - y_2 &<& 0\\
y_1+y_2 &\leq& 0, \mbox{  (this condition is redundant).  }
\end{eqnarray*}
 Then $\lambda$ moves downwards towards the wall generated by $\phi_1$.
  The positive closed triangle shrinks
  while the negative semi-closed one increases. When $\lambda$ reaches
  the wall, the closed triangle is reduced to a point (Fig. \ref{tope41}).
  When $\lambda$ enters the tope $(\phi_1,-\phi_3)$,   the negative semi-closed
  triangle deforms into a negative semi-closed quadrileral  (Fig. \ref{chamber1minus3}).
When  $\lambda$ reaches the wall generated by $-\phi_3$ (Fig.
\ref{chamberminus3minus2}), then enters the tope $(-\phi_3,-\phi_2)$, a new positive open triangle appears. Then
 $\lambda$ reaches the wall generated by $-\phi_2$
(Fig. \ref{chamberminus3minus2})  and enters the tope $(-\phi_2,-\phi_4)$
(Fig. \ref{chamberminus2minus4}). The analytic continuation is now a
positive open tetragon opposite to the (closed) initial one.
(This case is pointed out in \cite{varchenko}).

\section{Integrals and discrete sums over a partition
polytope}\label{Integrals_and_discrete_sums} As a  consequence of
the compacity result of Corollary \ref{th:virtual_polytopes},
together with the set theoretic relations of Corollary
\ref{th:continuity-on-closed-tope}, we recover properties of sums
and integrals over partition polytopes which were previously
obtained in \cite{brion-vergne-97-residue} and
\cite{szenes-vergne-2002}, \cite{deconcini-procesi-vergne-dahmen-micchelli-2008},\cite{dahmen-micchelli-1988}.
Moreover, the set-theoretic wall-crossing
formula has obvious implications for sums and integrals. In
particular, when applied to the number of points of a partition
polytope, it implies the wall-crossing formula of
\cite{paradan2004}, Theorem 5.2.  We will explain in more details this last point in Subsection \ref{paradan}.

\subsection{Generating functions of polyhedra and Brion's theorem}

Let $V$ be a real dimensional vector space. We choose a Lebesgue measure $dv$ on $V$.
Let us recall the notion of valuations and of generating functions of cones  (see  the survey \cite{BarviPom}).

Recall that a \emph{valuation} $F$ is a  map from a set of polyhedra
$ \p\subset V $ to a vector space $\mathcal M$ such that whenever the
characteristic functions $\suppresschi{[\p_i]}$ of a family of polyhedra $\p_i$
satisfy a linear relation $\sum_i r_i \suppresschi{[\p_i]}=0$, then the
elements $F(\p_i)$ satisfy the same relation $ \sum_i r_i
F(\p_i)=0$.
Thus any  valuation defined on the set of all polyhedra can be extended to the  ``analytic continuation", which is a signed sum of polytopes.
In particular, the valuation defined on the set of polyhedra by the Euler characteristic    (see   \cite{BarviPom})
is identically equal to $1$ on the ``analytic continuation" as follows from Brianchon-Gram decomposition and Euler relations.

We now study  two other classical instances of valuations.

 There exists a unique valuation  $\p \mapsto
\intpol(\p)$ which associates to every polyhedron
$\p\subseteq V$ a meromorphic function
$\intpol(\p)(\xi)$ on $V^*$, so that the following
properties hold:

\noindent(i)
 If  $\p$ contains a straight line,
  then $\intpol(\p)=0$.

\noindent(ii) If $\xi\in V^*$ is such that $\e^{\la \xi,x\ra}$ is
integrable over $\p$ for the measure $dv$,
  then
$$
\intpol(\p)(\xi)= \int_\p \e^{\la \xi,x\ra} \, dv.
$$
Moreover, for every point $s\in V$, one has
$$
\intpol(s+\p)(\xi) = \e^{\la
\xi,s\ra} \intpol(\p)(\xi).
$$
$I(\p)(\xi)$ is called the \emph{continuous generating
function} of~$\p$.

Assume that $V$ is equipped with a lattice $V_\Z$.

There exists a unique valuation  $\p\mapsto \discretepol(\p)$ which
associates to every rational polyhedron $\p\subseteq V$ a meromorphic function $\discretepol(\p)(\xi)$ on $V^*$,  so that

\noindent(i) if $\p$ contains a straight line, then
$\discretepol(\p)=0$;

\noindent(ii) if $\xi\in V^*$  is such that $\e^{\la \xi,x\ra}$ is
summable over the set
  of lattice points of $\p$, then
$$
\discretepol(\p)(\xi)= \sum_{x\in \, \p\cap V_\Z} \e^{\la \xi,x\ra}.
$$
Moreover, for every point $s\in V_\Z $, one has
$$
\discretepol(s+\p)(\xi) = \e^{\la \xi,s\ra}\discretepol(\p)(\xi).
$$
$\discretepol(\p)(\xi)$ is called the \emph{(discrete) generating
function} of $\p$.

These valuations are easily constructed, either by algebraic methods (see \cite{BarviPom}), or by introducing   the Fourier transforms of discrete or continuous measures associated to the polyhedron $\p$ (see Section \ref{section degenerate}).
Furthermore, there is  an important
property of the generating functions $\discretepol(\p)(\xi)$ and
$\intpol(\p)(\xi)$.
Introduce the space $\CM_{\ell}(V^*)$ of meromorphic functions on $V^*$ which  can be written as the quotient of a
function which is holomorphic near $\xi=0$ by a product of linear
forms.
The functions  $\intpol(\p)(\xi)$  and $\discretepol(\p)(\xi)$
 belong to the space $\CM_{\ell}(V^*)$.
Then a function $f(\xi)\in \CM_{\ell}(V^*)$  has a unique expansion into homogeneous
rational functions
$$
f(\xi)= \sum_{m\geq m_0}f_{[m]}(\xi),
$$
where the summands $f_{[m]}(\xi)$ have degree $m$ as we define now:
if $P$ is a homogeneous polynomial on $V^*$ of degree $p$, and $D$ a
product of $r$ linear forms, then $\frac{P}{D}$ is an element in
$\CM_{\ell}(V^*)$  homogeneous of degree $m=p-r$.

 Let $\p$ be a polytope with set of faces $\CF(\p)$, and  affine tangent cones  $\t_{\aff}(\p,\f)$  at $\f$.
 We obtain from the Brianchon-Gram theorem:

 $$\int_{\p}e^{\ll\xi,v\rr}dv=\sum_{\f} (-1)^{\dim \f}
 I(\t_{\aff}(\p,\f))(\xi).$$

 Furthermore, as the cone   $\t_{\aff}(\p,\f)$ contains a straight line, when the dimension of $\f$ is strictly greater than $0$, this gives the well-known Brion's formula:

 $$\int_{\p}e^{\ll\xi,v\rr}dv=\sum_{s}
 I(s+\c_s)(\xi).$$

 Here $s$ runs through the vertices of $\p$ and $\c_s$ is the tangent cone at $s$.

Similarly, when $V$ is a rational vector space with lattice $V_\Z$, and $\p$ a rational polytope,
we have

\begin{equation}\label{Brionsum}
\sum_{x\in \p\cap V_\Z} e^{\ll\xi,x\rr}=\sum_s S(s+\c_s)(\xi).
\end{equation}

These formulae are at the heart of Varchenko's  ``analytic continuation procedure": we see intuitively that if the vertices of a polytope $\q(b)$ vary ``analytically" with a parameter $b$, the integrals  and discrete sums  will also vary ''analytically". We will state precise results in the next section.

\subsection{Polynomiality and wall-crossing for integrals and sums}

Recall the following definition

\begin{definition}\label{quasi}
If $F$ is equipped with a lattice $\Lambda$, a quasi-polynomial function $f$ on $\Lambda$ is a function such that there exists a sublattice $\Lambda'\subset \Lambda$ so that, for any $\lambda_0\in \Lambda$, the function $\lambda'\to f(\lambda_0+\lambda')$  is given by the restriction to $\Lambda'$ of a polynomial function $f_{\lambda_0}$ on $F$.
\end{definition}

\begin{theorem}\label{th:polynom_integral_and_discrete}
Let $\Phi=(\phi_j)_{1\leq j\leq N} $ be a sequence of non zero elements
of a vector space $F$, generating F, and  spanning a salient cone.
Let $\tau\subset F$ be a $\Phi$-tope such that
$\tau$ is contained in the cone $\c(\Phi)$ generated by $\Phi$. For
$\lambda\in F$, let $V(\Phi,\lambda)$ be the affine subspace of
$\R^N$ defined by $\sum_{i=1}^N x_i \phi_i= \lambda$. Let
$$
\varchenko(\Phi,\tau)=\sum_{I\in \CG(\Phi,\tau)}(-1)^{|I|-\dim
F}\prod_{i\in I^c}\suppresschi{[x_i\geq 0]},
$$
 where  $\CG(\Phi,\tau)$ is the set of
 $I\subseteq \{1,\dots, N\}$ such that $\{\phi_i,i\in I\}$ generates $F$
 and such that $\tau$ is contained in the
cone generated by $\{\phi_i,i\in I\}$.

Let $h(x)$ be a polynomial function on $\R^N$.
 Fix a Lebesgue measure on the subspace $V$ and
let $dm_{\Phi} (x)$ be the corresponding Lebesgue measure on
$V(\Phi,\lambda)$. Define
\begin{equation}\label{def:polynom_integral}
\intpol(\Phi,\tau,h)(\lambda)= \int_{V(\Phi,\lambda)}
\varchenko(\Phi,\tau)(x) h(x)
 dm_{\Phi}(x).
\end{equation}
 In the case where $F$ is a rational space with
lattice $\lattice$ and  that the  $\phi_i$ are lattice vectors, define \begin{equation}\label{def:polynom_discrete_sum}
\discretepol(\Phi,\tau,h)(\lambda)= \sum_{x\in V(\Phi,\lambda)\cap
\Z^N} \varchenko(\Phi,\tau)(x) h(x).
\end{equation}
Then

 \noindent (i) $\lambda\mapsto \intpol(\Phi,\tau,h)(\lambda)$ is a
polynomial function on $F$.

\noindent (ii) $\lambda\mapsto \discretepol(\Phi,\tau,h)(\lambda)$
is a quasi-polynomial function on the lattice $\lattice\subset F $.

 \noindent (iii) If  $\lambda$ belongs to the closure  of the tope
$\tau$, we have
\begin{equation}\label{eq:polynom_integral}
 \intpol(\Phi,\tau,h)(\lambda)=\int_{\p(\Phi,\lambda)}h(x)
 dm_{\Phi}(x).
\end{equation}
  Let $\b(\Phi)$ be the zonotope generated by $\Phi$.
If $\lambda \in (\tau-\b(\Phi))\cap \lattice$, we have
\begin{equation}\label{eq:polynom_discrete_sum}
\discretepol(\Phi,\tau,h)(\lambda)= \sum_{x\in \p(\Phi,\lambda)\cap
\Z^N} h(x).
\end{equation}

 \noindent (iv)  Furthermore, we have the following wall-crossing formulas (with the
notations of Theorem \ref{th:wall_crossing}).
   For
$\lambda\in \tau_2$, we have
\begin{equation}\label{eq:wall_crossing_integral}
\intpol(\Phi,\tau_1,h)(\lambda)=\int_{\p(\Phi,\lambda)} h(x)
 dm_{\Phi}(x)-(-1)^{|A|}\int_{\p_{\rm flip}(\Phi,A,\lambda)} h(x)
 dm_{\Phi}(x).
\end{equation}
\begin{equation}\label{eq:wall_crossingsum}
\discretepol(\Phi,\tau_1,h)(\lambda)=\sum_{x\in \p(\Phi,\lambda)\cap
\Z^N} h(x)-(-1)^{|A|}\sum_{x\in \p_{\rm flip}(\Phi,A,\lambda) \cap \Z^N}
h(x).
\end{equation}
\end{theorem}
\begin{proof}
(iii)  follows immediately from Corollary
\ref{th:continuity-on-closed-tope} and (iv) from the  wall crossing
formulas of Corollary \ref{th:saut_geometrique}, together with
Corollary \ref{th:continuity-on-closed-tope}.

The proof  of the polynomiality in (i) and (ii) relies on  the
properties of generating functions, as we explain in \cite{BBDKV-2010} for the
weighted Ehrhart theory.

 To begin
with, observe that it is enough to prove the theorem in the case
where the weight $h(x)$ is a power of a linear form
$$
h(x)=\frac{\langle\xi,x\rangle^M}{M!},
$$
for $\xi\in (\R^N)^*$.    This is the term of $\xi$-degree $M$ of
the exponential $ e^{\langle\xi,x\rangle}$.  Thus, we consider the
 functions of $\xi\in (\R^N)^*$
\begin{equation}\label{def:generating_integral}
\intpol(\Phi,\tau)(\xi,\lambda)= \int_{V(\Phi,\lambda)}
\varchenko(\Phi,\tau)(x) e^{\langle\xi,x\rangle}dm_{\Phi}(x).
\end{equation}
\begin{equation}\label{def:generating_discrete_sum}
\discretepol(\Phi,\tau)(\xi,\lambda)= \sum_{x\in V(\Phi,\lambda)\cap
\Z^N} \varchenko(\Phi,\tau)(x) e^{\langle\xi,x\rangle}.
\end{equation}
As $\varchenko(\Phi,\tau)(x)\suppresschi{[V(\Phi,\lambda)]}(x)$ has bounded
support by Corollary \ref{co:bounded},
(\ref{def:generating_integral}) and
(\ref{def:generating_discrete_sum}) are holomorphic functions of
$\xi$. We recover $\intpol(\Phi,\tau,h)(\lambda)$ and
$\discretepol(\Phi,\tau,h)(\lambda)$  by taking their term of
$\xi$-degree $M$.

However the dependance on $\lambda$ can be analyzed by looking at each
summand in
\begin{equation}\label{repeat}
\varchenko(\Phi,\tau)\cap \suppresschi{[V(\Phi,\lambda)]}=\sum_{K\in \CG(\Phi,\tau)} (-1)^{|K|-\dim F} \suppresschi{[\t_K(\Phi,\lambda)]}.
\end{equation}

Indeed, it is immediate to extend the valuation $I(\p)$ defined in the preceding section  to a valuation $I(\p,\lambda)$ defined on polyhedrons contained in the affine space $V(\Phi,\lambda).$

Namely, there exists a  unique valuation
$\intpol(\p,\lambda)$  associating to every polyhedron
$\p\subseteq V(\Phi,\lambda)$ a meromorphic function
$\intpol(\p,\lambda)(\xi)$ on $\C^N$, so that the following
properties hold:

\noindent(i)
 If  $\p$ contains a straight line,
  then $\intpol(\p,\lambda)=0$.

   \noindent(ii) If $\xi\in \C^N$ is such that $\e^{\la \xi,x\ra}$ is
integrable over $\p$ for the measure $dm_{\Phi}$,
  then
$$
\intpol(\p,\lambda)(\xi)= \int_\p \e^{\la \xi,x\ra} \, dm_\Phi(x).$$

Moreover, for every point $s\in \R^N$, one has
$$
\intpol(s+\p,\lambda+\sum_{i=1}^N s_i\phi_i)(\xi) = \e^{\la
\xi,s\ra} \intpol(\p, \lambda)(\xi).
$$

Similarly, if $F$ is a space with a lattice $\Lambda$, and the elements $\phi_i$ belongs to $\Lambda$, then, for $\lambda\in \Lambda$, there  exists a unique valuation  $\p\mapsto \discretepol(\p,\lambda)$
associating to  any $\lambda\in \Lambda$ and every rational polyhedron $\p\subseteq V(\Phi,\lambda)$ a meromorphic
function $\discretepol(\p,\lambda)(\xi)$ on $\C^N$,  so that

\noindent(i) if $\p$ contains a straight line, then
$\discretepol(\p,\lambda)=0$;

\noindent(ii) if $\xi\in \C^N$  is such that $\e^{\la \xi,x\ra}$ is
summable over the set $\p\cap \Z^N$, then
$$
\discretepol(\p,\lambda)(\xi)= \sum_{x\in \, \p\cap \Z^N} \e^{\la \xi,x\ra}.
$$
Moreover, for every point $s\in\Z^N $, one has
$$
\discretepol(s+\p,\lambda+\sum_i s_i\phi_i)(\xi) = \e^{\la \xi,s\ra}\discretepol(\p,\lambda)(\xi).
$$

Look at Equation (\ref{repeat}).
The polyhedron $\t_K(\Phi,\lambda)$ contains a straight line as soon if $K\in \CG(\Phi,\tau)$  is not a basic subset.
Thus,  in terms of the valuations
 $\intpol(\p,\lambda)$, $\discretepol(\p,\lambda)$, we have
$$
\intpol(\Phi,\tau)(\xi,\lambda)=\sum_{K\in
\CB(\Phi,\tau)}\intpol(\t_K(\Phi,\lambda),\lambda)(\xi),
$$
and
$$
\discretepol(\Phi,\tau)(\xi,\lambda)=\sum_{K\in
\CB(\Phi,\tau)}\discretepol(\t_K(\Phi,\lambda),\lambda)(\xi).
$$
Each of these equations  expresses a holomorphic function as a sum
of meromorphic ones  whose poles cancel out.  Furthermore, we can
recover the term of $\xi$-degree $M$ by taking the homogeneous de degree in $\xi$ in each of these functions of $\xi$.

Regarding the dependance on $\lambda$, we have already observed  the following crucial fact:
the cone $\t_K(\Phi,\lambda)$ is the shift $s_K(\Phi,\lambda)+\a_0(K) $
of the \textbf{fixed cone}
$\a_0(K) $
 by the vertex $s_K(\Phi,\lambda) $ depending linearly of $\lambda$.

Let us first study the integral.
Using the translation property of the valuation $\intpol(\p,\lambda)$, we can express $\intpol(\p,\lambda)$  in function of the valuation $\intpol(\p)$ defined on polyhedrons contained in the fixed space $V$.
We  then have
\begin{equation*}\label{eq:generating_integral-coneK-1}
\intpol(\t_K(\Phi,\lambda),\lambda)(\xi)=e^{\langle\xi,s_K(\Phi,\lambda)\rangle}\intpol(\a_0(K))(\xi).
\end{equation*}
Only the first factor $e^{\langle\xi,s_K(\Phi,\lambda)\rangle}$ depends
on $\lambda$.
 Actually, it is easy to see that
$\intpol(\a_0(K))(\xi)$ is homogeneous of degree $-d$. Hence,
the term of $\xi$-degree $M$ of $\intpol(\t_K(\Phi,\lambda),\lambda)(\xi)$
is given by
$$
\intpol(\t_K(\Phi,\lambda),\lambda)(\xi)_{[M]}=\frac{\langle\xi,s_K(\Phi,\lambda)\rangle^{M+d}}{(M+d)!}
\intpol(\a_0(K))(\xi)_{[-d]}.
$$
Thus we have, for $h(x)=\frac{\langle\xi,x\rangle^M}{M!}$,
\begin{equation*}\label{eq:generating_integral-coneK-2}
\intpol(\Phi,\tau,h)(\lambda)=
\sum_{K\in \mathcal B(\Phi,\tau)}\frac{\langle\xi,s_K(\Phi,\lambda)\rangle^{M+d}}{(M+d)!}
\intpol(\a_0(K))(\xi)_{[-d]}.
\end{equation*}
The right hand side of this formula is a polynomial function of
$\lambda$ of degree $M+d$, with coefficients which are polynomial functions of $\xi$
of degree $M$, (although each $K$ summand has poles). Thus we have
proved (i).

Let us now study the discrete sum
$$\discretepol(\Phi,\tau)(\xi,\lambda)=\sum_{K\in
\CB(\Phi,\tau)}\discretepol(\t_K(\Phi,\lambda),\lambda)(\xi).
$$

 Consider a sublattice $\Lambda'$ of $\Lambda$ such that all elements $s_K(\Phi,\lambda')$ have integral coefficients for $\lambda'\in \Lambda'$ and all $K\in \mathcal B(\Phi,\tau)$.
If $D$ is the least common multiple of all determinants of the $\Phi$-basic subsets $I$ in $\mathcal B(\Phi)$, we can choose
$\Lambda'=D\Lambda$.

Thus, if $\lambda=\lambda_0+\lambda'$, with $\lambda_0\in \Lambda$,$\lambda'\in \Lambda'$, we obtain
$$\discretepol(\t_K(\Phi,\lambda_0+\lambda'),\lambda_0+\lambda')(\xi)=
\e^{\la\xi,s_K(\Phi,\lambda')\ra}\discretepol(\t_K(\Phi,\lambda_0),\lambda_0)(\xi).
$$
Indeed $\t_K(\Phi,\lambda_0+\lambda')=s_K(\Phi,\lambda')+\t_K(\Phi,\lambda_0)$.

Here again the dependance in $\lambda'$ is only through the factor $\e^{\la\xi,s_K(\Phi,\lambda')\ra}$ and $s_K(\Phi,\lambda')$ depends linearly on $\lambda'$.

If $f_{\lambda_0}(K,\xi)=\discretepol(\t_K(\Phi,\lambda_0),\lambda_0)(\xi) $ ,
a meromorphic function of $\xi$ of degree greater or equal to $-d$,
we obtain:
$$\discretepol(\t_K(\Phi,\lambda_0+\lambda'),\lambda_0+\lambda')(\xi)_{[M]}=\sum_{k=0}^{M+d}
\frac{\la\xi,s_K(\Phi,\lambda')\ra ^k}{k!} f_{\lambda_0}(K,\xi)_{[M-k]}.
$$
This is a polynomial function of $\lambda'$ of degree $M+d$.

Adding up the contributions, we see that we obtain that
  $$\lambda'\to \discretepol(\Phi,\tau)(\xi,\lambda_0+\lambda')$$
is a polynomial function of $\lambda'$ and $\xi$.

\end{proof}

\bigskip

\begin{remark}
Consider Equation  (\ref{eq:virtual_polytopes}):
\begin{equation}\label{bis}
\varchenko(\Phi,\tau)\suppresschi{[V(\Phi,\lambda)]}=\sum_{\{B,\lambda\in
\c(\Phi_{\rm flip}^B)\}}z(\Phi,\tau,B)\suppresschi{[Q_{\rm neg}^B]}\suppresschi{[V(\Phi,\lambda)]}.
\end{equation}

Consider the case where the elements $\phi_i$ are in a lattice $\Lambda$ of $F$. Summing up the function $h=1$ over $V(\Phi,\lambda)\cap \Z^N$ on both sides, we obtain an expression for the quasi polynomial function $\discretepol(\Phi,\tau,h)(\lambda)$ in function of the partition functions associated to the flipped systems $\Phi_{\rm flip}^B$.

The functions $\discretepol(\Phi,\tau,h)(\lambda)$  are elements of the Dahmen-Micchelli space associated to $\Phi$ and $\Lambda$. It was proved in \cite{deconcini-procesi-vergne-dahmen-micchelli-2008} that any Dahmen-Micchelli quasi polynomial can be expressed as a linear combination of partition functions associated to flipped systems $\Phi_{\rm flip}^B$. The equation
(\ref{bis}) can be considered as a ``set-theoretic" generalization of this theorem.

\end{remark}

\subsection{Paradan's convolution wall-crossing formulas}\label{paradan}
We assume that $F$ is equipped with a lattice $\Lambda$.

The convolution of two functions $f_1,f_2$  (satisfying adequate support conditions) on $\Lambda$ is defined  by
$$(f_1*f_2)(\mu)=\sum_{\lambda_1+\lambda_2=\mu}f_1(\lambda_1) f_2(\lambda_2).$$

If $\mu\in \Lambda$,  we write $\delta_\mu$ for the function on $\Lambda$ such that $f(\lambda)=\delta_\mu^{\lambda}$.

 Let $h$ be a polynomial function on $\R^N$, and consider
 $$E(\Phi,h)(\lambda)=\sum_{x\in \p(\Phi,\lambda)}h(x).$$

When $h$ is the constant function $1$, then
$$
k(\Phi)(\lambda)=E(\Phi,1)(\lambda)={\rm Card}(\p(\Phi,\lambda)\cap \Z^N)
$$
is  the partition function associated to the sequence $\Phi$.
The function $k(\Phi)(\lambda)$ is the  convolution product $f_{\phi_1}*\cdots*f_{\phi_N}$, where, for $\phi\in F$,
$$
f_\phi:=\sum_{n= 0}^{\infty}\delta_{n\phi}.
$$
Indeed, by definition $k(\Phi)(\lambda)$
is the number of solutions in integers $n_i\geq 0$ of the equation $\sum_i n_i\phi_i=\lambda$.

The case of a  polynomial function $h$ can be treated similarly.
Assume $h$ is a product  $h(x_1,x_2,\ldots, x_N)=\prod_{i=1}^Nh_i(x_i)$
where $h_i$ are polynomial functions on $\R$.
For $h$ a  polynomial function on $\R$, and $\phi$ a non zero element in $\Lambda$, introduce
$$
f_\Phi^h=\sum_{n=0}^{\infty}h(n)\delta_{n\phi}.
$$
Then we see that
$$
E(\Phi,h)=f_{\phi_1}^{h_1}*\cdots*f_{\phi_N}^{h_N}.
$$
With the notations of Theorem \ref{th:polynom_integral_and_discrete},
for each tope $\tau$, the function
  $\discretepol(\Phi,\tau,h)(\lambda)$ is a quasi-polynomial function  on the lattice $\Lambda$
such that $E(\Phi,h)(\lambda)=\discretepol(\Phi,h,\tau)(\lambda)$ for $\lambda\in (\tau-\b(\Phi))\cap \Lambda$.

Let $\tau_1,\tau_2$ be two adjacent topes separated by a wall $H$.
Let $\tau_{12}$ be the unique tope of $\Phi\cap H$
such that $\overline{\tau_1}\cap\overline{\tau_2}\subset
\overline{\tau_{12}}$.
 Paradan's formula is a formula for $\discretepol(\Phi,h,\tau_1)-\discretepol(\Phi,h,\tau_2)$
 when $\tau_1,\tau_2$ are adjacent topes in terms
 of the convolution of the  quasi-polynomial function
 $\discretepol(\Phi\cap H,h,\tau_{1,2})$ on $\Lambda\cap H$ with some the functions $f_\Phi^h$.

Before stating the formula, we
note one property of the function $\discretepol(\Phi,h,\tau)$.

Assume that $H$ is a face of the cone $\c(\Phi)$ and that $\tau$  is a tope with one of its wall equal to $H$.
 Let $\tau_H$ be the unique tope of $\Phi\cap H$ so that
$\overline\tau\cap H$ is contained in $\tau_H$.
Let $\discretepol(\Phi\cap H,h,\tau_H)$ be the quasi-polynomial function on $\Lambda\cap H$ associated to this data.

Let us denote the subsequence of elements $\phi_i$  not in $H$  by  $\Phi\setminus H=(\phi_1,\ldots, \phi_M)$.
Then if $n_1,n_2,\ldots, n_M$ are non negative integers, and $\lambda\in \Lambda$,  there are only a finite number of $n_i$ such that $\lambda-\sum_{i=1}^M n_i \phi_i$ belongs to $H$, as  the elements $\phi_i$ are all on one side of $H$.

Let $H^{\geq 0}$ be the closed half space delimited by $H$ and containing $\tau$.
\begin{proposition}\label{prop:covol}
For $\lambda\in H^{\geq 0}$,
$\discretepol(\Phi,h,\tau)(\lambda)$ is equal to
$$
\sum_{n_i\geq 0,\lambda-\sum_i n_i\phi_i\in H}   (h_1(n_1)\cdots h_M(n_M)) \discretepol(\Phi\cap H,h,\tau_H)(\lambda-\sum_{i=1}^M n_i\phi_i).
$$
\end{proposition}
In other words, on $H^{\geq 0}\cap \Lambda$,
$$\discretepol(\Phi,h,\tau)=\discretepol(\Phi\cap H,h,\tau_H)*f_{\phi_1}^{h_1}*\cdots*f_{\phi_M}^{h_M}.$$
\begin{proof}
As follows from \cite{MR2573189},  the right hand side, being
the convolution of a quasi-polynomial function on the lattice
$\Lambda\cap H$ with products  $f_{\phi_1}^{h_1}*\cdots*f_{\phi_M}^{h_M}$,
coincides with a quasi-polynomial function on the domain $H^{\geq 0}\cap \Lambda$.

Now, to prove that the left hand side coincide with the right hand side,
we will  use the fact  that two quasi-polynomial functions agreeing on $\c\cap \Lambda$,
where $\c$ is a cone with non empty interior, coincide on $\Lambda$.

If $\lambda\in \tau$ is sufficiently near a point of $\overline \tau\cap H$,
then the set  $(\lambda-\sum_{i=1}^M \R_{\geq 0}\phi_i)\cap H$ is contained in $\tau_H$.
We see that the set
$\c:=\{\lambda\in \tau;  (\lambda-\sum_{i=1}^M \R_{\geq 0}\phi_i)\cap H\subset \tau_H\}$
is an open cone in $\tau$. On $\c\cap \Lambda$, the function
    $\discretepol(\Phi,h,\tau)$ coincide with $E(\Phi,h)$. On the other hand,
    if we compute $E(\Phi,h)$ and $E(\Phi\cap H,h)$ by their respective convolution formulae,
    we obtain that the right hand side coincide also with $E(\Phi,h)$ for $\lambda\in \c\cap \Lambda$.
    This establishes the proposition.
\end{proof}

\bigskip

 Let $\mathcal X(\Phi,A,\tau_2)=\geom(\flip_A X(\Phi_{\rm flip}^A,\tau_2))$
 and let $$\discretepol(\Phi,A,h,\tau_2)(\lambda)=\sum_{x\in V(\Phi,\lambda)\cap \Z^N} \mathcal X(\Phi,A,\tau_2)(x).$$
 We then obtain
 $$\discretepol(\Phi,h,\tau_1)-\discretepol(\Phi,h,\tau_2)=\discretepol(\Phi,A,h,\tau_2).$$

Remark that the tope $\tau_2$ for the flipped system $\Phi_{\rm flip}^A$
is such that $\overline \tau_2\cap H$ is on the boundary of the cone $\Phi_{\rm flip}^A$.
Using a slight modification of  Proposition \ref{prop:covol} above,
we then can give the following ``convolution description" of the function
$\discretepol(\Phi,A,h,\tau_2).$

 Let $I^+:=\{a_1,a_2,\ldots,a_p\}$ be the set of indices $i$
 such that $\phi_i$ is on the open half space delimited by $H$ containing $\tau_2$; similarly let
$I^-:=\{b_1,b_2,\ldots,b_p\}$ be the set of indices $j$ such
  that $\phi_j$ is on the open half space delimited by $H$ containing $\tau_1$;
 then the sequence
 $$[\phi_{a_1},\ldots,\phi_{a_p},-\phi_{b_1},\ldots, -\phi_{b_q}]
 $$
 is contained in the open half space delimited by $H$ and containing $\tau_2$.

 Define
 $$\flip f_{\phi}^h=-\sum_{n=1}^{\infty} h(-n)\delta_{-n \phi}.
 $$
Then we have
\begin{proposition}
The quasi-polynomial function  $\discretepol(\Phi,A,h,\tau_2)$ is given by the convolution formula:
\begin{multline*}
\discretepol(\Phi,A,h,\tau_2)=\\
\discretepol(\Phi\cap H,h,\tau_{1,2})*\left(\prod_{a\in A}
\flip f_{\phi_a}^{h_a}*\prod_{b\in B} f_{\phi_b}^{h_b}-
\prod_{b\in B} \flip f_{\phi_b}^{h_b}*\prod_{a\in A} f_{\phi_a}^{h_a}\right).
\end{multline*}
\end{proposition}
This is the convolution formula given by Paradan \cite{paradan2004} for the jump.
It expresses the jump in terms of sums of the function
$\discretepol(\Phi\cap H,h,\tau_{1,2})$ associated to a lower dimensional system (see also \cite{MR2573189}).

\section{A refinement of Brion's theorem}\label{section degenerate}
Let $\p\subset V$ be a full-dimensional polytope in a
vector space $V$ provided with a lattice $V_\Z$.
 Recall  Brion's Formula (\ref{eq:brion-generating-function})
  for the  generating function
of a polytope.
$$
S(\p)(\xi)=\sum_{s\in \CV(\p)}S(s+\c_s)(\xi).
$$
As $\p$ is compact, the  function  $\xi\in V^* \mapsto S(\p)(\xi)$
is holomorphic, but the contribution of each cone  is a meromorphic
function with singularities along hyperplanes. More precisely, an
element $\xi \in V^*$ is singular for $S(s+\c_s)$ if and only if
$\xi$ is constant on some face $\f$ of $\p$ such that $s$ is a vertex
of $\f$ and $\dim \f>0$.

It is well known that Brion's formula  is the
combinatorial translation of  the localization formula in
equivariant cohomology, in the case of isolated fixed points. In
this section, we
 generalize (\ref{eq:brion-generating-function}) to the combinatorial
case which corresponds to non isolated fixed points \cite{MR685019}. In this
degenerate case, the connected components of the set of fixed points
correspond to the  faces of $\p$ on which $\xi$ is constant which
are maximal with respect to this property. The contribution of such
a face to the sum $S(\p)(\xi)$ is
$$
\sum_{s\in \CV(\f)}S(s+\c_s)(\xi).
$$
We will study this sum by relating it to a Brianchon-Gram
continuation of the face $\f$. We will assume that the polytope $\p$
is simple. The general case needs more efforts.

We need to introduce some  meromorphic functions similar to  the function $S(s+\c)(\xi)$.
Let $\q=s+\c$ be a   polyhedral cone in $V$, where $\c$ is a cone  generated by elements $g_j\in V_\Z$ .
Let $P$ be a quasi polynomial function on $V_\Z$.
The following sum $\sum_{x\in V_\Z\cap \q}P(\xi)e^{\ll \xi,x\rr}$
defines a generalized function  $F$ of the variable $\xi\in i V^*$.

It is easy to see that  $\prod_{i}(1-e^{\ll\xi,g_j\rr})F(\xi)$ is an analytic function of $\xi$.
 Thus, outside the affine hyperplanes in $iV^*$ defined by $\ll \xi,g_j\rr \in 2i\pi \Z$,
 the generalized function $F(\xi)$ is equal to $S(\q,P)(\xi)$, where $S(\q,P)(\xi)$
 is a meromorphic function of $\xi$ with poles on $\ll g_j,\xi\rr \in 2i\pi \Z$.
 In particular this function belongs  to the space $\CM_{\ell}(V^*)$  introduced before.
 We write
 $$S(\q,P)(\xi)=\sum_{x\in V_\Z\cap\q}P(\xi)e^{\ll\xi,x\rr}$$
   and depending on the context, we consider $S(\q,P)$ either as a generalized function
   of $\xi\in iV^*$ or as a meromorphic function of $\xi\in V_\C$.
If $\q$ is a cone  invariant by translation by a vector $v\in V_\Z$ ,
it is easy that the generalized function $S(\q,P)(\xi)$ is annihilated  by a power of $(1-e^{\ll v,\xi\rr})$.
The simplest case is when $P=1$, $V=\c=\R$, $V_\Z=\Z$, so that the equality is simply
$(1-e^{i\theta})\sum_{n\in \Z}e^{in \theta}=0$.
In particular, if $\q$ is a flat cone, the meromorphic function $S(\q,P)(\xi)$ is equal to $0$.

If $\f$ is a face of $\p$, we denote by
$\aff(\f)$ the affine space generated by $\f$ and by $\lin \f$ the linear space parallel to $\aff(\f)$, that is the space spanned by elements $x-y$ with $x,y\in \f$.
The projection $\t_{\rm trans}(\p,\f)$ of $\t_{\aff}(\p,\f)$ in $V/\lin \f$ is called the transverse cone. Note that this transverse cone is a salient cone in $V/\lin \f$ with vertex $y_0$ the projection of any $y\in \f$.

\begin{theorem}\label{th:Brion_generalise}
Let $V$ be a rational vector space with lattice $V_\Z$.  Let
$\p\subset V$ be a simple rational polytope and let $\f$ be a face
of $\p$. Let $\t_{\rm trans}(\p,f)\subset V/\lin\f$ be the transverse cone.
The tangent cone to $\p$ at the vertex $s$ is denoted by $s+\c_s$. For
$\xi\in V^*$, let
$$
S(s+\c_s)(\xi)=\sum_{x\in (s+\c_s)\cap V_\Z}e^{\langle \xi,x
\rangle}.
$$
The set of vertices of $\f$ is denoted by $\CV(\f)$.

\noindent (i) The sum $\sum_{s\in \CV(\f)}S(s+\c_s)(\xi)$ restricts
to a meromorphic function on $\lin\f^\perp\subset V^*$, which is
given by
$$
\sum_{s\in \CV(\f)}S(s+\c_s)(\xi)=\sum_{y\in\t_{\rm trans}(\p,\f)\cap
{(V/\lin\f)_\Z}}e^{\langle\xi,y\rangle}P(y),
$$
 where $P(y)$ is a quasi-polynomial function  on the
projected lattice \\
$(V/\lin\f)_\Z\subset V/\lin\f$. Moreover if $\xi$
is regular with respect to the cone $\t_{\rm trans}(\p,\f)$, that is if $\xi$ is
not constant on a face strictly containing $\f$, then $ \sum_{s\in
\CV(\f)}S(s+\c_s)(\xi)$ is holomorphic at $\xi$.

\noindent (ii)  For $y$ close enough to the vertex $y_0$ of the
transverse cone, $P(y)$ is the number of lattice points of the slice
$\p\cap (\lin\f+y)$.
\end{theorem}
\begin{proof}
We compute  the signed sum of the generating functions of
the tangent cones $\t_{\aff}(\p,\g)$ where $\g$ runs over the set
$\CF(\f)\subset \CF(\p)$  of  faces of $\f$. Since
$\t_{\aff}(\p,\g)$ contains lines if $\g$ is not a vertex, we have
\begin{equation}\label{eq:trans}
\sum_{s\in \CV(\f)}S(s+\c_s)(\xi)=\sum_{\g\in \CF(\f)}(-1)^{\dim
\g}S(\t_{\aff}(\p,\g))(\xi).
\end{equation}

We will relate the right hand side  to sums over slices of $\p$ by affine subspaces parallel to $\f$.

We define
\begin{equation}\label{eq:defTy}
\CT(y)(x)=\sum_{\g\in \CF(\f)}(-1)^{\dim
\g}\suppresschi{[\t_{\rm aff}(\p,\g)\cap(\aff(\f)+y)]}(x).
\end{equation}

The support of $\CT(y)$ is illustrated in
Fig.\ref{fig:Brianchon-Gram continuation of a face}.

Let us only observe that, as $\t_{\aff}(\p,\g)\cap \aff(\f)$ is the tangent cone of the polytope
$\f\subset \aff(\f)$ along  its face  $\g$, we have,  by
Brianchon-Gram theorem,
\begin{eqnarray*}
 \CT(0) = \suppresschi{[\f]}.
\end{eqnarray*}
Moreover,  if $y$  is small enough, then $\CT(y)$ is the characteristic
function of the intersection $\p\cap (\aff(\f)+y)$.
 This result can be deduced from the Euler relations.
 In the next section, in the case where   $\p$ is
 simple, we will obtain it as a consequence of Corollary
 \ref{th:continuity-on-closed-tope} which is of course
itself based on  the Euler relations  via the Brianchon-Gram
theorem.

 Let us compute the right hand side of Equation (\ref{eq:trans}).
If $\xi\in \lin\f^\perp$ then $e^{\langle\xi,x\rangle}$ is constant
on $\lin\f+y$. Identifying $\lin\f^\perp$ with $(V/\lin\f)^*$, we
denote this constant value by  $e^{\langle\xi,y\rangle}$.

Thus, we slice the lattice $V_\Z$ in slices parallel to the subspace
$\lin\f$. The slices are indexed by the projected lattice
$(V/\lin\f)_\Z$. We write
\begin{multline}\label{eq:slicing}
 \sum_{\g\in \CF(\f)}(-1)^{\dim
\g}S(\t_{\aff}(\p,\g))(\xi)= \sum_{\g\in \CF(\f)}(-1)^{\dim
\g}\sum_{x\in
\t_{\aff}(\p,\g)}e^{\langle\xi,x\rangle}\\
=\sum_{y\in (V/\lin\f)_\Z} \sum_{\g\in \CF(\f)}(-1)^{\dim
\g}\sum_{x\in \t_{\aff}(\p,\g)\cap (\lin\f+y)\cap
V_\Z}e^{\langle\xi,x\rangle}.
\end{multline}
 Let $y_0\in
V/\lin\f$ be the projection of the face $\f$. From
(\ref{eq:slicing}), we obtain
\begin{equation}\label{eq:useTy}
\sum_{s\in \CV(\f)}S(s+\c_s)(\xi)=\sum_{y\in (V/\lin\f)_\Z}
e^{\langle\xi,y\rangle} \sum_{x\in V_\Z}\CT(y-y_0)(x).
\end{equation}
The shift by $y_0$ is there to make further notations simpler. At
this point, we postpone the proof of Theorem
\ref{th:Brion_generalise} until the next section, where  we will
relate $\CT(y)$ to the
 Brianchon-Gram continuation of the face $\f$, under the assumption that $\p$ is simple.

\subsection{Brianchon-Gram continuation of a face of a
partition polytope}  Let  $\p=\p(\Phi,\lambda)\subset \R^N$.
 Let $\f$ be a face of $\p$.
 If $\lambda$ is regular and belongs to a tope $\tau$,
 then there is a unique   $I\in\CG(\Phi,\tau)$ such that   $\f=\f(\Phi,\lambda ,I)$ is  the
corresponding face. We have $\dim \f=|I|-\dim F$. If $\lambda$ is on
a wall, there may be several such pairs $(\tau,I)$.

\begin{definition}\label{def:BG_continuation_of_face} Let $I\subset \{1,\dots,N\}$
be such that the sequence $\Phi_I$ generates $F$.
Let
$\widetilde{\Phi}_I=(\widetilde{\phi}_i), 1\leq i \leq N$, be the
sequence of elements in $F\oplus \R^{I^c}$ defined by
$\widetilde{\phi}_i=\phi_i$ if $i\in I$ and $\widetilde{\phi}_i=
\phi_i\oplus e_i$, if $i\in I^c$.
\end{definition}
\begin{lemma}\label{phi_tilde_lemma}
\noindent(i)  The sequence $\widetilde{\Phi}_I$ generates a salient
cone of full dimension in $F\oplus \R^{I^c}$.

\noindent(ii)  $ V(\widetilde{\Phi}_I,(\lambda,y))=\{x\in\R^N; \;
\sum_{i=1}^N x_i\phi_i=\lambda, \; x_i=y_i \mbox{  for  } i\in
I^c\}$.
\noindent(iii) Let $\tau$ be a $\Phi_I$-tope. Let  $R$ be an open
quadrant in $\R^{I^c}$. Then $ \{(\lambda,y)\in F\oplus \R^{I^c};\;
y\in R, \lambda -\sum_{i\in I^c}y_i\phi_i \in \tau\} $ is a
$\widetilde{\Phi}_I$-tope and all $\widetilde{\Phi}_I$-topes are of
this form.

\noindent(iv)   Let $\tau$ be a $\Phi_I$-tope, let
$\tau_I$ be the $\widetilde{\Phi}_I$-tope which consists
of $(\lambda,y)$ such that $y_i>0$ for $i\in I^c$ and $\lambda
-\sum_{i\in I^c}y_i\phi_i \in \tau$. Then
$\CG(\widetilde{\Phi}_I,\tau_I)$ is the set of $K\cup
I^c$, where $K\subseteq I $ and $K\in \CG(\Phi_I,\tau)$. Hence
$$
X(\widetilde{\Phi}_I,\tau_I)=\sum_{K\in
\CG(\Phi_I,\tau)}(-1)^{|K|-\dim F}\prod_{i\in I\setminus K}p_i
\prod_{i\in K\cup I^c}(p_i+q_i).
$$
\end{lemma}
\begin{proof}
(i) follows from the fact that $\Phi_I$ generates $F$. (ii) is
immediate.  Consider the linear bijection from $F\oplus \R^{I^c}$ to
itself defined by $(\lambda,y)\mapsto (\lambda -\sum_{i\in
I^c}y_i\phi_i ,y)$. The image of  $\widetilde{\phi}_i$ is $\phi_i$
if $i\in I$, and $ e_i$ if $i\in I^c$. Therefore the
$\widetilde{\Phi}_I$-topes are the pull-backs of the topes relative
to the sequence $\psi_i=\phi_i$ if $i\in I$ and $\psi_i=e_i$ if
$i\in I^c$. The latter are the products of $\Phi_I$-topes in $F$
with the quadrants in $\R^{I^c}$. This proves (iii).

Let  $\widetilde{K}\subseteq \{1,\dots,N\}$. Then
$\widetilde{\Phi}_{\widetilde{K}}$ generates $F\oplus \R^{I^c}$ if
and only if $\widetilde{K}=K\cup I^c$, where $K\subseteq I$ is such
that $\Phi_K$ generates $F$. Moreover $\tau_I \subset
c((\widetilde{\Phi}_I)_{\widetilde{K}}) $ if and only if $\tau
\subset \c(\Phi_K)$, whence (iv).
\end{proof}

\begin{proposition}\label{slice}
Let $\tau$ be a $\Phi$-tope and let $\lambda\in \overline{\tau}$.
 Let $\p=\p(\Phi,\lambda)$ and $\f=\f_I(\Phi,\lambda)$ be a
face of $\p$. Assume that $\dim \f=|I|-\dim F$.  We identify the
quotient space  $V/\lin\f$ with $\R^{I^c}$ by the projection
parallel to $\R^I$.  Let $\tau_I$ be the $\Phi_I$-tope which
contains $\tau$. If $y_i\geq 0$ for $i\in I^c$ and  $\lambda
-\sum_{i\in I^c}y_i\phi_i \in \overline{\tau_I}$, then
\begin{equation}\label{eq:p_cap_parallel_to_face}
\varchenko(\widetilde{\Phi}_I,\tau_I)\suppresschi{[V(\widetilde{\Phi}_I,(\lambda,y))]}
=\suppresschi{[\p(\widetilde{\Phi}_I,(\lambda,y))]}=\suppresschi{[\p(\Phi,\lambda)\cap
(\aff(\f)+y)]}.
\end{equation}
In particular, if $\lambda$ is regular, the conditions $y_i\geq 0$
for $i\in I^c$ and  $\lambda -\sum_{i\in I^c}y_i\phi_i \in
\overline{\tau_I}$ define  a neighborhood of $y=0$ in $\R_{\geq
0}^{I^c}$ on which (\ref{eq:p_cap_parallel_to_face}) holds.
\end{proposition}
\begin{proof}
The conditions $y_i\geq 0$ for $i\in I^c$ and  $\lambda -\sum_{i\in
I^c}y_i\phi_i \in \overline{\tau_I}$ mean that $(\lambda,y)$ belongs
to the closure of the $\widetilde{\Phi}_I$-tope $\tau_I$
associated to $\tau_I$. Therefore by Corollary
\ref{th:continuity-on-closed-tope}, we have
$$
\varchenko(\widetilde{\Phi}_I,\tau_I)\suppresschi{[V(\widetilde{\Phi}_I,(\lambda,y))]}
=\suppresschi{[\p(\widetilde{\Phi}_I,(\lambda,y))]}.
$$
Moreover, as $\dim \f=|I|-\dim F$, the affine span $\aff(\f)$ is
given by $ \aff(\f)= \{x\in V(\Phi,\lambda); x_i=0 \mbox{  for }
i\in I^c\}$. It follows that
$V(\widetilde{\Phi}_I,(\lambda,y))=\aff(\f)+y$, hence $
\p(\widetilde{\Phi}_I,(\lambda,y))=\p(\Phi,\lambda)\cap
(\aff(\f)+y)$.
\end{proof}
\begin{remark}\label{subdivision_transverse_cone}
Define $\q_0(\p,\f,\tau)\subseteq \R_{\geq 0}^{I^c}$ by
$$
\q_0(\p,\f,\tau)= \{y=(y_i)\in \R^{I^c}\; ; y_i\geq 0 \mbox{ for }i\in I^c, \lambda
-\sum_{i\in I^c}y_i\phi_i \in \overline{\tau_I}\}.
$$
The set $\q_0(\p,\f,\tau)$ is a polytope in $V/\lin\f\simeq
\R^{I^c}$. Let us denote its cone  at vertex $0$ by
$\t_0(\p,\f,\tau)$.
\begin{multline*}
\t_0(\p,\f,\tau)=\{y=(y_i)\in \R^{I^c}\; ;y_i\geq 0 \mbox{ for } i\in I^c, \\
\lambda
-\epsilon \sum_{i\in I^c}y_i\phi_i \in \overline{\tau_I}\mbox{ for }
\epsilon> 0  \mbox{ small enough }\}.
\end{multline*}
 Then $\t_0(\p,\f,\tau)$ is a subcone of
the transverse cone $\t_0(\p,\f)$.
If $\lambda\in \tau$ is regular, then $\t_0(\p,\f,\tau)= \t_0(\p,f)=\R_{\geq
0}^{I^c}$ .

 If $\lambda $ lies on a wall of a tope ${\tau}$, then $\t_0(\p,\f,\tau)$ may
 be
 strictly contained  in the transverse cone $\t_0(\p,\f)$. When we consider all  the topes $\tau$ such that $\lambda\in
 \overline{\tau}$, the cones
 $\t_0(\p,\f,\tau')$  form a subdivision of $\t_0(\p,f)$. An example is illustrated in  Fig. \ref{fig:tipi}.
The polytope $\p\subset \R^3$ is a tipi with four poles, with vertices
$(0,0,0), (1,0,0), (0,1,0), (0,0,1), (1,1,0)$ and $\f$ is  the
vertical edge with vertices $(0,0,0),(0,0,1)$.
The picture shows also the corresponding system $\Phi$ such that
$\p$ corresponds to a partition polytope $\p(\Phi,\lambda)$.
In this case, $\lambda$ belongs to the wall generated by $\phi_3$,
thus $\lambda$ belongs to two tope closures $\tau_1$ and $\tau_2$. We identify the quotient
$V/\lin(\f)$ with the ground.
Then the  sets  $\q_0(\p,\f,\tau_i)$ are the two triangles which subdivide the ground face of the tipi.

This remark suggests how to modify Proposition \ref{th:face-contribution_simple}  in the case of a non simple polytope.
\end{remark}
\begin{figure}[!h]
\begin{center}
  \includegraphics[width=1.5 in]{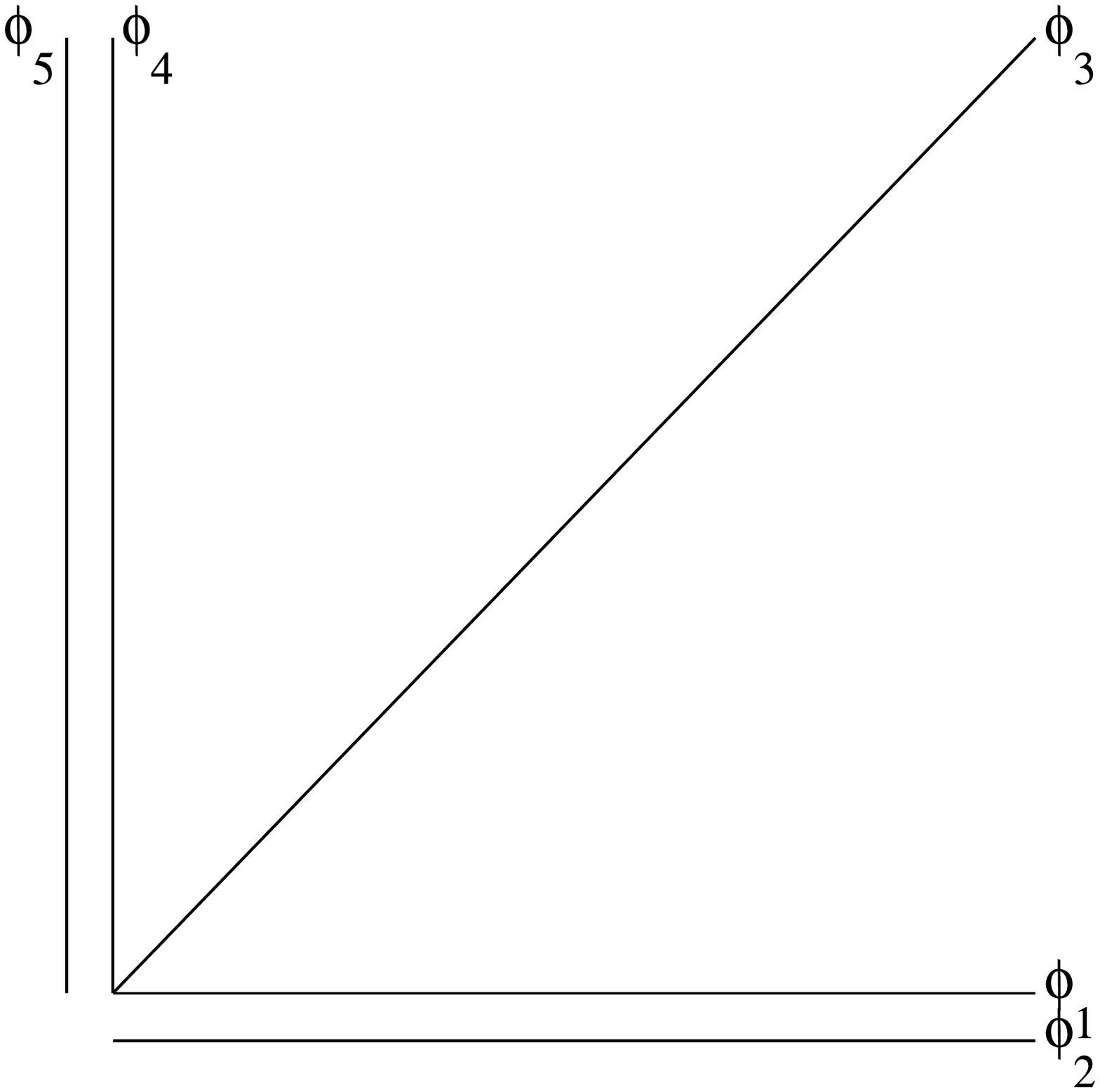}
  \includegraphics[width=2 in]{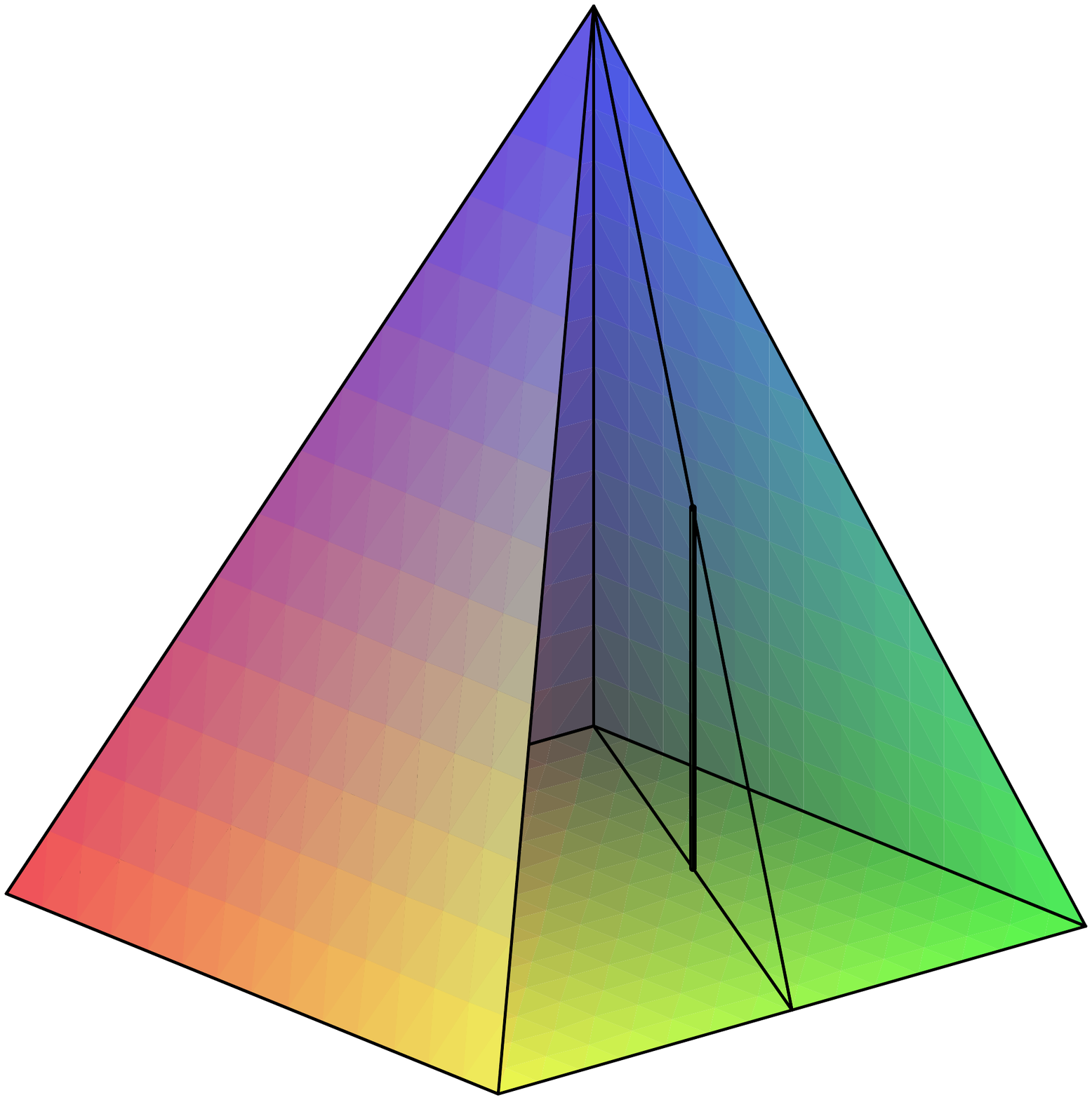}
  \caption{}
  \label{fig:tipi}
  \end{center}
\end{figure}

\begin{figure}[!h]
\begin{center}
  \includegraphics[width=2 in]{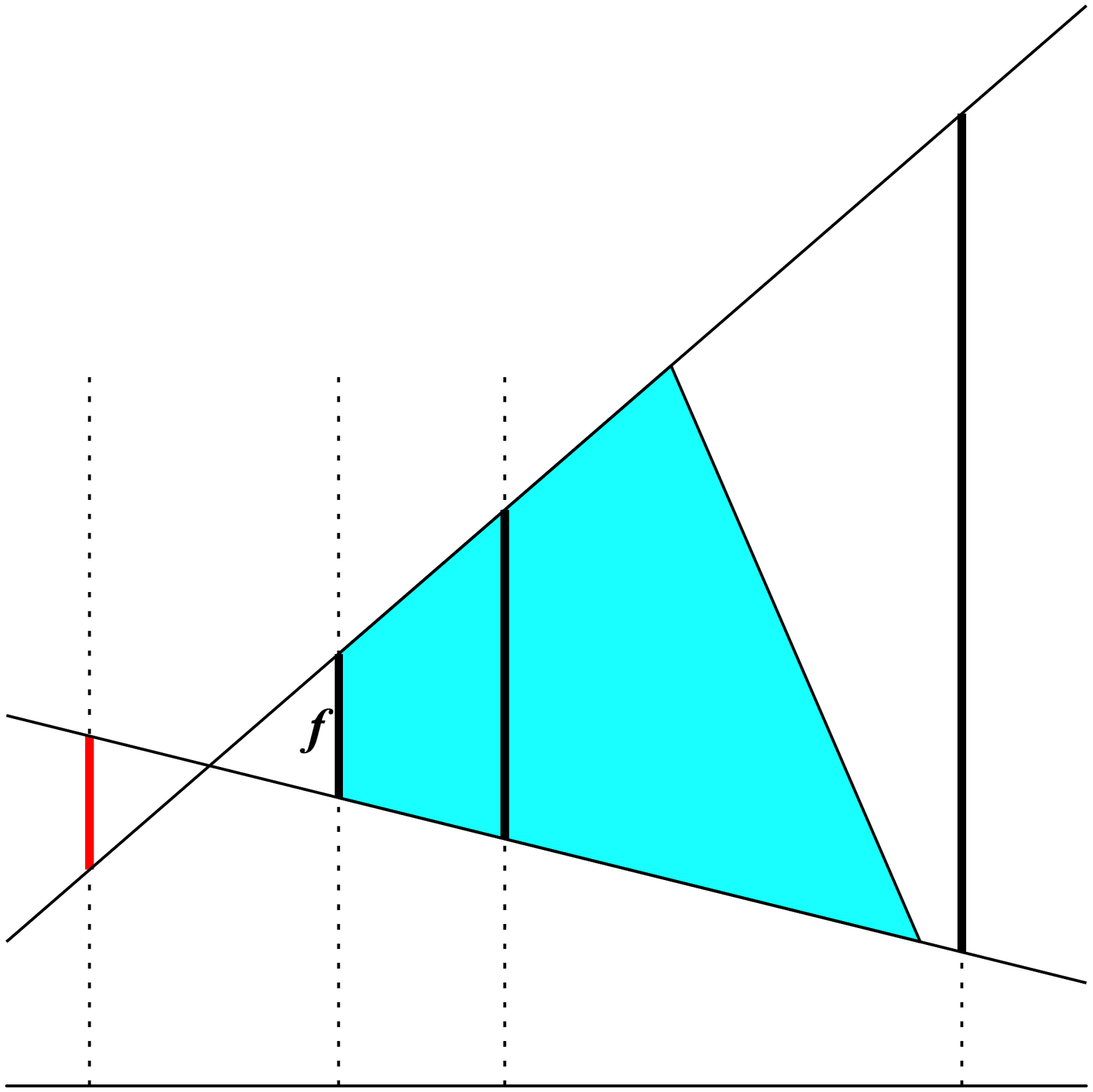}
\includegraphics[width=2.5 in]{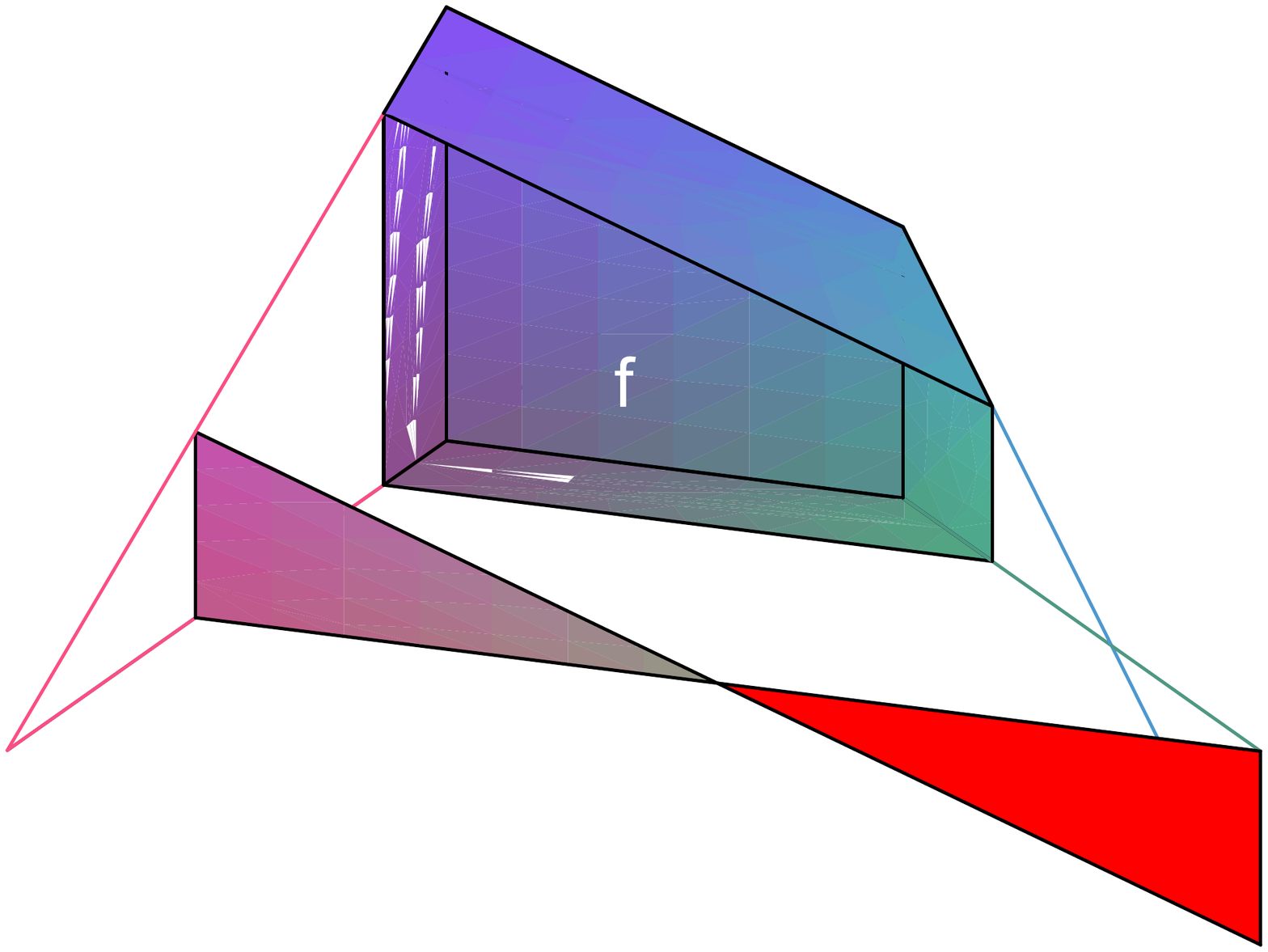}
  \caption{Brianchon-Gram continuation of a face. The segment and the triangle in red  come with a minus sign.
  The end-points of the segment have to be deleted and two edges of the triangle also. }
  \label{fig:Brianchon-Gram continuation of a face}
  \end{center}
\end{figure}
\begin{proposition}\label{th:face-contribution_simple}
 Let $\Phi=(\phi_j)_{1\leq j\leq N} $ be a sequence of non zero elements
of a vector space $F$, generating F, and  spanning a salient cone.
 Let $\tau$ be a $\Phi$-tope,  $ \lambda\in {\tau}$ a regular element
and  $I\in \CG(\Phi,\tau)$.   Let $\p=\p(\Phi,\lambda)$ and
$\f=\f(\Phi,\lambda,I)$.

We identify the quotient space  $V/\lin\f$ with $\R^{I^c}$
by the projection parallel to $\R^I$.  For $y\in \R^{I^c}$, let
\begin{equation}\label{cones-at-vertices-of-a-face}
\CT(y)=\sum_{\g\in \CF(\f)}(-1)^{\dim
\g}\suppresschi{[\t_{\rm aff}(\p,\g)\cap(\aff(\f)+{y})]},
\end{equation}
where the set of faces of $\f$ is denoted by $\CF(\f)$ .

Let $\widetilde{\Phi}_I=(\widetilde{\phi}_i), 1\leq i \leq N$, be
the sequence of elements in $F\oplus \R^{I^c}$ defined by
$\widetilde{\phi}_i=\phi_i$ if $i\in I$ and $\widetilde{\phi}_i=
\phi_i\oplus e_i$, if $i\in I^c$. Let $\tau_I$ be the
$\widetilde{\Phi}_I$-tope which consists of elements $(\lambda,y)\in
F\oplus \R^{I^c}$ such that $y_i>0$ for $i\in I^c$ and $\lambda
-\sum_{i\in I^c}y_i\phi_i \in \tau_I$, where $\tau_I$ is the unique
$\Phi_I$-tope which contains $\tau$.  Then
\begin{equation}\label{eq:face-contribution}
\CT(y)(x)= \varchenko(\widetilde{\Phi}_I,\tau_I)(x)
\suppresschi{[V(\widetilde{\Phi}_I,(\lambda,y))]}(x)\prod_{i\in
I^c}\suppresschi{[y_i\geq 0]},
\end{equation}
\end{proposition}

\begin{proof}
The faces $\g$ of $\f=\f(\Phi,\lambda,I)$ are indexed by the subsets
$K\in \CG(\Phi,\tau)$ which are contained in $I$. For
$\g=\f(\Phi,\lambda,K)$,  we have
$$
\t_{\aff}(\p,\g)=\{x\in V(\Phi,\lambda); x_i\geq 0 \mbox{  for  }
i\in K^c\}.
$$
We write (\ref{cones-at-vertices-of-a-face}) as
\begin{equation}\label{cones-at-vertices-of-a-face_simple}
\CT(y)= \sum_{K\in  \CG(\Phi,\tau),K\subseteq I}(-1)^{|K|-\dim
F}\prod_{i\in K^c}\suppresschi{[x_i\geq 0]}\; \suppresschi{[\aff(\f)+{y}]}.
\end{equation}
We observe that $\CG({\Phi}_I,{\tau}_I)=\{K\in
\CG(\Phi,\tau),K\subseteq I\}$.  Therefore, by Lemma
\ref{phi_tilde_lemma}, we have
\begin{multline*}
\varchenko(\widetilde{\Phi}_I,\tau_I)\suppresschi{[V(\widetilde{\Phi}_I,(\lambda,y))]}=\\
\sum_{K\in \CG(\Phi,\tau),K\subseteq I}(-1)^{|{K}|-\dim
{F}}\prod_{i\in I\setminus K}\suppresschi{[x_i\geq
0]}\suppresschi{[V(\widetilde{\Phi}_I,(\lambda,y))]}.
\end{multline*}
We factor out $\prod_{i\in I^c}\suppresschi{[x_i\geq 0]}$ in each summand
of (\ref{cones-at-vertices-of-a-face_simple}). \\
As
$V(\widetilde{\Phi}_I,(\lambda,y))=\aff(\f)+y$, we obtain
$$
\CT(y)(x)= \varchenko(\widetilde{\Phi}_I,\tau_I)(x)
\suppresschi{[V(\widetilde{\Phi}_I,(\lambda,y))]}(x)\prod_{i\in
I^c}\suppresschi{[x_i\geq 0]}.
$$
As  $x_i=y_i$ for $i\in I^c$ if $x\in
V(\widetilde{\Phi}_I,(\lambda,y))$,  we obtain
(\ref{eq:face-contribution})
\end{proof}
We resume the proof of Theorem \ref{th:Brion_generalise}.
\begin{proof}[Proof of Theorem \ref{th:Brion_generalise}]

We identify $\p$ with a partition polytope $\p(\Phi,\lambda)$ by an
affine map $V\simeq V(\Phi,\lambda)$. We can assume that $\lambda$
is regular. Some care is needed with respect to the lattice
 $V_\Z$.
In general, its image $V(\Phi,\lambda)_\Z$ in $ V(\Phi,\lambda)$ is
not $\Z^N \cap V(\Phi,\lambda)$. However we can always write
$V(\Phi,\lambda)_\Z = (b + \Gamma)\cap V(\Phi,\lambda)$ , where
$b\in \Q^N$ and $\Gamma$ is a lattice in $\R^N$, ($\Gamma$ is
a fixed lattice and  $b$ projects on $\lambda$). Let $\tau$ be the
$\Phi$-tope which contains $\lambda$ and let $I\in \CG(\Phi,\tau)$
such that $\f$ is identified with the face $\f(\Phi,\lambda,I)$.
Then $V/\lin\f$ is identified with $V/\lin\f\simeq \R^{I^c}$
and the projected lattice $(V/\lin\f)_\Z$ is identified with a
lattice in $\R^{I^c}$. By Proposition
\ref{th:face-contribution_simple}, we have, for every $x\in \R^N$,
$$
\CT(y)(x)=\varchenko(\widetilde{\Phi}_I,\tau_I)\suppresschi{[V(\widetilde{\Phi}_I,(\lambda,y))]}(x)\prod_{i\in
I^c}\suppresschi{[y_i\geq 0]}.
$$
So we define
\begin{equation*}\label{eq:brion_generalise_Py-2}
P(y)=\sum_{x\in b+\Gamma}
\varchenko(\widetilde{\Phi}_I,\tau_I)\suppresschi{[V(\widetilde{\Phi}_I,(\lambda,y-y_0))]}(x).
\end{equation*}
Then $P(y)$  is a quasi-polynomial function of $y\in (V/\lin\f)_\Z$.
This fact follows from a minor generalization of Theorem
\ref{th:polynom_integral_and_discrete} (ii). We only have to take
care of the shifts: the summation is over $x\in b+\Gamma$ and the
parameter $y-y_0$ in the Brianchon-Gram function  runs over the
shifted lattice $(V/\lin\f )_\Z -y_0$.

The   equalities (\ref{eq:slicing}) and
(\ref{eq:useTy}) of generalized functions imply   equalities of holomorphic functions
of $\xi$ in an open subset of $(\lin\f)^\perp$, hence $\sum_{s\in
\CV(\f)}S(s+\c_s)(\xi)$ restricts to a meromorphic function on
$(\lin\f)^\perp$, given by
\begin{equation}\label{eq:brion_generalise_Py-1}
\sum_{s\in \CV(\f)}S(s+\c_s)(\xi)=\sum_{y\in \t_{\rm trans}(\p,\f)\cap
(V/\lin\f)_\Z} e^{\langle\xi,y\rangle}P(y).
\end{equation}
So we have proved (i).

By Proposition \ref{slice}, for $y\in \t_{\rm trans}(\p,\f)$ close to the
vertex, $\CT(y-y_0)$ is the characteristic function of the slice $\p\cap
(\aff(\f) +y)$, hence (ii).
\end{proof}\end{proof}

\section{Cohomology of line bundles over a toric variety}
Let us indicate  the relation of  our work   with toric varieties.
 Let $\Phi=(\phi_j)_{1\leq j\leq N} $ be a sequence of non zero elements
of a vector space $F$, generating F, and  spanning a salient cone.
Assume that the $\phi_i$'s  belong to a lattice $\Lambda$
and let $T$ be the torus with character group $\Lambda$
embedded in $T^N=S_1^N$ by the characters of $T$ associated to $(\phi_i)$.
This determines an action of $T$ in the complex space $\C^N$.
Each tope $\tau$ determines a toric variety $M_\tau$  (with orbifold singularities)
for the quotient torus $T^N/T$, in the following way.
If $\lambda\in \tau$, $M_\tau$ is the reduced manifold  $\C^N//_{\lambda}T$ at $\lambda\in \t^*$.
Then the vectors $\phi_i$ parameterize the boundary divisors $D_i$ in $M_\tau$ and
 each element $\lambda\in \Lambda$ determines a $T^N$-equivariant sheaf
 $\mathcal O(\lambda)$ on $M_\tau$.

 The lattice of characters of the $d$-dimensional  torus $T^N/T$ is identified with $V\cap \Z^N$.

 The torus $T^N$ acts on the cohomology groups $H^{i}(M_\tau,\mathcal O(\lambda))$.
When $\lambda\in \tau$,
then all the cohomology groups  $H^i$ for $i>0$ vanish,
and a weight $m\in \Z^N$ of $T^N$ occurs in $H^0(M_\tau,\CO(\lambda))$
if and only if $m\in \p(\Phi,\lambda)\cap \Z^N$.
Thus the dimension of the space $H^0(M_\tau,\CO(\lambda))$
is just the number of integral points in $\p(\Phi,\lambda)$.

If  $\lambda\in \Lambda$ does not  belong to the tope $\tau$, and $i>0$,
 the  cohomology space
$H^i(M_\tau,\CO(\lambda))$ is in general not zero.
It is natural to introduce the virtual space
 $$
 {\mathcal H}(\tau,\lambda):=\sum_{i=0}^d (-1)^i H^i(M_\tau,\CO(\lambda)).
 $$
It follows from the Kawasaki-Riemann-Roch theorem that the
 virtual dimension of ${\mathcal H}(\tau,\lambda)$ is a quasi polynomial function of $\lambda$.

More precisely,  we can use the fixed point theorem to compute
 the character of $T^N$ in ${\mathcal H}(\tau,\lambda)$ (see  \cite{brion-vergne-toric-1997}).
As the  construction of the present article reproduces  this fixed point theorem at the level of sets,
we obtain the weight decomposition of the $T^N$-module
 $$
 {\mathcal H}(\tau,\lambda)=\sum_{m\in \Z^N\cap V(\Phi,\lambda)} \varchenko(\Phi,\tau)(m) e^m.
 $$
In other words, the function
$\varchenko(\Phi,\tau)$ on $\Z^N$ computes simultaneously
(for all sheaves $\mathcal O(\lambda)$)
the multiplicity of a weight $m$
in the alternate sum of cohomology spaces.
In particular,  the function
 $\varchenko(\Phi,\tau)\cap [V(\Phi,\lambda)]$ is the constructible function on  $V(\Phi,\lambda)$ associated by Morelli \cite{MR1234308} to the sheaf $\mathcal O(\lambda)$.

Recall the formula
\begin{equation*}
\varchenko(\Phi,\tau)=\sum_B z(\Phi,\tau,B) [Q_{\rm neg}^B].
 \end{equation*}
Let us comment on the explicit computation of   the coefficients $z(\Phi,\tau,B)$ of $\varchenko(\Phi,\tau)$.
We wrote a brute force Maple program to compute $X(\Phi,\tau)$,
out of  its definition (Equation (\ref{eq:geometric_brianchon_gram_})),
by enumerating the generating subsets of the system $\Phi$
and checking which ones are in $\mathcal G(\Phi,\tau).$
It would be certainly more efficient to  use Theorem \ref{th:polarized-sum},
and then determine $\mathcal B(\Phi,\tau)$ using the reverse-search algorithm of Avis-Fukuda \cite{MR1174359}.
Anyway,  we obtain the decomposition as a sum of monomials
$$
X(\Phi,\tau)=\sum_B z(\Phi,\tau,B) \prod_{i\in B^c}p_i\prod_{j\in B}q_j.
$$
If $m\in \Z^N$, we denote by $B_m$  the set of indices $i$ such that  $m_i<0$.
Then the multiplicity of $m$ in the $T^N$ module ${\mathcal H}(\tau,\lambda)$ is obtained by
computing the coefficient $z(\Phi,\tau,B_m)$ of the monomial
$\prod_{i\in B_m^c}p_i \prod_{i\in B_m}q_i $ in  $X(\Phi,\tau)$.

Y. Karshon and S. Tolman \cite{karshon-tolman-presymplectic-1993}
have studied the representation space ${\mathcal H}(\tau,\lambda)$
associated to a non-ample line bundle on the manifold $M_\tau$,
and they have given an algorithm to compute  a weight in this representation space
by wall crossing.
Our algorithm (Theorem \ref{th:salient})
to determine $z(\Phi,\tau,B)$ is probably very similar.
However, as we deal with arbitrary ``weights" $\phi_i$ (not assumed rational),
our methods use ``only  linear algebra",  not geometry.

By summing up the multiplicities of the  weights in $\mathcal H(\tau,\lambda)$, we obtain the expression of the function
$$
\lambda\mapsto  \dim {\mathcal H}(\tau,\lambda)=\sum_B z(\Phi,\tau,B) {\rm cardinal} (\p_{\rm flip}(\Phi,B,\lambda)\cap \Z^N)
$$
as a sum of partition functions with respect to particular flipped systems.

Remark that if $\lambda\in (\tau-\b(\Phi))\cap \Lambda$,
the continuity property asserts that
the dimension of ${\mathcal H}(\tau,\lambda)$ is still equal to the dimension of $H^0$,
that is the cardinal of  $\p(\Phi,\lambda)\cap \Z^N$.
This is in accordance with the following vanishing theorem
\cite{mustata-vanishing-toric-2002}.
\begin{theorem}
If  $\lambda\in (\tau-\b(\Phi))\cap \Lambda$ then $H^i(M_\tau,\CO(\lambda))=0$ for  $i>0$.
\end{theorem}
It would be interesting to study  the locally quasi  polynomial
function $h_i(\Phi,\tau)(\lambda)=\dim H^i(M_\tau,\mathcal O(\lambda))$ for each $i$.
>From Demazure's description of the individual cohomology groups
$H^i(M_\tau,\mathcal O(\lambda))$
(see for example the  forthcoming book \cite{Cox-Little-Schenck},
Chapter 9),
 we see that it is a locally quasi-polynomial function, sum of partition functions of flipped systems.
Thus each  locally  quasi polynomial function $h_i(\Phi,\tau)$
is a particular element of the generalized Dahmen-Micchelli space $\mathcal F(\Phi)$ introduced in \cite{deconcini-procesi-vergne-dahmen-micchelli-2008}.
It would be interesting to study the relations between these different  locally quasi polynomial functions on $\Lambda$.

Let us give a last example to illustrate the method.
We consider the hexagon defined by the following  inequalities in $\R^2$.
$x_1\geq 0,\, x_1\leq 2, \,x_2\geq 0,\, x_{{1}}+x_{{2}}\geq 1,\,x_{{1}}+x_{{2}}\leq 4,\,x_{{1}}-x_{{2}}\geq -2$.
The corresponding toric variety $M_{\rm hex}$ of dimension $2$  is defined by the fan with  edges
$(1,0),(1,1),(0,1),(-1,0),(-1,-1),(1,-1)$.

We can also describe $M_{\rm hex}$ as a reduced Hamiltonian manifold, with the help of an ample line bundle.
We consider the standard torus of dimension $4$ acting in $\C^6$ with the following list $\Phi$ of weights
$$
((1, 0, 0, 0), (0, 1, 0, 0), (0, 0, 1, 0), (0, 0, 0, 1), (-1, -1, 1, 1), (1, -1, 0, 1)).
$$
If $\tau$  is the tope which contains the vector $[2,-1,2,4]$,
then
the reduced manifold $M_\tau$ is  the manifold $M_{\rm hex}$.

We compute $X(\Phi,\tau)$ (by brute force) and obtain:
\begin{multline*}
X(\Phi,\tau)=\\
p_{{1}}p_{{2}}p_{{3}}p_{{4}}p_{{5}}p_{{6}}-p_{{1}}p_{{2}}p_{{3}}p_{{4}
}q_{{5}}q_{{6}}-p_{{1}}p_{{2}}p_{{3}}p_{{5}}q_{{4}}q_{{6}}-p_{{1}}p_{{
2}}p_{{3}}p_{{6}}q_{{4}}q_{{5}}\\
-2\,p_{{1}}p_{{2}}p_{{3}}q_{{4}}q_{{5}}
q_{{6}}\\
-p_{{1}}p_{{2}}p_{{4}}p_{{6}}q_{{3}}q_{{5}}-p_{{1}}p_{{2}}p_{{4
}}q_{{3}}q_{{5}}q_{{6}}-p_{{1}}p_{{2}}p_{{6}}q_{{3}}q_{{4}}q_{{5}}-p_{
{1}}p_{{2}}q_{{3}}q_{{4}}q_{{5}}q_{{6}}-p_{{1}}p_{{3}}p_{{5}}p_{{6}}q_
{{2}}q_{{4}}\\
-p_{{1}}p_{{3}}p_{{5}}q_{{2}}q_{{4}}q_{{6}}-p_{{1}}p_{{3}}
p_{{6}}q_{{2}}q_{{4}}q_{{5}}-p_{{1}}p_{{3}}q_{{2}}q_{{4}}q_{{5}}q_{{6}
}-p_{{1}}p_{{4}}p_{{5}}p_{{6}}q_{{2}}q_{{3}}-p_{{1}}p_{{4}}p_{{6}}q_{{
2}}q_{{3}}q_{{5}}\\
-p_{{1}}p_{{5}}p_{{6}}q_{{2}}q_{{3}}q_{{4}}-p_{{1}}p_
{{6}}q_{{2}}q_{{3}}q_{{4}}q_{{5}}-p_{{2}}p_{{3}}p_{{4}}p_{{5}}q_{{1}}q
_{{6}}-p_{{2}}p_{{3}}p_{{4}}q_{{1}}q_{{5}}q_{{6}}
-p_{{2}}p_{{3}}p_{{5}
}q_{{1}}q_{{4}}q_{{6}}\\
-p_{{2}}p_{{3}}q_{{1}}q_{{4}}q_{{5}}q_{{6}}
-p_{{2}}p_{{4}}p_{{5}}p_{{6}}q_{{1}}q_{{3}}
-p_{{2}}p_{{4}}p_{{5}}q_{{1}}q_{{3}}q_{{6}}
-p_{{2}}p_{{4}}p_{{6}}q_{{1}}q_{{3}}q_{{5}}
-p_{{2}}p_{{4}}q_{{1}}q_{{3}}q_{{5}}q_{{6}}\\
-p_{{3}}p_{{4}}p_{{5}}p_{{6}}q_{{1}}q_{{2}}
-p_{{3}}p_{{4}}p_{{5}}q_{{1}}q_{{2}}q_{{6}}
-p_{{3}}p_{{5}}p_{{6}}q_{{1}}q_{{2}}q_{{4}}
-p_{{3}}p_{{5}}q_{{1}}q_{{2}}q_{{4}}q_{{6}}\\
-2\,p_{{4}}
p_{{5}}p_{{6}}q_{{1}}q_{{2}}q_{{3}}\\
-p_{{4}}p_{{5}}q_{{1}}q_{{2}}q_{{3}
}q_{{6}}-p_{{4}}p_{{6}}q_{{1}}q_{{2}}q_{{3}}q_{{5}}
-p_{{5}}p_{{6}}q_{{1}}q_{{2}}q_{{3}}q_{{4}}+q_{{1}}q_{{2}}q_{{3}}q_{{4}}q_{{5}}q_{{6}}.
\end{multline*}
We can immediately read on this expression the multiplicity
of a weight $m=(m_1,m_2,m_3,m_4,m_5,m_6)$ in the space ${\mathcal H}(\tau,\lambda)$
for any $m$  and any $\lambda$.
We see that the multiplicities of $m$ can be $0$,$1$,$-1$,$-2$
depending on the quadrant in which $m$ lies.

For example, for $\lambda=(200,434,378,-400)$,
the weight $$m=(200, 234, 478, -200, -100, -100)$$ has multiplicity $-2$
in the space ${\mathcal H}(\tau,\lambda)$.
Indeed the coefficient of $p_1p_2p_3q_4q_5q_6$ in $X(\Phi,\tau)$ is $-2$.

Given $\lambda\in \Z^4$, we parameterize the integral points in $V(\Phi,\lambda)$ by $(x_1,x_2)\in\Z^2$,
with corresponding $m\in \Z^6$ given by
$$
m=(\lambda_{{1}}+x_1-x_2,\lambda_{{2}}+x_1+x_2,\lambda_{{3}}-x_1,\lambda_{{4}}-x_1-x_2
,x_1,x_2).$$
With this parametrization,
the figures    \ref{fig:hexagon_1}, \ref{fig:hexagon_234} and \ref{fig:hexagon_567}
describe the support of the module $\mathcal H(\tau,\lambda)$
as $\lambda$ moves along the line joining $\lambda_0=(200, -100,200,400$ (in the ample cone) to $\lambda_1=(200,434,378,-400)$.
The line crosses six walls.

 We assign colors to the multiplicities: $blue=1$, $yellow:=-1$, $red:=-1$, $green:=-1$, $magenta:=-2$, $black:=-1$, $khaki=-1$.
In the first figure   \ref{fig:hexagon_1}, $\lambda$  is in the starting tope (the ample cone).
In the last three steps, a polygon with  multiplicity $-2$ (colored in magenta) has appeared in the middle of the picture.
\begin{figure}[!h]
\begin{center}
\includegraphics[width=1.in]{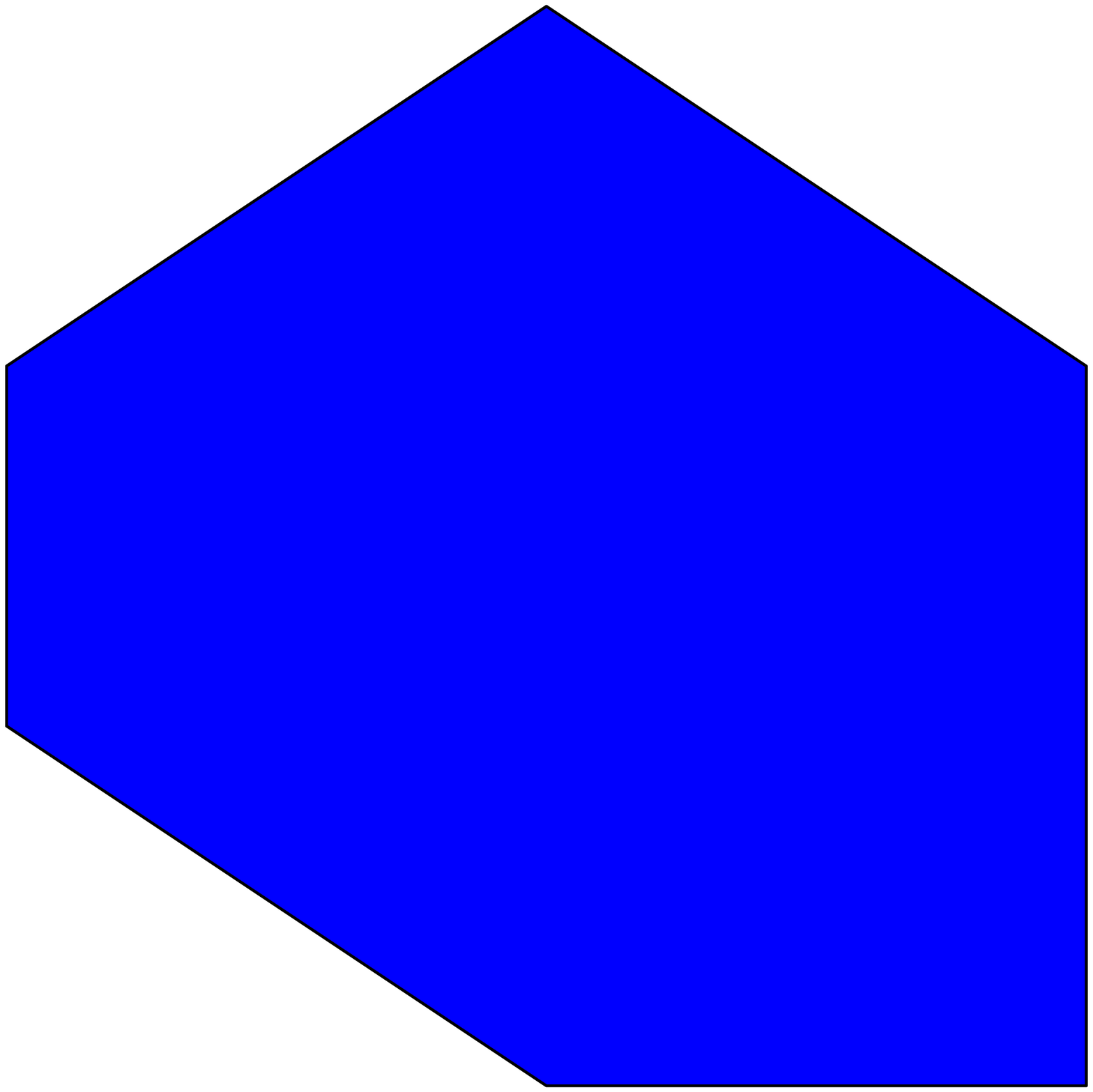}
\caption{ At the beginning $\lambda=(200, -100, 200, 400)$ is in the ample cone. The partition polytope
is an hexagon.}
 \label{fig:hexagon_1}
  \end{center}
\end{figure}
\begin{figure}[!h]
\begin{center}
 \includegraphics[width=1.in]{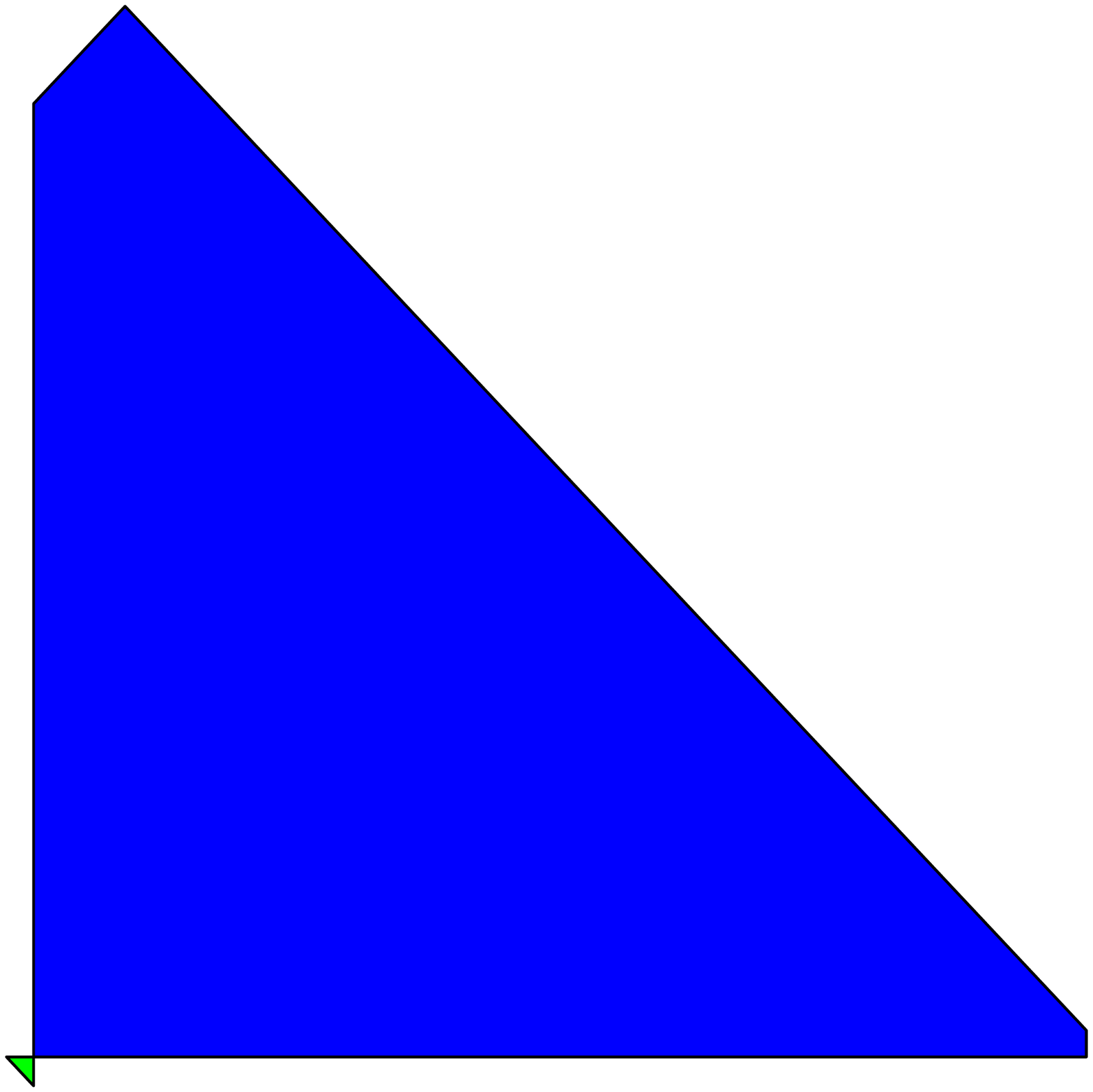}
 \includegraphics[width=1.in]{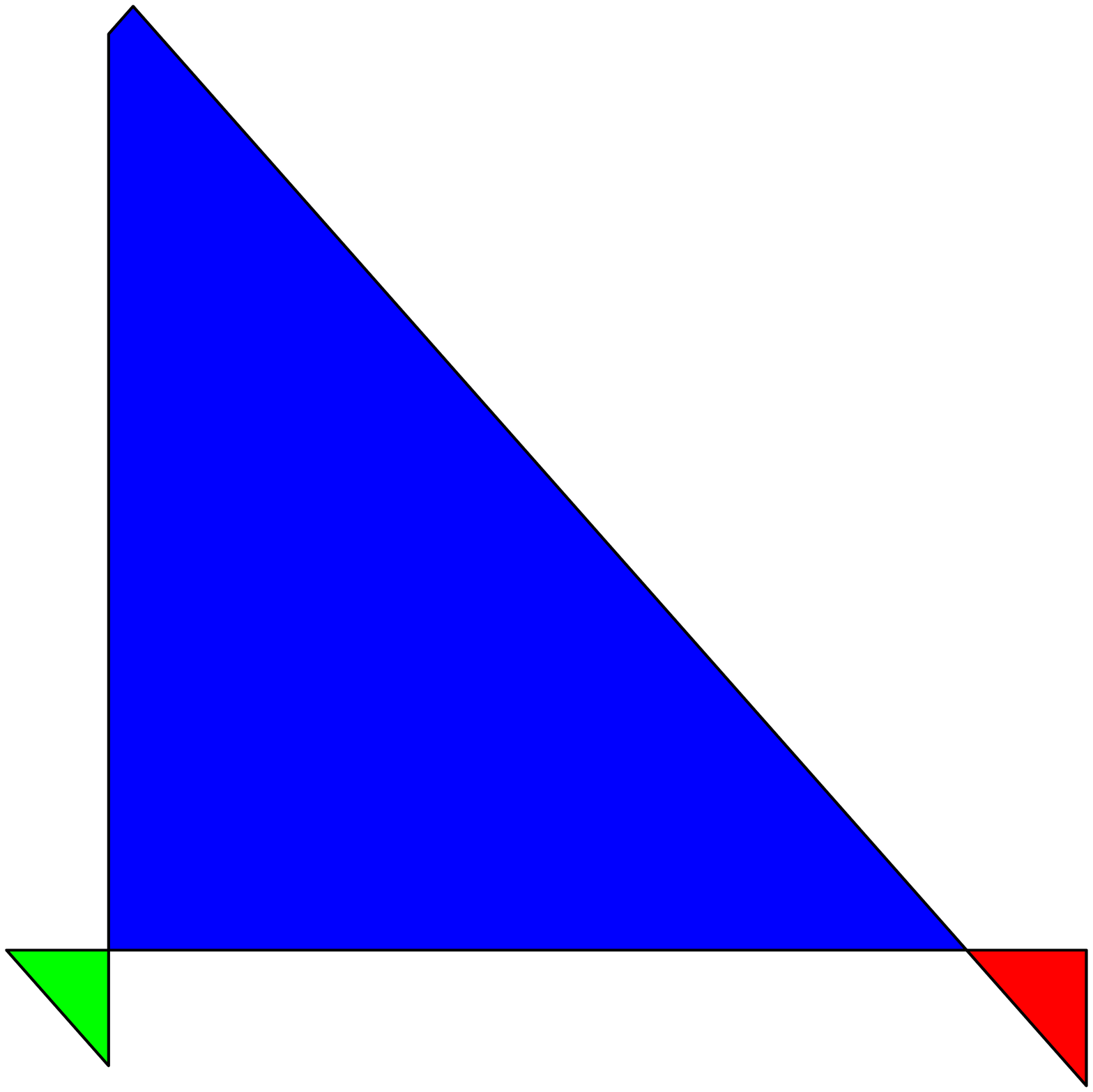}\includegraphics[width=1.in]{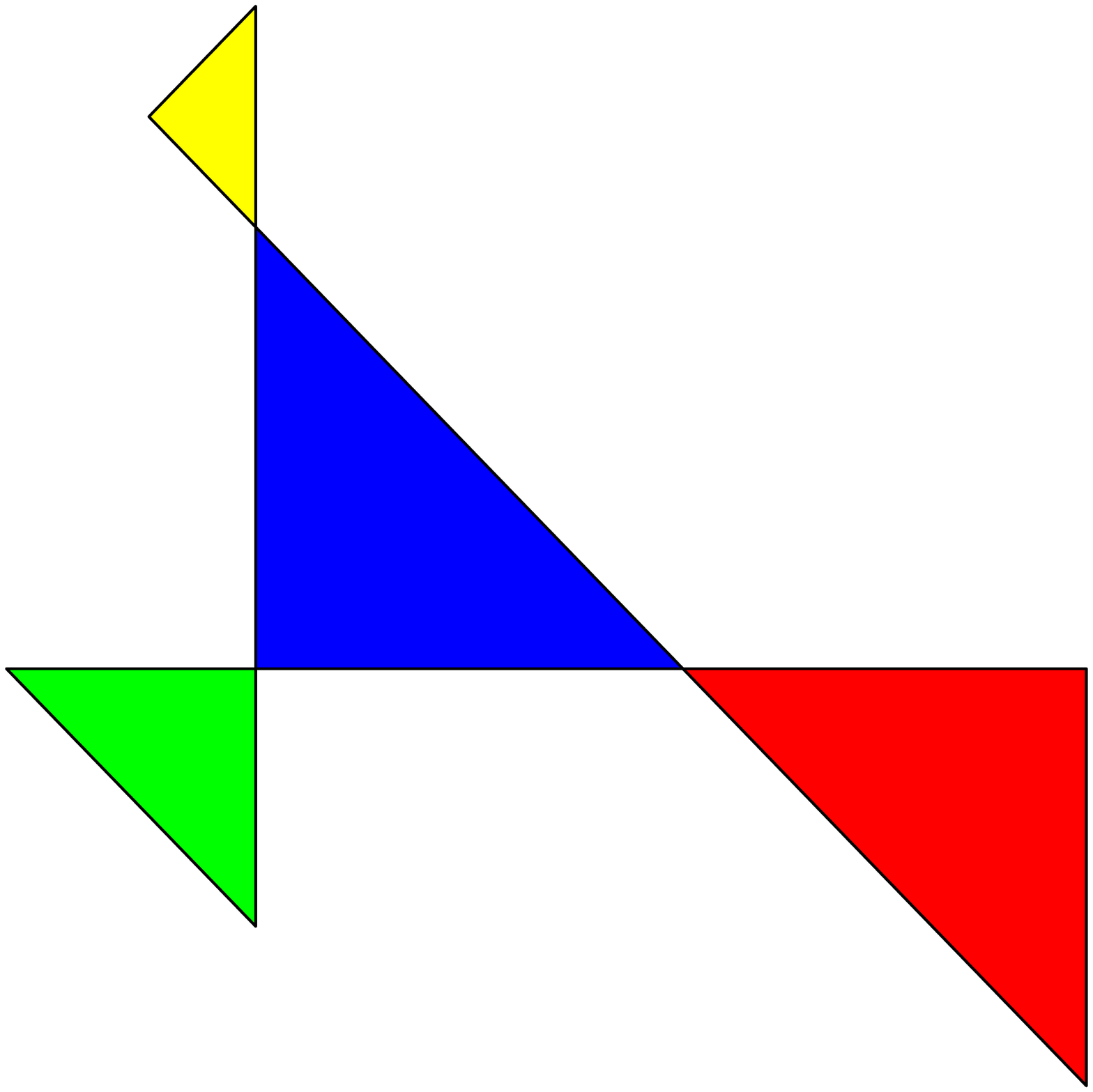}
\caption{From left to right,  $\lambda$  crosses three walls, one at a time. The new triangles have multiplicity -1.}
  \label{fig:hexagon_234}
  \end{center}
\end{figure}
\begin{figure}[!h]
\begin{center}
\includegraphics[width=1 in]{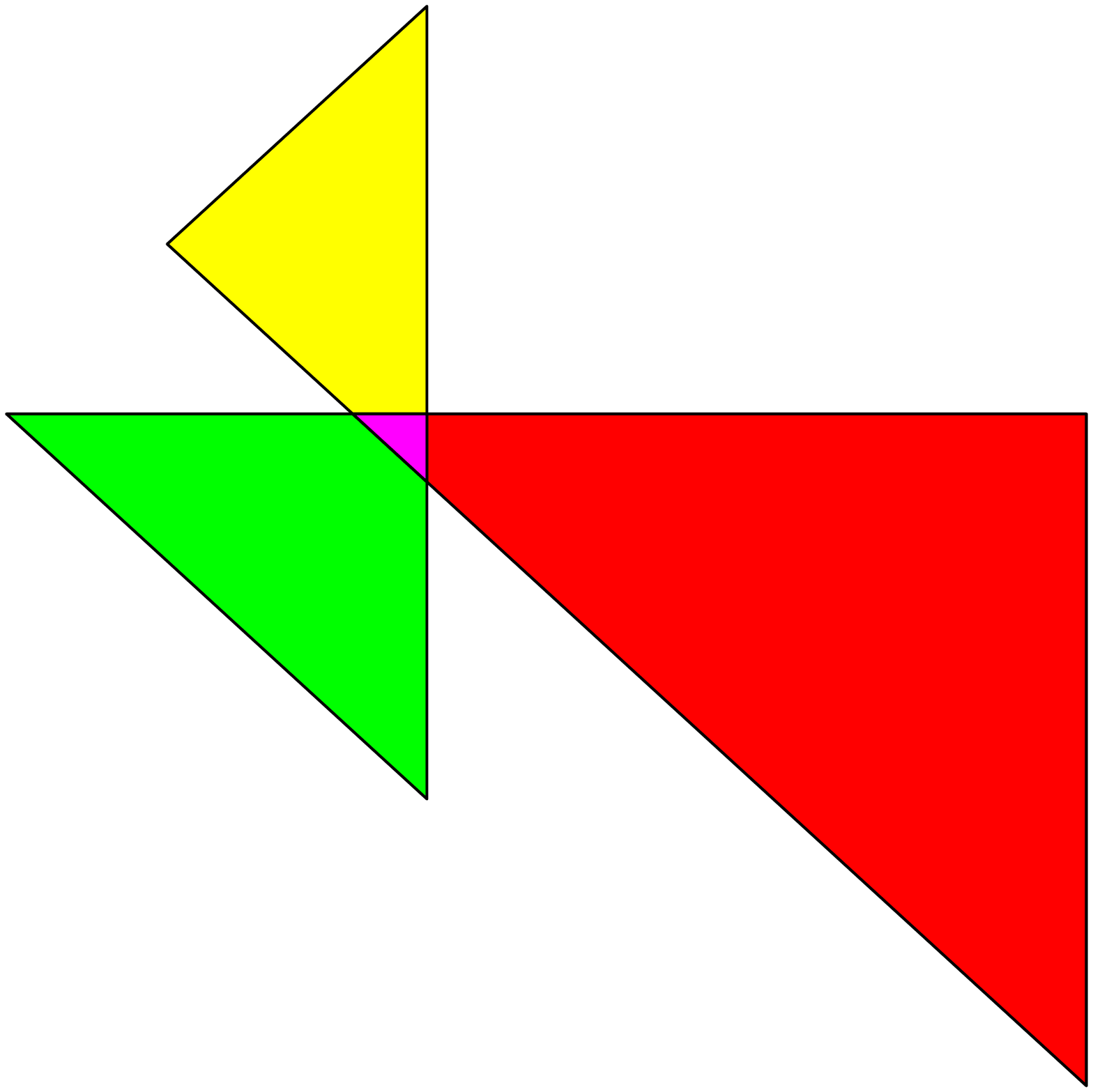}
 \includegraphics[width=1.in]{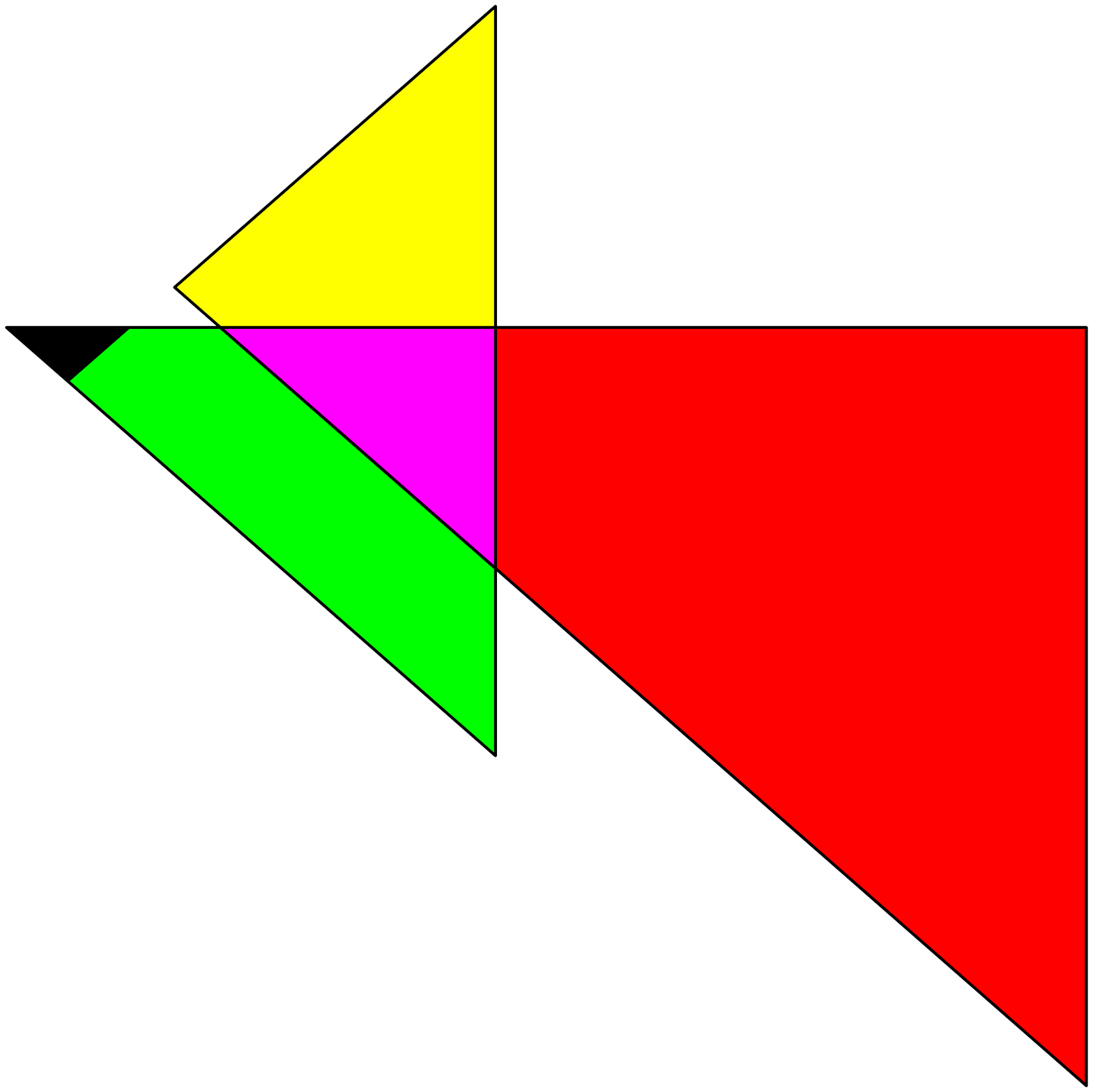}\includegraphics[width=1 in]{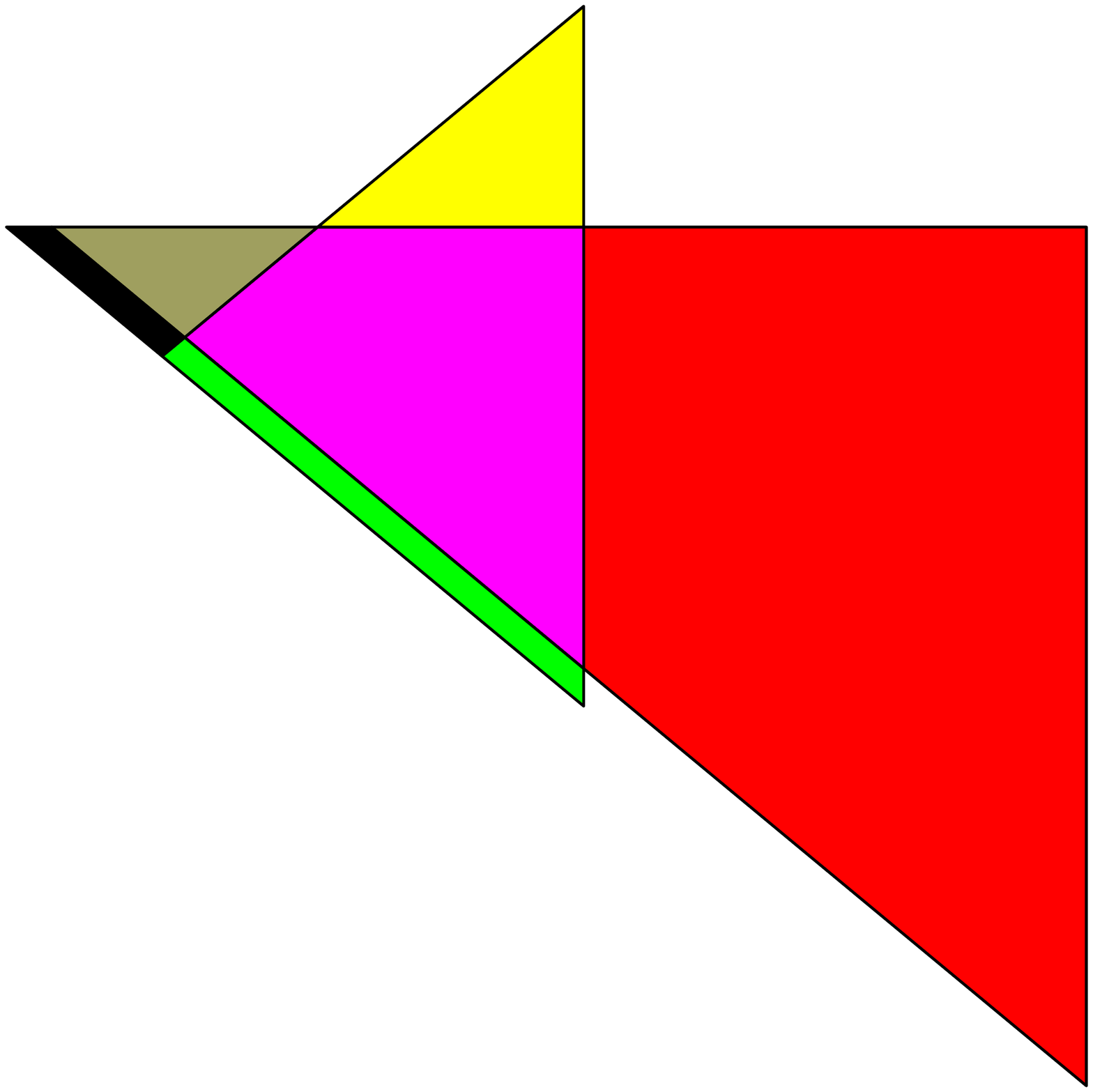}
\caption{$\lambda$ crosses three more walls. The polytope colored in magenta has multiplicity -2. }
 \label{fig:hexagon_567}
  \end{center}
\end{figure}

\newpage
\bibliographystyle{amsabbrv}
\bibliography{biblioNicolepassemuraille-sacre}

\providecommand{\bysame}{\leavevmode\hbox to3em{\hrulefill}\thinspace}
\providecommand{\MR}{\relax\ifhmode\unskip\space\fi MR }
\providecommand{\MRhref}[2]{%
  \href{http://www.ams.org/mathscinet-getitem?mr=#1}{#2}
}
\providecommand{\href}[2]{#2}
\begin{thebibliography}{10}

\bibitem{MR1174359}
D.~Avis and K.~Fukuda, \emph{A pivoting algorithm for convex hulls and vertex
  enumeration of arrangements and polyhedra}, Discrete Comput. Geom. \textbf{8}
  (1992), no.~3, 295--313, ACM Symposium on Computational Geometry (North
  Conway, NH, 1991).

\bibitem{BBDKV-2010}
V.~Baldoni, N.~Berline, J.~D. Loera, M.~K{\"o}ppe, and M.~Vergne,
  \emph{Computation of the highest coefficients of weighted {E}hrhart
  quasi-polynomials for a rational polytope}, arXiv:1011.1602 [math.CO], 2010.

\bibitem{BarviPom}
A.~I. Barvinok and J.~E. Pommersheim, \emph{An algorithmic theory of lattice
  points in polyhedra}, New Perspectives in Algebraic Combinatorics (L.~J.
  Billera, A.~Bj\"orner, C.~Greene, R.~E. Simion, and R.~P. Stanley, eds.),
  Math. Sci. Res. Inst. Publ., vol.~38, Cambridge Univ. Press, Cambridge, 1999,
  pp.~91--147.

\bibitem{MR685019}
N.~Berline and M.~Vergne, \emph{Classes caract\'eristiques \'equivariantes.
  {F}ormule de localisation en cohomologie \'equivariante}, C. R. Acad. Sci.
  Paris S\'er. I Math. \textbf{295} (1982), no.~9, 539--541.

\bibitem{MR2573189}
A.~Boysal and M.~Vergne, \emph{Paradan's wall crossing formula for partition
  functions and {K}hovanski-{P}ukhlikov differential operator}, Ann. Inst.
  Fourier (Grenoble) \textbf{59} (2009), no.~5, 1715--1752.

\bibitem{Brion88}
M.~Brion, \emph{Points entiers dans les poly{\`e}dres convexes}, Ann. Sci.
  {\'E}cole Norm. Sup. \textbf{21} (1988), no.~4, 653--663.

\bibitem{brion-vergne-toric-1997}
M.~Brion and M.~Vergne, \emph{An equivariant {R}iemann-{R}och theorem for
  complete, simplicial toric varieties}, J. Reine Angew. Math. \textbf{482}
  (1997), 67--92.

\bibitem{brion-vergne-97-lattice}
M.~Brion and M.~Vergne, \emph{Lattice points in simple polytopes}, J.A.M.S.
  \textbf{10} (1997).

\bibitem{brion-vergne-97-residue}
M.~Brion and M.~Vergne, \emph{Residue formulae, vector partition functions and
  lattice points in rational polytopes}, J. Amer. Math. Soc. \textbf{10}
  (1997), 797--833.

\bibitem{cavalieri-2010}
R.~Cavalieri, P.~Johnson, and H.~Markwig, \emph{Chamber structure of double
  {H}urwitz numbers}, 2010, arXiv:1003.1805v1.

\bibitem{Cox-Little-Schenck}
D.~Cox, J.~Little, and H.~Schenck, \emph{Toric varieties}, 2010, Available at
  http://www.cs.amherst.edu/~dac/toric.html.

\bibitem{dahmen-micchelli-1988}
W.~Dahmen and C.~A. Micchelli, \emph{The number of solutions to linear
  {D}iophantine equations and multivariate splines}, Trans. Amer. Math. Soc.
  \textbf{308} (1988), no.~2, 509--532.

\bibitem{deconcini-procesi-vergne-dahmen-micchelli-2008}
C.~{De Concini}, C.~Procesi, and M.~Vergne, \emph{Vector partition function and
  generalized {D}ahmen-{M}icchelli spaces}, 2008, arXiv:0805.2907v2.

\bibitem{karshon-tolman-presymplectic-1993}
Y.~Karshon and S.~Tolman, \emph{The moment map and line bundles over
  presymplectic toric manifolds}, J. Differential Geom. \textbf{38} (1993),
  no.~3, 465--484.

\bibitem{lawrence91}
J.~Lawrence, \emph{Polytope volume computation}, Math. Comp. \textbf{57}
  (1991), no.~195, 259--271.

\bibitem{MR1234308}
R.~Morelli, \emph{The {$K$}-theory of a toric variety}, Adv. Math. \textbf{100}
  (1993), no.~2, 154--182.

\bibitem{mustata-vanishing-toric-2002}
M.~Mustata, \emph{Vanishing theorems on toric varieties}, Tohoku Math. J. (2)
  \textbf{54} (2002), no.~3, 451--470.

\bibitem{paradan2004}
P.-E. Paradan, \emph{Jump formulas in {H}amiltonian geometry}, 2004, arXiv math
  0411306, to appear in Geometric Aspects of Analysis and Mechanics,
  Proceedings of the conference in honor of H. Duistermaat, Utrecht 2007.

\bibitem{szenes-vergne-2002}
A.~Szenes and M.~Vergne, \emph{Residue formulae for vector partitions and
  {E}uler-{M}aclaurin sums}, Formal power series and algebraic combinatorics -
  Scottsdale, AZ, 2001, Adv. in Appl. Math., vol.~30, 2003, arXiv math 0202253,
  pp.~295–--342.

\bibitem{varchenko}
A.~N. Varchenko, \emph{Combinatorics and topology of the arrangement of affine
  hyperplanes in the real space}, Functional Anal. Appl. \textbf{21, no. 1}
  (1987), 9–--19, Russian original publ. Funktsional. Anal. i Prilozhen. 21
  (1987), no. 1, p.11–-22.

\end{thebibliography}
\end{document}